\documentclass[10pt]{article}
\usepackage{a4}

\usepackage{amsmath, amsfonts,amssymb,amscd,amsthm}
\usepackage{graphicx}

\usepackage{bm}
\usepackage{subfig}
\usepackage{graphicx}
\usepackage{pict2e}
\usepackage{color}
\usepackage[percent]{overpic}
\usepackage[colorlinks=true,linkcolor=black,urlcolor=black]{hyperref}
\usepackage[english]{babel}
\usepackage{algorithm, algorithmic}

\usepackage[title]{appendix}

\newtheorem{theorem}{Theorem}
\newtheorem{theorem*}{Theorem}

\newtheorem{proposition}[theorem]{Proposition}
\newtheorem{definition}[theorem]{Definition}

\DeclareMathOperator{\R}{\mathbb{R}}

\begin{document}
\title{
Space-time boundary elements for frictional contact in elastodynamics}
\author{Alessandra Aimi\thanks{Department of Mathematical, Physical and Computer Sciences,
University of Parma, Parco Area delle Scienze, 53/A, 43124, Parma,
Italy, email: alessandra.aimi@unipr.it} \and Giulia Di Credico\thanks{Department of Engineering and Architecture,
University of Parma, Parco Area delle Scienze, 181/A, 43124, Parma,
Italy, email: giulia.dicredico@unipr.it} ${}^\ddagger$ \and Heiko Gimperlein\thanks{Engineering Mathematics, University of Innsbruck, Innsbruck, Austria,  email: heiko.gimperlein@uibk.ac.at\\ \noindent 
 A.~Aimi and G.~Di~Credico~are members of the INDAM-GNCS Research Group, Italy. \\\noindent \emph{Acknowledgements:} We thank Peter Gamnitzer for helpful discussions.}}
\date{}

\providecommand{\keywords}[1]{{\textit{Key words:}} #1}

\maketitle \vskip 0.5cm
\begin{abstract}
\noindent This article studies a boundary element method for dynamic frictional contact between linearly elastic bodies. We formulate these problems {as a variational inequality on the boundary, involving the elastodynamic} Poincar\'{e}-Steklov operator. The variational inequality is solved in a mixed formulation using boundary elements in {space and time}. In the model problem of unilateral Tresca friction contact with a rigid obstacle we obtain an a priori estimate for the resulting Galerkin approximations. Numerical experiments in two space dimensions demonstrate the stability, energy conservation and convergence of the proposed method for contact problems involving concrete and steel in the linearly elastic regime. They address both unilateral and two-sided dynamic contact with Tresca or Coulomb friction. 
\end{abstract}

\section{Introduction}
\label{intro}

The deformation of an elastic body naturally involves contact with surrounding materials.
The resulting nonpenetration constraints and friction lead to significant challenges for numerical simulations, especially in dynamic problems, because the time discretization interacts with the contact conditions. In the absence of dissipation, for common finite element procedures of dynamic contact at most stability or the approximate conservation of energy have been established. See \cite{bur2,chouly0,choulybook,choulydyn2,doyen2,twosided,gwinner,sten2,hauret,cont2,cont1,chouly3} for examples from the extensive mathematical finite element literature on static and dynamic contact problems.

As the contact takes place on the surface,  boundary element and coupled finite element / boundary element approaches provide an efficient alternative in static or quasi-static problems \cite{gwinsteph}. Their numerical analysis has been extensively studied and is well-understood in the context of elliptic variational inequalities. Rigorous boundary element methods for dynamic problems, however, were only recently introduced and investigated for simplified dynamic model problems, such as scalar and simple elastodynamic unilateral Signorini problems \cite{ourpaper2,contact}. 

In this article we propose and study a time-domain boundary element method for the full, dynamic frictional contact problem between two linearly elastic bodies. We investigate its properties and its performance from model situations to simulations with realistic material parameters, in particular the stability, convergence and energy conservation.\\
In order to describe our results more precisely, we consider the dynamics of a homogeneous, linearly elastic body in a bounded domain $\Omega \subset \mathbb{R}^d$, $d=2,3$, whose boundary $\partial \Omega$ is denoted by $\Gamma$. The dynamics of the body is described by the Navier-Lam\'{e} equations
\begin{equation}\label{navierlame}
\nabla \cdot \sigma(\textbf{u})-\varrho\ddot{\textbf{u}}=\textbf{0} 
\end{equation}
for the unknown displacement vector $\textbf{u}$ in $\Omega$ at times $t \in (0,T]$. Here, a dot denotes the derivative with respect to $t$, $\varrho$ the homogeneous mass density and  $\sigma(\textbf{u})$ the Cauchy stress tensor
$$\sigma(\textbf{u})=\varrho (c_{\mathtt{P}}^2-2c_{\mathtt{S}}^2)\,(\nabla \cdot \textbf{u})I+\varrho c_{\mathtt{S}}^2\,(\nabla \textbf{u}+\nabla \textbf{u}^\top)\,,$$
where the pressure and shear velocities are denoted by $c_{\mathtt{P}}$, $c_{\mathtt{S}}$.

The contact boundary conditions on a given contact boundary $\Gamma_C \subset \Gamma$ combine the non-penetration of an obstacle with friction tangent to the boundary. To describe them, we denote by {$\vert_{\Gamma}$ the restriction to $\Gamma$}, by $\textbf{n}$ the outward-pointing unit normal vector to $\Gamma$, and by $v_\perp$ and $\mathbf{v}_\parallel$ the normal {component}, respectively {the} tangential {component}, of a vector $\textbf{v}$ at $\Gamma$. In particular $v_\perp:=-\textbf{v}\cdot \textbf{n}$. In terms of the elastic traction \begin{equation}
\label{trazione}
\textbf{p}=\sigma\left(\textbf{u}\right)_{\vert_{\Gamma}} \bf n\, ,\end{equation}
for small deformations the non-penetration condition becomes
\begin{equation}\label{contactbc}\begin{cases}
{ u}_\perp \geq g\ ,\, {p}_\perp\geq {f}_\perp\ ,\\ {u}_\perp > g\ \Longrightarrow {p}_\perp = {f}_\perp\ .
\end{cases}\end{equation}
The mechanical interpretations of the gap function $g$ and the prescribed {surface} force ${\bf f}$ are discussed in Section \ref{mechsetup} below. The non-penetration is accompanied by friction parallel to $\Gamma_C$, and it is given by
\begin{align}\label{frictionbc}
\begin{cases} \| \mathbf{p}_\parallel \|  \leq \mathcal{F}, \\
 \| \mathbf{p}_\parallel \|  < \mathcal{F} \Longrightarrow \dot{\mathbf{u}}_\parallel=0, \\
 \| \mathbf{p}_\parallel \|  = \mathcal{F} \Longrightarrow \exists\, \alpha \geq 0 : \dot{\mathbf{u}}_\parallel=-\alpha \mathbf{p}_\parallel.
\end{cases}\end{align}
The friction law is determined by the friction threshold $\mathcal{F}\geq 0$. Section \ref{mechsetup} discusses the mechanical interpretation of the Tresca and Coulomb friction laws.  

The contact boundary conditions \eqref{contactbc}, \eqref{frictionbc} on $\Gamma_C$ are complemented by boundary conditions prescribing the traction on $\Gamma_N \subset \Gamma$ and the displacement on the remainder $\Gamma_D = \Gamma \setminus \overline{\Gamma_\Sigma}$, where $\Gamma_\Sigma:=\Gamma_N \cup \Gamma_C$ is the union of the contact and traction boundaries.\\
Building on \cite{ourpaper2}, we use the elastodynamic Poincar\'{e}-Steklov operator to formulate the contact problem \eqref{navierlame}, \eqref{contactbc}, \eqref{frictionbc} as a variational inequality on the boundary $\Gamma$. 
We discretize an equivalent mixed formulation for the displacement and the contact forces using a time-domain Galerkin boundary element method and solve the resulting nonlinear problem by an Uzawa
algorithm.\\
This article presents the details of the algebraic formulation and implementation of the energetic space-time boundary element method for both the above-described unilateral friction problem and the frictional contact between two bodies. Numerical results confirm the stability and convergence of the proposed method in two dimensions. We present results in unilateral model problems with Tresca and Coulomb friction, as well as for contact involving concrete and steel. Both polygonal and curved contact boundaries are considered, as well as contact between two elastic bodies.   Theoretically, we obtain an a priori estimate in the model problem of unilateral Tresca friction contact, generalizing the results in \cite{ourpaper2} for the unilateral Signorini problem.\\
The methods developed in this article build on recent advances for boundary elements for {hyperbolic evolution equations}, see \cite{costabel04,gimperleinreview,hd,sayas} for an overview.  
For the numerical analysis and computational aspects in the case of elastodynamic problems we refer to \cite{AimiJCAM, ourpaper, Becache1993, Becache1994, chaillat,  falletta, schanz}.\\
This article is structured as follows: in Section \ref{mechsetup} the dynamic frictional contact problem will be introduced, for both unilateral and two-body contact. In Section \ref{Unilateral}, we introduce and analyze variational and mixed formulations for the unilateral frictional contact problem, together with its discretization, giving also a priori error estimates. In this Section both {Tresca} and Coulomb frictions are taken into account. The generalization to the two-body contact problem will be considered in Section \ref{Bilateral}. Algorithmic details will be provided in Section \ref{sec:psh}. Several numerical results will be presented in Section \ref{sec;numres}. Some conclusions are given in Section \ref{sec:conclusions}.\\

\noindent \emph{Notation:}  
We write $f \lesssim g$ if $f \leq Cg$ for some constant $C>0$, and $f \lesssim_\sigma g$ if 
the constant $C>0$ may depend on a parameter $\sigma$.

\section{Dynamic frictional contact:  
problem formulations} \label{mechsetup}

\begin{figure}[h!]
\centering
\subfloat[]{\includegraphics[width=0.31\textwidth]{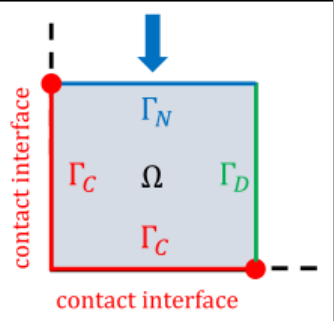}}
\quad
\subfloat[]{\includegraphics[width=0.31\textwidth]{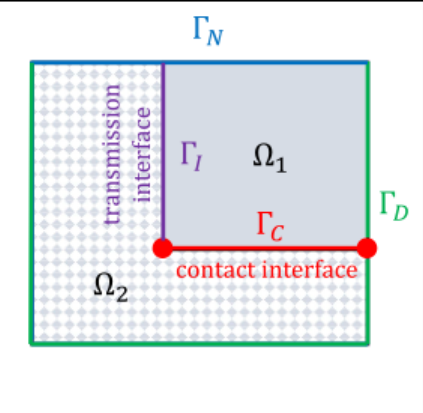}}
\caption{Schematic depiction of unilateral (a) and bilateral (b) contact problems.}
\label{Figure:setup}
\end{figure}

This subsection introduces the equations which govern the contact problems studied in this article. As illustrated in 
Figure \ref{Figure:setup}(a) for $d=2$, we consider the {dynamic} deformation of a linearly elastic body, described by the 
Navier-Lam\'{e} equations \eqref{navierlame} for the displacement $\textbf{u}$ in a bounded domain $\Omega \subset \mathbb{R}^d$, $d=2,3$, with boundary $\Gamma:=\partial \Omega$. 
For the remainder of this article we assume that $\Omega$ is a bounded polygonal or polyhedral Lipschitz domain.
Starting from the reference configuration, $\textbf{u} = \textbf{0}$ for times $t\leq 0$, the dynamics of the body {are} due to surface 
forces $\mathbf{f}$ prescribed on the subset $\Gamma_\Sigma \subset \Gamma$ already introduced in the previous section.   

The unilateral contact problem considered in this article describes the  impossibility of the body to penetrate an adjacent
rigid surface. In Figure \ref{Figure:setup}(a) the impenetrable obstacle is given by the positive $x$- and $y$-axes.
Nonpenetration of the obstacle leads to the contact condition \eqref{contactbc} on the part $\Gamma_C \subset \Gamma$ of the boundary where contact may occur. 
More precisely, from \eqref{contactbc} contact takes place when the normal displacement satisfies $u_\perp = g$, where $g$ describes the gap between the reference configuration and the obstacle.  
Contact is avoided when $u_\perp > g$, and then only the applied surface forces ${p}_\perp = {f}_\perp$ act on $\Gamma_C$.
In addition to nonpenetration, contact leads to frictional forces tangential to the rigid surface on $\Gamma_C$. The friction law is specified by the friction threshold $\mathcal{F} \geq 0$. 
The tangential displacement $\mathbf{u}_\parallel$ of the body remains fixed, $\dot{\mathbf{u}}_\parallel=0$, until the magnitude of the parallel traction $\mathbf{p}_\parallel$ reaches  $\mathcal{F}$. Once $\| \mathbf{p}_\parallel \|  = \mathcal{F}$  the body deforms with a velocity  {opposite} to the direction of the traction $\mathbf{p}_\parallel$.

Typical friction laws include Tresca and Coulomb friction, both of which are considered in this article. Tresca friction corresponds to a prescribed threshold $\mathcal{F} \in L^\infty(\Gamma_C)$  which is independent of the traction ${\bf p}$
 and the nonpenetration condition. Friction may then take place even when $u_\perp > g$, leading to unphysical non-zero tangential tractions outside the physical contact area. This problem is avoided by the Coulomb friction law, which replaces the prescribed friction threshold $\mathcal{F}$ by the function $\mathcal{F} = \mathcal{F}_c\,|p_\perp|$. Coulomb friction is often used in applications as a model for dry friction; we also refer to \cite{wrig} for further background on Coulomb friction and other friction laws.

At last, the body is fixed on $\Gamma_D\subset \Gamma$. Hence,
the full system of governing equations describing the unilateral frictional contact problem is then given by:
\begin{subequations}  \label{prob:strong_formulation}
\begin{alignat}{2}
\nabla \cdot \sigma(\textbf{u})-\varrho\ddot{\textbf{u}}&=\textbf{0} & \quad & \text{in } (0,T]\times  \Omega \label{strongformPDE}\\
	\textbf{u}&=\textbf{0} & \quad & \text{on } (0,T]\times  \Gamma_D \label{prob:strongformulation_d}\\
  \textbf{p} &=\mathbf{f} & \quad & \text{on } (0,T]\times  \Gamma_N \label{prob:strong_formulation_1}
 \end{alignat}
\end{subequations}
together with the contact conditions \eqref{contactbc}, \eqref{frictionbc} on $(0,T]\times  \Gamma_C$ and homogeneous 
initial condition $\mathbf{u}\equiv \mathbf{0}$ 
in $\Omega$ for $t\leq 0$. 
Here and in the following we assume that $\textbf{f}_\parallel = 0$ on $(0,T]\times\Gamma_C$. 
It is also convenient to 
write the contact boundary conditions \eqref{contactbc}, \eqref{frictionbc} in a compact form, as
\begin{align}
 u_\perp\geq g\ ,\ p_\perp\geq f_\perp,\ \ (p_\perp-f_\perp)(u_\perp-g)&=0,
\label{fullcontactbc1}
\\
  \| \mathbf{p}_\parallel \|  \leq \mathcal{F},\ \mathbf{p}_\parallel \cdot\dot{\mathbf{u}}_\parallel+\mathcal{F}\,\| \dot{\mathbf{u}}_\parallel \| &=0. \label{fullcontactbc2}
\end{align}

\vspace{0.2in}
The unilateral frictional contact problem described above readily generalizes to a frictional contact problem between two linearly elastic bodies, as depicted in 
Figure \ref{Figure:setup}(b). Their time dependent deformation is described by the 
Navier-Lam\'{e} equations for their displacements $\textbf{u}_1, \textbf{u}_2$  in $\Omega_1$, respectively $\Omega_2$: 
\begin{equation}
    \nabla \cdot \sigma_j(\mathbf{u}_j)-\varrho\ddot{\mathbf{u}}_j=\textbf{0} \quad  \text{in } (0,T]\times  \Omega_j\,,j=1,2. \label{strongformPDE2b}
\end{equation}
Starting from the reference configuration, $\textbf{u}_j = \textbf{0}$, $j=1,2$, for times $t\leq 0$, the dynamics is now governed by transmission and contact conditions at the interface, on $\Gamma_I$ and $\Gamma_C$. We also impose surface forces on $\Gamma_N \subset \partial \Omega_1 \cup \partial \Omega_2$ and fix the bodies on $\Gamma_D \subset \partial \Omega_1 \cup \partial \Omega_2$, as described above for unilateral contact.\\
In this article we restrict to a formulation valid for small relative deformations of $\Omega_1$ and $\Omega_2$, see {\cite{chouly3}} for general, large deformations in a time-independent contact problem. The contact between the two bodies takes place at the interface $\partial\Omega_1 \cap \partial\Omega_2$, where either transmission conditions (on $\Gamma_I$) or contact conditions (on $\Gamma_C$) are imposed. The transmission conditions describe a rigid connection between $\Omega_1$ and $\Omega_2$:
\begin{subequations}  \label{prob:strong_formulation2b}
\begin{alignat}{2}
	\mathbf{u}_1-\mathbf{u}_2&=\mathbf{g} & \quad & \text{on } (0,T]\times  \Gamma_I  \label{prob:strong_formulation2b_2}\\
  \mathbf{p}_1+\mathbf{p}_2 &=\mathbf{f} & \quad & \text{on } (0,T]\times  \Gamma_I. \label{prob:strong_formulation2b_1}
\end{alignat}
\end{subequations}
{Here $\textbf{p}_i=\sigma\left(\textbf{u}_i\right)_{\vert_{\Gamma}} \textbf{n}_i\,$, where $\textbf{n}_i$ is the outward-pointing unit normal vector to $\partial\Omega_i$}. Data $\mathbf{g} = \mathbf{0}$,  $\mathbf{f}= \mathbf{0}$, correspond to the continuity of displacements and equilibrium of  tractions, while general $\mathbf{g}$,  $\mathbf{f}$ allow to model a fixed, rigid relative displacement or a surface {force} at the interface.
The contact conditions on $(0,T]\times  \Gamma_C$ allow the opening of a small gap and friction at the contact interface $\Gamma_C$; they depend on the \emph{relative} displacements and \emph{relative} tractions of the two bodies:  
\begin{align}
 u_{1,\perp}-u_{2,\perp}\geq g_{\perp}
 \ ,\ p_{1,\perp}\geq 0,\ p_{2,\perp}\leq  f_\perp,\ \begin{cases} p_{1,\perp} = f_\perp -p_{2,\perp}= 0 & \text{ if }u_{1,\perp}-u_{2,\perp}> g_{\perp}\\ p_{1,\perp} + p_{2,\perp} = f_\perp & \text{ if } u_{1,\perp}-u_{2,\perp}= g_{\perp} \end{cases} \text,
\label{fullcontactbc12b}\\
  \mathbf{p}_{1,\parallel} + \mathbf{p}_{2,\parallel}-\textbf{f}_\parallel = 0, \ \| \mathbf{p}_{1,\parallel} \|  \leq \mathcal{F},\ \mathbf{p}_{1,\parallel} \cdot( \dot{\mathbf{u}}_{1,\parallel}-\dot{\mathbf{u}}_{2,\parallel}-\dot{\mathbf{g}}_{\parallel})+\mathcal{F}\| \dot{\mathbf{u}}_{1,\parallel}-\dot{\mathbf{u}}_{2,\parallel}-\dot{\mathbf{g}}_{\parallel} \| =0. \label{fullcontactbc22b}
\end{align}

{On the interface $\partial \Omega_1\cap \partial \Omega_2$ the normal vector is chosen to point from $\Omega_1$ to $\Omega_2$.}\\ The traction and displacement conditions are as for the unilateral contact problem:
\begin{subequations}  
\label{sfdn2}
\begin{alignat}{2}
	\textbf{u}_j&=\textbf{0} & \quad & \text{on } (0,T]\times  \Gamma_D \label{prob:strongformulation_d_2b}\\
  \textbf{p}_j &=\mathbf{f} & \quad & \text{on } (0,T]\times  \Gamma_N. \label{prob:strong_formulation_1_2b}
\end{alignat}
\end{subequations}
For the contact of two bodies, it will be convenient to define the interface $\Gamma_\Sigma':=\Gamma_C\cup \Gamma_I$ and $\Gamma_\Sigma$ as $\Gamma_\Sigma:=\Gamma_N\cup\Gamma_\Sigma'$.

\section{Unilateral frictional contact}\label{Unilateral}

Boundary element methods reduce the differential equation \eqref{navierlame} to {a singular integral equation on} the boundary of the domain. The resulting lower-dimensional problem can lead to efficient numerical schemes for contact problems, where the key difficulties are on the boundary. 

In this article we reduce the unilateral frictional contact problem \eqref{prob:strong_formulation}, \eqref{fullcontactbc1}, \eqref{fullcontactbc2} to an equivalent variational inequality on the boundary.  
The reduction involves the Poincar\'{e}-Steklov operator $\mathcal{S}$ on $(0,T]\times  \Gamma$, defined by
\begin{equation}
\label{operator_S}
\mathcal{S} \left({\bf u}{|_\Gamma}\right) := \sigma\left(\textbf{u}\right)_{|_\Gamma}  {\bf n}=\mathbf{p}.
\end{equation}
Here, ${\bf u}$ is the solution to the elastodynamic equation \eqref{strongformPDE} in $(0,T]\times  \Omega$, given a prescribed Dirichlet datum ${\bf u}|_\Gamma$. The traction $\mathbf{p}$ in the last equality was defined in \eqref{trazione}. 
To simplify the notation, the subscript $\vert_\Gamma$ in the argument of the operator {$\cal S$} is omitted whenever it is clear from the context. Mapping properties of $\mathcal{S}$ are recalled in \cite{ourpaper2} (see Theorem 5 there).

\subsection{Variational inequality and mixed formulation for Tresca friction} 
In a first step we consider the dynamic Tresca friction problem, where the friction threshold $0\leq\mathcal{F} \in L^\infty(\Gamma_C)$ in \eqref{fullcontactbc2} is independent of $\mathbf{p}$.  
Its variational formulation involves weighted $L^2$ inner products in space and time, defined by
\begin{align}
&\langle \textbf{u},\textbf{v} \rangle_{0,\Gamma,(0,T]}:=\int_0^T \int_{\Gamma}\textbf{u}(t,\textbf{x})\cdot\textbf{v}(t,\textbf{x})\: d\Gamma_{\textbf{x}}\:dt,\label{0_product}\\
&\langle \textbf{u},\textbf{v} \rangle_{\sigma,\Gamma,\mathbb{R}^+}:=\int_0^\infty e^{-2\sigma t}\int_{\Gamma}\textbf{u}(t,\textbf{x})\cdot\textbf{v}(t,\textbf{x})\: d\Gamma_{\textbf{x}}\:dt,\label{sigma_product}  
\end{align}
where $\sigma>0$. {The index {\small{$0$}} on the left-hand side of \eqref{0_product} is motivated by the standard $L^2$ inner product and the choice $\sigma=0$ in \eqref{sigma_product}.} The subscript $\Gamma$ may also be replaced by a subset $\Gamma' \subset \Gamma$ of the boundary, when relevant. 

The variational formulation of the  contact problem involves the friction functional 
\begin{equation}   
\label{friction functional}j(\mathbf{v}):= {\int_0^T}\int_{\Gamma_C} \mathcal{F}\, \|\dot{\bf v}_\parallel\| \,d\Gamma_{\textbf{x}}\,{dt},
\end{equation}
More concisely, we write $j(\mathbf{v}) =\langle {\cal F},\|\dot{\bf v}_\parallel\| \rangle_{0,\Gamma_C,(0,T]}$.
We also require the linear operator $\partial_{t,\parallel}$ which differentiates only the tangential component of a vectorial function $\mathbf{v}$ with respect to time; in particular, $\partial_{t,\parallel} \mathbf{v} = \dot{\mathbf{v}}$, if $v_\perp = 0$, while $\partial_{t,\parallel} \mathbf{v} = \mathbf{v}$, if $\mathbf{v}_\parallel = \mathbf{0}$.\\

The formulation of \eqref{prob:strong_formulation}, \eqref{fullcontactbc1}, \eqref{fullcontactbc2} as a variational inequality in terms of $\mathcal{S}$ then reads as follows, for a given gap function $g$ on $(0,T]\times  \Gamma_C$ and surface forces $\textbf{f}$ on $(0,T]\times  \Gamma_\Sigma$: \\

\noindent \textit{find} $\textbf{u} \in \mathcal{C}:=\left\lbrace \textbf{v} : (0,T]\times  \Gamma \to \mathbb{R}^d: \mathbf{v} = \mathbf{0} \; \text{~a.e.~on~} (0,T]\times  \Gamma_D, \,v_\perp \geq g \text{~a.e.~on~} (0,T]\times  \Gamma_C \right\rbrace$ \textit{such that}
\begin{align} \label{eq:VarIneq}
\langle {\mathcal{S} {\bf u}}, \partial_{t,\parallel}({\bf v}-{\bf u})\rangle_{0,\Gamma_\Sigma,(0,T]} + j({\bf v})-j({\bf u})\geq \langle {\bf f}, \partial_{t,\parallel}({\bf v}-{\bf u})\rangle_{0,\Gamma_\Sigma,(0,T]} \qquad \quad  \forall\, {\bf v}  \in \mathcal{C}.
\end{align}

In order to relate the variational inequality \eqref{eq:VarIneq} to the  contact problem \eqref{prob:strong_formulation}, \eqref{fullcontactbc1}, \eqref{fullcontactbc2}, recall that the displacement $\mathbf{u}$ in the domain $(0,T]\times \Omega$ can be recovered from its boundary trace $\mathbf{u}|_\Gamma$ on $(0,T]\times \Gamma$. This is typical for direct boundary element methods and we review it below in Section \ref{birs} (see, in particular, the representation formula \eqref{representation formula}).

\begin{proposition}\label{prop:varineq}
The variational inequality \eqref{eq:VarIneq} for the displacement $\mathbf{u}|_\Gamma$ on $(0,T]\times \Gamma$ is equivalent to the unilateral frictional contact problem \eqref{prob:strong_formulation}, \eqref{fullcontactbc1}, \eqref{fullcontactbc2} for the solution $\mathbf{u}$ in $(0,T]\times \Omega$.
\end{proposition}
\begin{proof} 
First, we show that the boundary trace $\mathbf{u}|_\Gamma$ satisfies \eqref{eq:VarIneq}, if $\mathbf{u}$ is a solution to \eqref{prob:strong_formulation}, \eqref{fullcontactbc1}, \eqref{fullcontactbc2}.  

In the perpendicular component on $\Gamma_C$, we write the non-penetration boundary condition  \eqref{fullcontactbc1} in terms of $\mathcal{S}$ using the definition \eqref{operator_S}:
\begin{equation}\label{scontactbc2}
{ u}_\perp \geq g\ ,\ \mathcal{S}(\mathbf{u}|_\Gamma)_\perp\geq {f}_\perp\ , \ ({u}_\perp - g)(\mathcal{S}(\mathbf{u}|_\Gamma)_\perp - f_\perp) = 0\ .\end{equation}
The condition \eqref{fullcontactbc2} for the parallel component on $\Gamma_C$ becomes, 
\begin{align}
  \| \mathcal{S}(\mathbf{u}|_\Gamma)_\parallel \|  \leq \mathcal{F},\ \mathcal{S}(\mathbf{u}|_\Gamma)_\parallel \cdot\dot{\mathbf{u}}_\parallel+\mathcal{F}\,\| \dot{\mathbf{u}}_\parallel \| &=0. \label{scontactbc}
\end{align}
On $\Gamma_N$ the Neumann boundary conditions \eqref{prob:strong_formulation_1} imply 
$\mathcal{S}(\mathbf{u}|_\Gamma) = \mathbf{f}.$
Recalling that $\mathbf{f}_\parallel = \mathbf{0}$ on $(0,T]\times  \Gamma_C$, we find for any 
$\mathbf{v} \in \mathcal{C}$
\begin{align*}&\langle\mathcal{S}(\mathbf{u}|_\Gamma) - \mathbf{f},\partial_{t,\parallel}(\mathbf{v}-\mathbf{u}|_\Gamma)\rangle_{0,\Gamma_\Sigma,(0,T]} + j(\mathbf{v}) - j(\mathbf{u}|_\Gamma) \\&= \langle\mathcal{S}(\mathbf{u}|_\Gamma)_\parallel - {\mathbf{f}}_\parallel,(\dot{\mathbf{v}}-\dot{\mathbf{u}}|_\Gamma)_\parallel\rangle_{0,\Gamma_N,(0,T]} + \langle\mathcal{S}(\mathbf{u}|_\Gamma)_\parallel - \mathbf{f}_\parallel,(\dot{\mathbf{v}}-\dot{\mathbf{u}}|_\Gamma)_\parallel\rangle_{0,\Gamma_C,(0,T]}+ j(\mathbf{v}) - j(\mathbf{u}|_\Gamma)\\&\qquad + \langle\mathcal{S}(\mathbf{u}|_\Gamma)_\perp - f_\perp,(\mathbf{v}-\mathbf{u}|_\Gamma)_\perp\rangle_{0,\Gamma_\Sigma,(0,T]} \\
& = 0 + \langle\mathcal{S}(\mathbf{u}|_\Gamma)_\parallel ,(\dot{\mathbf{v}}-\dot{\mathbf{u}}|_\Gamma)_\parallel\rangle_{0,\Gamma_C,(0,T]}+ j(\mathbf{v}) - j(\mathbf{u}|_\Gamma) \\&\qquad + \langle\mathcal{S}(\mathbf{u}|_\Gamma)_\perp - f_\perp,v_\perp-g\rangle_{0,\Gamma_\Sigma,(0,T]}-\langle\mathcal{S}(\mathbf{u}|_\Gamma)_\perp - {f}_\perp,u_\perp-g\rangle_{0,\Gamma_\Sigma,(0,T]}
\end{align*}
We are going to show that the right hand side is $\geq 0$. First note that 
$\langle\mathcal{S}(\mathbf{u}|_\Gamma)_\perp - f_\perp,v_\perp-g\rangle_{0,\Gamma_\Sigma,(0,T]} \geq 0$ and the last term is trivial, so that we only need to show $\langle\mathcal{S}(\mathbf{u}|_\Gamma)_\parallel,(\dot{\mathbf{v}}-\dot{\mathbf{u}}|_\Gamma)_\parallel\rangle_{0,\Gamma_C,(0,T]}+ j(\mathbf{v}) - j(\mathbf{u}|_\Gamma) \geq 0$. From
\eqref{scontactbc}
$ \mathcal{S}(\mathbf{u}|_\Gamma)_\parallel \cdot\dot{\mathbf{u}}_\parallel+\mathcal{F}\,\| \dot{\mathbf{u}}_\parallel \| =0$, and we observe 
$\langle\mathcal{S}(\mathbf{u}|_\Gamma)_\parallel,\dot{\mathbf{u}}_\parallel|_\Gamma\rangle_{0,\Gamma_C,(0,T]}+ j(\mathbf{u}|_\Gamma)=0$ and therefore
$$\langle\mathcal{S}(\mathbf{u}|_\Gamma)_\parallel - \mathbf{f}_\parallel,(\dot{\mathbf{v}}-\dot{\mathbf{u}}_\parallel|_\Gamma)\rangle_{0,\Gamma_C,(0,T]}+ j(\mathbf{v}) - j(\mathbf{u}|_\Gamma) = \langle\mathcal{S}(\mathbf{u}_\parallel|_\Gamma),\dot{\mathbf{v}}_\parallel\rangle_{0,\Gamma_C,(0,T]}+ j(\mathbf{v})\,.$$
Using the first inequality in \eqref{scontactbc}, $$\langle\mathcal{S}(\mathbf{u}|_\Gamma)_\parallel,\dot{\mathbf{v}}_\parallel\rangle_{0,\Gamma_C,(0,T]} \geq -
\langle {\cal F},\|\dot{\bf v}_\parallel\| \rangle_{0,\Gamma_C,(0,T]}= -j(\mathbf{v}),$$ so that
$$ \langle\mathcal{S}(\mathbf{u}_\parallel|_\Gamma),\dot{\mathbf{v}}_\parallel\rangle_{0,\Gamma_C,(0,T]}+ j(\mathbf{v}) \geq 0.$$

Summarizing, we have proved that
$$\langle\mathcal{S}(\mathbf{u}|_\Gamma)-\mathbf{f},\partial_{t,\parallel}(\mathbf{v}-\mathbf{u}|_\Gamma)\rangle_{0,\Gamma_\Sigma,(0,T]} + j(\mathbf{v}) - j(\mathbf{u}|_\Gamma)\geq 0,$$
and \eqref{eq:VarIneq} follows.\\

We now show the converse assertion, that a sufficiently smooth solution ${\bf u}|_\Gamma$ of \eqref{eq:VarIneq} leads to a solution ${\bf u}$ of the frictional contact problem \eqref{prob:strong_formulation}, \eqref{fullcontactbc1}, \eqref{fullcontactbc2}.
First note that the  solution $\mathbf{u}$ in the domain $(0,T]\times  \Omega$ can be recovered from its boundary values on $(0,T]\times  \Gamma$ (see the representation formula \eqref{representation formula} in Section \ref{birs}). We therefore only need to show that the solution $\mathbf{u}$ satisfies the boundary conditions.\\
The homogeneous Dirichlet condition on $\Gamma_D$ is satisfied because $\textbf{u} \in \mathcal{C}$.
 To obtain the traction boundary conditions, choose $\mathbf{v} = \mathbf{u}|_\Gamma + \mathbf{w} \in \mathcal{C}$, for $\mathbf{w}$ { with $w_\perp \geq g-u_\perp$  on $(0,T]\times  \Gamma_C$}.
Then \eqref{eq:VarIneq} implies
\begin{equation}\label{eq:auxproof1}\langle\mathcal{S}(\mathbf{u}|_\Gamma),\partial_{t,\parallel}\mathbf{w}\rangle_{0,\Gamma_\Sigma,(0,T]} + j(\mathbf{u}|_\Gamma + \mathbf{w}) - j(\mathbf{u}|_\Gamma) \geq \langle\mathbf{f},\partial_{t,\parallel}\mathbf{w}\rangle_{0,\Gamma_\Sigma,(0,T]}.\end{equation}
Similarly, with
$\mathbf{v} = \mathbf{u}|_\Gamma - \mathbf{w} \in \mathcal{C}$, with $w_\perp \leq u_\perp-g$ on $(0,T]\times  \Gamma_C$, 
 we obtain
\begin{equation}\label{eq:auxproof2}
-\langle\mathcal{S}(\mathbf{u}|_\Gamma),\partial_{t,\parallel}\mathbf{w}\rangle_{0,\Gamma_\Sigma,(0,T]} + j(\mathbf{u}|_\Gamma - \mathbf{w}) - j(\mathbf{u}|_\Gamma) \geq -\langle\mathbf{f},\partial_{t,\parallel}\mathbf{w}\rangle_{0,\Gamma_\Sigma,(0,T]}.\end{equation}
When we consider $\mathbf{w}=0$ a.e. on $(0,T]\times  \Gamma_C$ in \eqref{eq:auxproof1}, \eqref{eq:auxproof2}, we conclude that
$$\langle\mathcal{S}(\mathbf{u}|_\Gamma),\partial_{t,\parallel}\mathbf{w}\rangle_{0,\Gamma_N,(0,T]} = \langle\mathbf{f},\partial_{t,\parallel}\mathbf{w}\rangle_{0,\Gamma_N,(0,T]} $$
and therefore the traction boundary condition \eqref{prob:strong_formulation_1}: $\mathbf{p} = \mathcal{S}(\mathbf{u}|_\Gamma) = \mathbf{f}$  on $\Gamma_N$ holds.\\

To obtain the contact boundary conditions on $\Gamma_C$ we first verify the non-penetration condition \eqref{scontactbc2} in the normal component. Note that $u_\perp\geq g$ on $\Gamma_C$ because $\mathbf{u}|_\Gamma \in \mathcal{C}$. Now choose $\mathbf{w}_\parallel = \mathbf{0}$, $w_\perp \geq 0$ a.e. on $(0,T]\times \Gamma_C$ in \eqref{eq:auxproof1}. 
As we already know $\mathcal{S}(\mathbf{u}|_\Gamma)=\mathbf{f}$ on $\Gamma_N$, we  obtain 
$$0\geq \langle-\mathcal{S}(\mathbf{u}|_\Gamma)+\mathbf{f},\partial_{t,\parallel}\mathbf{w}\rangle_{0,\Gamma_\Sigma,(0,T]} = \langle-\mathcal{S}(\mathbf{u}|_\Gamma)_\perp+f_\perp,w_\perp\rangle_{0,\Gamma_C,(0,T]}.$$
The assertion $\mathcal{S}(\mathbf{u}|_\Gamma)_\perp\geq f_\perp$ a.e. on $(0,T] \times \Gamma_C$ immediately follows. 

The remaining condition in \eqref{scontactbc2} is obtained by choosing $\mathbf{w}$ such that $\mathbf{w}=  (g-u_\perp|_\Gamma)\mathbf{n}$ a.e. on $(0,T]\times  \Gamma_C$. Then \eqref{eq:auxproof1}, \eqref{eq:auxproof2} imply
\begin{equation} \label{eq:auxproof3}
0 = \langle-\mathcal{S}(\mathbf{u}|_\Gamma)_\perp+f_\perp,g-u_\perp|_\Gamma\rangle_{0,\Gamma_C,(0,T]}.  \end{equation}
From above we recall that $g-u_\perp|_\Gamma\leq 0$ and $-\mathcal{S}(\mathbf{u}|_\Gamma)_\perp+f_\perp \leq 0$ a.e., so that  from \eqref{eq:auxproof3} we conclude $$0= (-\mathcal{S}(\mathbf{u}|_\Gamma)_\perp+f_\perp)(g-u_\perp|_\Gamma)$$ a.e. on $(0,T]\times \Gamma_C$, i.e.~the equality in \eqref{scontactbc2}.\\

We finally verify the friction conditions \eqref{scontactbc} on $(0,T]\times  \Gamma_C$. Recall the Neumann condition and, from below \eqref{prob:strong_formulation}, that $\mathbf{f}_\parallel = \mathbf{0}$ there. We again use \eqref{eq:auxproof1}, here with $\mathbf{w}$ such that $w_\perp = 0$, as well as the definition of the functional $j$: 
$$\langle\mathcal{S}(\mathbf{u}|_\Gamma)_\parallel,\dot{\mathbf{w}}_\parallel\rangle_{0,\Gamma_C,(0,T]} + \langle {\cal F},\|(\dot{\mathbf{u}}|_\Gamma + \dot{\mathbf{w}})_\parallel\| - \|\dot{\mathbf{u}}_\parallel|_\Gamma\| \rangle_{0,\Gamma_C,(0,T]}
\geq 0.$$
From the triangle inequality, $\|(\dot{\mathbf{u}}|_\Gamma + \dot{\mathbf{w}})_\parallel\| - \|\dot{\mathbf{u}}_\parallel|_\Gamma\| \leq \|\dot{\mathbf{w}}_\parallel\|$, so that
$$ \langle {\cal F},\|\dot{\mathbf{w}}_\parallel\| \rangle_{0,\Gamma_C,(0,T]}
\geq -\langle\mathcal{S}(\mathbf{u}|_\Gamma)_\parallel,\dot{\mathbf{w}}_\parallel\rangle_{0,\Gamma_C,(0,T]}.$$
An analogous argument with \eqref{eq:auxproof2}, instead of \eqref{eq:auxproof1}, leads to
$$ 
\langle {\cal F},\|\dot{\mathbf{w}}_\parallel\| \rangle_{0,\Gamma_C,(0,T]}
\geq \langle\mathcal{S}(\mathbf{u}|_\Gamma)_\parallel,\dot{\mathbf{w}}_\parallel\rangle_{0,\Gamma_C,(0,T]}.$$
We conclude the first condition in \eqref{scontactbc}: $\| \mathcal{S}(\mathbf{u}|_\Gamma)_\parallel \|  \leq \mathcal{F}$  a.e. on $(0,T]\times  \Gamma_C$.

The second condition in \eqref{scontactbc} then follows from \eqref{eq:auxproof1}, \eqref{eq:auxproof2} with $\mathbf{w} = \mathbf{u}_\parallel$. Indeed, from \eqref{eq:auxproof1} we have that $\langle\mathcal{S}(\mathbf{u}|_\Gamma)_\parallel-\mathbf{f}_\parallel,\dot{\mathbf{u}}_\parallel\rangle_{0,\Gamma_\Sigma
,(0,T]} + j(\mathbf{u}|_\Gamma)\geq 0$, while \eqref{eq:auxproof2} implies the opposite inequality, $\langle\mathcal{S}(\mathbf{u}|_\Gamma)_\parallel-\mathbf{f}_\parallel,\dot{\mathbf{u}}_\parallel\rangle_{0,\Gamma_\Sigma
,(0,T]} + j(\mathbf{u}|_\Gamma)\leq 0$. As $\mathcal{S}(\mathbf{u}|_\Gamma) = \mathbf{f}$ on $(0,T]\times \Gamma_N$ and $\mathbf{f}_\parallel = \mathbf{0}$, $\| \mathcal{S}(\mathbf{u}|_\Gamma)_\parallel \|  \leq \mathcal{F}$ on $(0,T]\times \Gamma_C$, we obtain that $\mathcal{S}(\mathbf{u}|_\Gamma)_\parallel \cdot\dot{\mathbf{u}}_\parallel+\mathcal{F}\,\| \dot{\mathbf{u}}_\parallel \|=0$ a.e. on $(0,T]\times  \Gamma_C$\\

Therefore, we proved that a sufficiently regular solution of \eqref{eq:VarIneq} leads to a solution $\mathbf{u}$ of the frictional contact problem \eqref{prob:strong_formulation}, \eqref{fullcontactbc1}, \eqref{fullcontactbc2}.
\end{proof}

For the numerical approximation of the frictional contact problem we consider a mixed formulation of the variational inequality \eqref{eq:VarIneq}, which involves both the displacement and the contact forces. A precise statement, amenable to the discretization and error analysis below, requires us to introduce appropriate function spaces. 
Based on the scalar products \eqref{0_product}, \eqref{sigma_product}, we define space-time Sobolev spaces $H^r(\mathcal{I},\tilde{H}^{s}({\Gamma'}))$ on subsets ${\Gamma'} \subset \Gamma$ and for time intervals $\mathcal{I}=[0,T],\mathbb{R}^+$, which are introduced in Appendix.\\

The mixed formulation of problem \eqref{eq:VarIneq} relies on the closed convex set
\begin{align}
M^+(\mathcal{F})&:=\left\lbrace \pmb{\mu} \in H^{1/2}([0,T],\tilde{H}^{-1/2}(\Gamma_C))^d: \right.\nonumber \\&\qquad \left.\left\langle \pmb{\mu},\mathbf{w}\right\rangle_{0,\Gamma_C,(0,T]} \leq \left\langle\mathcal{F}, \|\textbf{w}_\parallel\|\right\rangle_{0,\Gamma_C,(0,T]},\, \forall \ \mathbf{w} \in H^{-1/2}([0,T],H^{1/2}(\Gamma_\Sigma))^d, w_\perp\leq 0 
\right\rbrace\,.
\end{align}

Let us observe that the above definition of $M^+(\mathcal{F})$ implies $\left\langle \mu_\perp,w_\perp\right\rangle_{0,\Gamma_C,(0,T]} \leq 0$, a weak formulation of the  inequality $\mu_\perp\geq 0$. Moreover, it can be shown that $\left\langle \pmb{\mu}_\parallel,\mathbf{w}_\parallel\right\rangle_{0,\Gamma_C,(0,T]} \leq \left\langle\mathcal{F}, \|\textbf{w}_\parallel\|\right\rangle_{0,\Gamma_C,(0,T]}$, which is a weak formulation of the  inequality $\|\pmb{\mu}_\parallel\| \leq \mathcal{F}$.

Now, let us show that \begin{equation}\label{eq:lambdamp}{\pmb{\lambda}:= \mathcal{S}{\bf u}- {\bf f}} \in M^+(\mathcal{F}).
\end{equation}
At first note that because of the traction boundary condition $\pmb{\lambda} = \mathbf{0}$ in $(0,T]\times \Gamma_N$. 
Next, we show $\left\langle \lambda_\perp,w_\perp\right\rangle_{0,\Gamma_C,(0,T]} \leq 0$ for all $w_\perp \leq 0$. This is a weak formulation of the non-penetration boundary condition $p_\perp\geq {f}_\perp$, and the proof is similar to Proposition \ref{prop:varineq}. In fact,
choose ${\bf v} = {\bf u} - {\bf w}$ with ${\bf w}_\parallel = {\bf 0}$  in \eqref{eq:VarIneq} and note that $u_\perp - w_\perp \geq g$; we obtain $$-\left\langle \lambda_\perp,w_\perp\right\rangle_{0,\Gamma_C,(0,T]} = -\langle {(\mathcal{S} {\bf u})_\perp}- f_\perp, w_\perp\rangle_{0,\Gamma_C,(0,T]} \geq 0,$$ or $\langle \lambda_\perp, w_\perp\rangle_{0,\Gamma_C,(0,T]} \leq 0$. 
To show \eqref{eq:lambdamp}, it remains to prove a weak formulation of the friction boundary condition, 
i.e.~that for all $\mathbf{w}_\parallel$: 
\begin{equation}\label{tesi}
\left\langle \pmb{\lambda}_\parallel,\mathbf{w}_\parallel\right\rangle_{0,\Gamma_C,(0,T]} = \left\langle (\mathcal{S} {\bf u})_\parallel,\mathbf{w}_\parallel\right\rangle_{0,\Gamma_C,(0,T]} \leq \left\langle\mathcal{F}, \|\textbf{w}_\parallel\|\right\rangle_{0,\Gamma_C,(0,T]}, 
\end{equation}
recalling that $\mathbf{f}_\parallel=\mathbf{0}$ on $(0,T]\times  \Gamma_C$.
For this, choose ${\bf v} = {\bf u} - {\bf w}$ with $w_\perp =0$ in \eqref{eq:VarIneq} and use $\|({\bf u}-{\bf w})_\parallel\|- \|{\bf u}_\parallel\| \leq \|{\bf w}_\parallel\|$. We obtain \begin{align*}&-\langle {(\mathcal{S} {\bf u})_\parallel}, {\bf w}_\parallel\rangle_{0,\Gamma_C,(0,T]} + \langle {\cal F},\|\mathbf{w}_\parallel\| \rangle_{0,\Gamma_C,(0,T]}\,
\\&\qquad \qquad\geq -\langle {(\mathcal{S} {\bf u})_\parallel}, {\bf w}_\parallel\rangle_{0,\Gamma_C,(0,T]} + \langle {\cal F},\|({\bf u}-{\bf w})_\parallel\|- \|{\bf u}_\parallel\| \rangle_{0,\Gamma_C,(0,T]}\,
\geq 0,\end{align*} from which \eqref{tesi} follows. Altogether, we conclude \eqref{eq:lambdamp}. \\

With $M^+(\mathcal{F})$ as set of Lagrange multipliers, the mixed formulation of the variational inequality \eqref{eq:VarIneq} reads, for data $g \in H^{1/2}([0,T],H^{1/2}(\Gamma_C))$ and $\textbf{f}\in H^{1/2}([0,T],\tilde{H}^{-1/2}(\Gamma_\Sigma))^d$:\\ 

\textit{find} $(\textbf{u},{\pmb{\lambda}}) \in H^{1/2}([0,T],\tilde{H}^{1/2}(\Gamma_\Sigma))^d \times M^+(\mathcal{F})$ \textit{such that} 
\begin{subequations} \label{eq:MixedProblem}
 \begin{alignat}{2}
\left\langle \mathcal{S}\textbf{u},\textbf{v} \right\rangle_{0,\Gamma_\Sigma,(0,T]} - {\left\langle \pmb{\lambda},\mathbf{v} \right\rangle_{0,\Gamma_C,(0,T]}} &= \left\langle \textbf{f},\textbf{v} \right\rangle_{0,\Gamma_\Sigma,(0,T]} &\quad &\forall \textbf{v}\in H^{1/2}([0,T],\tilde{H}^{1/2}(\Gamma_\Sigma))^d \label{eq:WeakMixedVarEq}\\
{\left\langle \pmb{\mu} -\pmb{\lambda},\partial_{t,\parallel}\bf{u} \right\rangle_{0,\Gamma_C,(0,T]}} & \geq \left\langle g,\mu_\perp-\lambda_\perp \right\rangle_{0,\Gamma_C,(0,T]} &\quad &\forall \pmb{\mu} \in M^+(\mathcal{F}) .\label{eq:ContContactConstraints}
\end{alignat}
\end{subequations}

\begin{theorem}\label{equivtheorem:unilateral}
The mixed formulation \eqref{eq:MixedProblem}, the variational inequality \eqref{eq:VarIneq} and the unilateral frictional contact problem \eqref{prob:strong_formulation}, \eqref{fullcontactbc1}, \eqref{fullcontactbc2} are equivalent.
\end{theorem}
\begin{proof}

The equivalence of the variational inequality and the unilateral frictional contact problem follows from Proposition \ref{prop:varineq}. 

We show that if $\mathbf{u}$ is a solution to the unilateral frictional contact problem \eqref{prob:strong_formulation}, \eqref{fullcontactbc1}, \eqref{fullcontactbc2}, then $(\mathbf{u}|_\Gamma, \mathcal{S}(\mathbf{u}|_\Gamma)-\mathbf{f})$ is a solution to the mixed formulation \eqref{eq:MixedProblem}. Indeed, $\pmb{\lambda}=\mathcal{S}(\mathbf{u}|_\Gamma)-\mathbf{f} \in M^+(\mathcal{F})$ 
as shown above
and  \eqref{eq:WeakMixedVarEq} is satisfied by definition of $\pmb{\lambda}$. To see \eqref{eq:ContContactConstraints}, we show that for all $\pmb{\mu} \in M^+(\mathcal{F})$
$$\left\langle \pmb{\mu}_\parallel -\pmb{\lambda}_\parallel,\dot{\bf{u}}_\parallel \right\rangle_{0,\Gamma_C,(0,T]}  \geq 0 \quad \text{and} \quad \left\langle \mu_\perp -\lambda_\perp,u_\perp-g \right\rangle_{0,\Gamma_C,(0,T]}\geq 0.$$
For the inequality involving the parallel components, observe that $|\left\langle \pmb{\mu}_\parallel,\dot{\bf{u}}_\parallel \right\rangle_{0,\Gamma_C,(0,T]}|  \leq \left\langle \mathcal{F},\|\dot{\bf{u}}_\parallel\| \right\rangle_{0,\Gamma_C,(0,T]}$, because $\pmb{\mu} \in M^+(\mathcal{F})$. From \eqref{fullcontactbc2} and remembering that $\mathbf{f}_\parallel=\mathbf{0}$, $\left\langle -\pmb{\lambda}_\parallel,\dot{\bf{u}}_\parallel \right\rangle_{0,\Gamma_C,(0,T]} =\left\langle \mathcal{F},\|\dot{\bf{u}}_\parallel\| \right\rangle_{0,\Gamma_C,(0,T]}$. Combining these estimates, the claimed inequality follows: $$\left\langle \pmb{\mu}_\parallel -\pmb{\lambda}_\parallel,\dot{\bf{u}}_\parallel \right\rangle_{0,\Gamma_C,(0,T]} \geq -|\left\langle\pmb{\mu}_\parallel ,\dot{\bf{u}}_\parallel \right\rangle_{0,\Gamma_C,(0,T]}| + \left\langle \mathcal{F},\|\dot{\bf{u}}_\parallel\| \right\rangle_{0,\Gamma_C,(0,T]} \geq 0.$$
For the  inequality involving the perpendicular component, recall that  $ \lambda_\perp (u_\perp-g) =0$ and $u_\perp-g \geq 0$ from \eqref{fullcontactbc1}. Hence, from $\pmb{\mu} \in M^+(\mathcal{F})$ we conclude  $$\left\langle \mu_\perp -\lambda_\perp,u_\perp-g \right\rangle_{0,\Gamma_C,(0,T]} = \left\langle \mu_\perp,u_\perp-g \right\rangle_{0,\Gamma_C,(0,T]}\geq 0.$$
This shows that $(\mathbf{u}|_\Gamma, \mathcal{S}(\mathbf{u}|_\Gamma)-\mathbf{f})$ is indeed a solution to the mixed formulation \eqref{eq:MixedProblem}.\\

It remains to show the converse, i.e.~that a solution to the mixed formulation \eqref{eq:MixedProblem} corresponds to a solution to the unilateral frictional contact problem \eqref{prob:strong_formulation}, \eqref{fullcontactbc1}, \eqref{fullcontactbc2}. This is shown similarly to the proof of Theorem 14 in \cite{contact} for the wave equation.\\
First, a solution $\mathbf{u}$ in the domain $(0,T]\times  \Omega$ can be recovered from its boundary values on $(0,T]\times  \Gamma$ (see the representation formula \eqref{representation formula} in Section \ref{birs}). We now show that the solution $\mathbf{u}$ satisfies the boundary conditions. 

The null Dirichlet condition \eqref{prob:strongformulation_d} is immediate from the chosen function spaces in the mixed formulation, while the traction boundary condition \eqref{prob:strong_formulation_1} in $(0,T]\times  \Gamma_N$ follows from \eqref{eq:WeakMixedVarEq}, because $\pmb{\lambda}=\mathbf{0}$ there. For the contact conditions, $\pmb{\lambda} \in M^+(\mathcal{F})$ implies $\|\mathbf{p}_\parallel\|\leq \mathcal{F}$ and $p_\perp \geq f_\perp$. It remains to show
$$u_\perp \geq g, \quad (p_\perp-f_\perp)(u_\perp-g)=0\quad \text{and} \quad \mathbf{p}_\parallel\cdot \dot{\mathbf{u}}_\parallel + \mathcal{F}\|\dot{\mathbf{u}}_\parallel\|=0.$$ 
First, choosing $\pmb{\mu} \in M^+(\mathcal{F})$ with $\pmb{\mu}_\parallel =\pmb{\lambda}_\parallel$ in \eqref{eq:ContContactConstraints}, we obtain $\left\langle \mu_\perp -\lambda_\perp,u_\perp-g \right\rangle_{0,\Gamma_C,(0,T]}\geq 0$. 
With the further choice $\mu_\perp = \lambda_\perp + \widetilde{\mu}_\perp$, we find $\left\langle \widetilde{\mu}_\perp, u_\perp-g \right\rangle_{0,\Gamma_C,(0,T]}\geq 0$ for all $\widetilde{\mu}_\perp\geq 0$, hence $u_\perp-g \geq 0$ a.e. on $(0,T]\times  \Gamma_C$.

Similarly, with $\mu_\perp = 0$, respectively $\mu_\perp = 2\lambda_\perp$, we find $\mp \left\langle \lambda_\perp,u_\perp-g \right\rangle_{0,\Gamma_C,(0,T]}\geq 0$, and therefore $\left\langle \lambda_\perp,u_\perp-g \right\rangle_{0,\Gamma_C,(0,T]}= 0$. As both $\lambda_\perp\geq 0$ and $u_\perp-g \geq 0$ we conclude $\lambda_\perp(u_\perp-g) = 0$ a.e. on $(0,T]\times  \Gamma_C$. 

To see the remaining identity, $\mathbf{p}_\parallel\cdot \dot{\mathbf{u}}_\parallel + \mathcal{F}\|\dot{\mathbf{u}}_\parallel\|=0$, we may assume $\dot{\mathbf{u}}_\parallel \neq \mathbf{0}$.  Choose $\pmb{\mu} \in M^+(\mathcal{F})$ with $\mu_\perp =\lambda_\perp$ and $\pmb{\mu}_\parallel = - \mathcal{F} \frac{\dot{\bf{u}}_\parallel}{\|\dot{\bf{u}}_\parallel\|}$, respectively $\pmb{\mu}_\parallel = 2 \pmb{\lambda}_\parallel+ \mathcal{F} \frac{\dot{\bf{u}}_\parallel}{\|\dot{\bf{u}}_\parallel\|}$, in \eqref{eq:ContContactConstraints}, to obtain $$0\leq \left\langle \pmb{\mu}_\parallel -\pmb{\lambda}_\parallel,\dot{\bf{u}}_\parallel \right\rangle_{0,\Gamma_C,(0,T]} = \textstyle{\mp\left\langle \mathcal{F} \frac{\dot{\bf{u}}_\parallel}{\|\dot{\bf{u}}_\parallel\|} +\textbf{p}_\parallel,\dot{\bf{u}}_\parallel \right\rangle_{0,\Gamma_C,(0,T]}}.$$ This shows $0=\left\langle \mathcal{F} \frac{\dot{\bf{u}}_\parallel}{\|\dot{\bf{u}}_\parallel\|} +\textbf{p}_\parallel,\dot{\bf{u}}_\parallel \right\rangle_{0,\Gamma_C,(0,T]}$, and therefore $0=\mathcal{F}\|\dot{\mathbf{u}}_\parallel\|+\mathbf{p}_\parallel\cdot \dot{\mathbf{u}}_\parallel$ a.e. on $(0,T]\times  \Gamma_C$. 
\end{proof}

In spite of the interest in dynamic contact problems, their rigorous mathematical analysis is widely open. Even the existence of solutions is only known for certain dissipative materials or for not perfectly rigid, but dissipative obstacles, see e.g.~\cite{cocou}. Without dissipation the existence of a solution has been proven for simplified problems involving the scalar wave equation in special geometries \cite{cooper, lebeau}, without friction. The analysis of \cite{cooper} served as a starting point for the development of boundary element methods for the dynamic Signorini problem in \cite{contact}.

\subsection{Discretization} 

To solve 
the mixed formulation \eqref{eq:MixedProblem}, in a discretized form,  we  consider a uniform decomposition of the time interval $[0,T]$ with time step $\Delta t=\frac{T}{N_{\Delta t}}$, $N_{\Delta t}\in\mathbb{N}^{+}$, generated by the time instants $t_{\ell}=\ell\Delta t$, $\ell=0,\ldots,N_{\Delta t}$. We define the corresponding spaces
\begin{equation}\label{time_set}
\begin{array}{ll}
 V^{-1}_{\Delta t}=&\left\lbrace v_{\Delta t}\in L^2([0,T])
\: : \: {v_{\Delta t}}_{\vert_{[t_\ell,t_{\ell+1}]}}\in \mathcal{P}_0, \:\forall \ell=0,..., N_{\Delta t}-1 \right\rbrace ,\\[4pt]
V^0_{\Delta t}=&\left\lbrace r_{\Delta t}\in C^0([0,T])
\: : \: {r_{\Delta t}}_{\vert_{[t_\ell,t_{\ell+1}]}}\in \mathcal{P}_1, \:\forall \ell=0,..., N_{\Delta t}-1,\: v(0)=0  \right\rbrace,
\end{array}
\end{equation}
where $\mathcal{P}_s$, $s\geq 0$, is the space of the algebraic polynomials of degree $s$. For the space discretization with $d=2$, we introduce a boundary mesh constituted by a set of straight line segments $\mathcal{T}=\left\lbrace e_1,... ,e_M \right\rbrace$ such that $h_i:=length(e_i)\leqslant h$, $e_i\cap e_j=\emptyset$ if $i\neq j$ and $\cup_{i=1}^M \overline{e}_i=\overline{\Gamma}$ if $\Gamma$ is polygonal, or a suitably fine approximation of $\Gamma$ otherwise. 
For $d=3$, we assume that $\Gamma$ is triangulated by  $\mathcal{T}=\{e_1,\cdots,e_{M}\}$, with $h_i:=diam(e_i)\leqslant h$, $e_i\cap e_j=\emptyset$ if $i\neq j$ and, if $\overline{e_i}\cap \overline{e_j} \neq \emptyset$, the intersection is either an edge or a vertex of both triangles. On $\mathcal{T}$  we consider the spaces of piecewise polynomial functions 
\begin{equation}\label{pol_space_4_discontinuous_f}
X^{-1}_{h,\Gamma}=\left\lbrace w_h\in L^2(\Gamma)\: : \: w_h\vert_{e_i}\in \mathcal{P}_s,\: e_i\in \mathcal{T}  \right\rbrace\subset H^{-1/2}(\Gamma), \quad X^{-1}_{h,\Gamma'} = X^{-1}_{h,\Gamma} \cap \widetilde{H}^{-1/2}(\Gamma'), 
\end{equation}
\begin{equation}\label{pol_space_4_continuous_f}
X^{0}_{h,\Gamma}=\left\lbrace w_h\in C^0(\Gamma)\: : \: w_h\vert_{e_i}\in \mathcal{P}_s,\: e_i\in \mathcal{T}  \right\rbrace \subset H^{1/2}(\Gamma), \quad X^{0}_{h,\Gamma'} = X^{0}_{h,\Gamma} \cap \widetilde{H}^{1/2}(\Gamma'),
\end{equation}
where $\Gamma' \subset \Gamma$. \\
The full discretization of \eqref{eq:MixedProblem} involves the following subspace of $M^+(\mathcal{F})$ {on a second space-time mesh with mesh parameters $H$, $\Delta T$}:
\begin{align}
M^+_{H,\Delta T}(\mathcal{F}):=\left\lbrace \pmb{\mu}_{H,\Delta T} \in (X^{-1}_{H,\Gamma_C}\otimes V^{-1}_{\Delta T})^d: \mu_{\perp,H,\Delta T} \geq 0 {\text{ and } \|\pmb{\mu}_{\parallel,H,\Delta T}\|\leq \mathcal{F}} \text{ on }  (0,T] \times\Gamma_C  \right\rbrace .
\end{align}
Denoting a discretized version of the Poincar\'e-Steklov operator by ${\cal S}_{h,\Delta t}$, it reads:\\

\textit{find} $(\textbf{u}_{h,\Delta t},\pmb{\lambda}_{H,\Delta T}) \in (X^0_{h,\Gamma_\Sigma}\otimes V^0_{\Delta t})^d \times M^+_{H,\Delta T}(\mathcal{F})$ \textit{such that} 
\begin{subequations} \label{eq:MixedProblemh}
\begin{alignat}{2}
 &\left\langle \mathcal{S}_{h,\Delta t}\textbf{u}_{h,\Delta t},\textbf{v}_{h,\Delta t} \right\rangle_{0,\Gamma_\Sigma,(0,T]} -\left\langle \pmb{\lambda}_{H,\Delta T},{\bf{v}}_{h,\Delta t} \right\rangle_{0,\Gamma_C,(0,T]} = \left\langle \textbf{f},{\bf{v}}_{h,\Delta t} \right\rangle_{0,\Gamma_\Sigma,(0,T]} \quad\forall \textbf{v}_{h,\Delta t} \in (X^0_{h,\Gamma_\Sigma}\otimes V^0_{\Delta t})^d \label{eq:WeakMixedVarEqh} \\
&\left\langle \pmb{\mu}_{H,\Delta T} -\pmb{\lambda}_{H,\Delta T},\partial_{t,\parallel}{\bf{u}}_{h,\Delta t} \right\rangle_{0,\Gamma_C,(0,T]}  \geq \left\langle g,\mu_{\perp,H,\Delta T}-\lambda_{\perp, H,\Delta T} \right\rangle_{0,\Gamma_C,(0,T]}  \quad\quad\quad\;\:\forall \pmb{\mu}_{H,\Delta T} \in M^+_{H,\Delta T}(\mathcal{F}). \label{eq:ContContactConstraintsh}
\end{alignat}
\end{subequations}

A standard solver for the discrete formulation \eqref{eq:MixedProblemh} is given by the Uzawa algorithm, which involves the $L^2$-projection {$\text{Pr}_C : L^2((0,T]; L^2(\Gamma ) )^d \to M^+_{H,\Delta T}(\mathcal{F})$, defined by 
$$\text{Pr}_C({\bf{w}})_\perp = \max\{{w}_\perp,0\},\qquad \text{Pr}_C(\bf{w})_\parallel = \begin{cases}{\bf{w}}_\parallel,\quad  & \text{ if } \|\bf{w}_\parallel\| \leq\mathcal{F},\\ \mathcal{F}\frac{{\bf{w}}_\parallel}{\|{\bf{w}}_\parallel\|}, & \text{ if } \|\bf{w}_\parallel\|  > \mathcal{F},\end{cases} $$
as well as the function $\bf{g}$ such that $g_\perp = g$,  $\mathbf{g}_\parallel = \mathbf{0}$ (i.e.~${\bf{g}}=-g \bf{n}$). This algorithm is given explicitly by: 
\begin{algorithm}[H]
\caption{(Uzawa algorithm for unilateral frictional contact)}
\label{alg1}
\begin{algorithmic}
\STATE Fix $\rho>0$.
\STATE $k=0$,  $\pmb{\lambda}_{H,\Delta T}^{(0)}= 0$
\WHILE{stopping criterion not satisfied}
\STATE \textbf{solve}$\quad$ equation \eqref{eq:WeakMixedVarEqh} for ${\bf u}_{h,\Delta t}^{(k)}$
\STATE \textbf{compute}$\quad$ $\pmb{\lambda}^{(k+1)}_{H,\Delta T}= \text{Pr}_C (\pmb{\lambda}_{H,\Delta T}^{(k)} -\rho (\partial_{t,\parallel}{\bf{u}}_{h,\Delta t}^{(k)}-\bf{g})) $ 
\STATE $k \leftarrow k+1$
\ENDWHILE
\end{algorithmic}
\end{algorithm}
The implementation of Algorithm \ref{alg1} will be discussed in Section  \ref{sec:psh}.
} For the contact problem without friction the convergence of the Uzawa method has been proved in \cite{ourpaper2}.
\subsection{Analysis}
In this section we state an a priori error estimate for the mixed formulation of the unilateral frictional contact problem \eqref{eq:MixedProblem} and its discretization \eqref{eq:MixedProblemh}. It builds on the corresponding analysis for the Signorini problem in \cite{ourpaper2} and requires the  formulation of \eqref{eq:MixedProblem} presented for times $t \in \mathbb{R}^+$, using the generalized inner product \eqref{sigma_product} for $\sigma>0$ and considering
\begin{align}
{\widetilde{M}}^+(\mathcal{F})&:=\left\lbrace \pmb{\mu} \in H^{1/2}_\sigma(\mathbb{R}^+,\tilde{H}^{-1/2}(\Gamma_C))^d: \right.\nonumber \\&\qquad {\left. \left\langle \pmb{\mu},\mathbf{v}\right\rangle_{\sigma,\Gamma_C,\mathbb{R}^+} \leq \left\langle\mathcal{F}, \|\textbf{v}_\parallel\|\right\rangle_{\sigma,\Gamma_C,\mathbb{R}^+},\, \forall \ \mathbf{v} \in H^{-1/2}_\sigma(\mathbb{R}^+,H^{1/2}(\Gamma_\Sigma))^d, v_\perp\leq 0 
\right\rbrace\ .}
\end{align}
For $\textbf{f} \in H^{1/2}_\sigma\left(\mathbb{R}^+,H^{-1/2}(\Gamma_\Sigma)\right)^d$, {$\mathcal{F} \in L^\infty(\Gamma_C)$} and $g \in H^{1/2}_\sigma(\mathbb{R}^+,H^{1/2}(\Gamma_C))$ the formulation is given by:\\

\noindent \textit{find} $(\textbf{u},\pmb{\lambda}) \in H^{1/2}_\sigma(\mathbb{R}^+,\tilde{H}^{1/2}(\Gamma_\Sigma))^d \times {\widetilde{M}}^+({\cal F})$ \textit{such that} 
\begin{subequations} \label{eq:MixedProblemsigma}
 \begin{alignat}{2}
\left\langle \mathcal{S}\textbf{u},\textbf{v} \right\rangle_{\sigma,\Gamma_\Sigma,\mathbb{R}^+} - {\left\langle \pmb{\lambda},\mathbf{v} \right\rangle_{\sigma,\Gamma_C,\mathbb{R}^+}} &= \left\langle \textbf{f},\textbf{v} \right\rangle_{\sigma,\Gamma_\Sigma,\mathbb{R}^+} &\quad &\forall \textbf{v}\in H^{1/2}_\sigma(\mathbb{R}^+,\tilde{H}^{1/2}(\Gamma_\Sigma))^d \label{eq:WeakMixedVarEqsigma}\\
{\left\langle \pmb{\mu} -\pmb{\lambda},\partial_{t,\parallel}\bf{u} \right\rangle_{\sigma,\Gamma_C,\mathbb{R}^+}} & \geq \left\langle g,\mu_\perp-\lambda_\perp \right\rangle_{\sigma,\Gamma_C,\mathbb{R}^+} &\quad &\forall \pmb{\mu} \in {\widetilde{M}}^+(\mathcal{F}) .\label{eq:ContContactConstraintssigma}
\end{alignat}
\end{subequations}

The corresponding discretization, generalizing \eqref{eq:MixedProblemh}, is given by:\\

\noindent \textit{find} $(\textbf{u}_{h,\Delta t},\pmb{\lambda}_{H,\Delta T}) \in (X^0_{h,\Gamma_\Sigma}\otimes V^0_{\Delta t})^d \times M^+_{H,\Delta T}$ \textit{such that}
\begin{subequations} \label{eq:MixedProblemhsigma}
\begin{alignat}{2}
 &\left\langle \mathcal{S}_{h,\Delta t}\textbf{u}_{h,\Delta t},\textbf{v}_{h,\Delta t} \right\rangle_{\sigma,\Gamma_\Sigma,\mathbb{R}^+} -\left\langle \pmb{\lambda}_{H,\Delta T},{\bf{v}}_{h,\Delta t} \right\rangle_{\sigma,\Gamma_C,\mathbb{R}^+} = \left\langle \textbf{f},{\bf{v}}_{h,\Delta t} \right\rangle_{\sigma,\Gamma_\Sigma,\mathbb{R}^+} \ \forall \textbf{v}_{h,\Delta t} \in (X^0_{h,\Gamma_\Sigma}\otimes V^0_{\Delta t})^d \label{eq:WeakMixedVarEqhsigma} \\
&\left\langle \pmb{\mu}_{H,\Delta T} -\pmb{\lambda}_{H,\Delta T},\partial_{t,\parallel}{\bf{u}}_{h,\Delta t} \right\rangle_{\sigma,\Gamma_C,\mathbb{R}^+}  \geq \left\langle g,\mu_{\perp,H,\Delta T}-\lambda_{\perp, H,\Delta T} \right\rangle_{\sigma,\Gamma_C,\mathbb{R}^+}  \quad\quad\:\forall \pmb{\mu}_{H,\Delta T} \in M^+_{H,\Delta T}(\mathcal{F})\ , \label{eq:ContContactConstraintshsigma}
\end{alignat}
\end{subequations}
where the involved discrete functional spaces contain functions supported in a finite number of time steps.\\
A key ingredient to obtain the a priori error estimate is the following inf-sup estimate. For scalar ${\lambda}_{ H, \Delta T}$ the proof of the following Theorem \ref{discInfSup} is given in \cite[Theorem 15]{contact}. The estimate extends verbatim to vector-valued $\pmb{\lambda}_{ H, \Delta T}$, by applying it to the individual components: 

\begin{theorem}\label{discInfSup}
Let $C>0$ sufficiently small and $\frac{\max\{h, \Delta t\}}{\min\{H, \Delta T\}}<C$. Then there exists $\alpha>0$ such that 
$\forall\,\pmb{\lambda}_{H, \Delta T} \in  (X^{-1}_{H,\Gamma_C}\otimes V^{-1}_{\Delta T})^d$: 
$$\sup_{\mathbf{v}_{h,\Delta t} \in (X^0_{h,\Gamma}\otimes V^0_{\Delta t})^d} \frac{\langle \mathbf{v}_{h,\Delta t}, \pmb{\lambda}_{ H, \Delta T}\rangle_{\sigma, \Gamma_C, \mathbb{R}^+}}{\|\mathbf{v}_{h,\Delta t}\|_{ 0,\frac{1}{2}, \sigma, \ast}} \geq \alpha\, \|\pmb{\lambda}_{H,\Delta T}\|_{0, -\frac{1}{2}, \sigma }\ .$$
\end{theorem}

\begin{theorem}\label{apriori}
Let $(\textbf{u},\pmb{\lambda}) \in H^{1/2}_\sigma(\mathbb{R}^+,\tilde{H}^{1/2}(\Gamma_\Sigma))^d \times {\widetilde{M}}^+({\cal F})$ be a solution to the mixed problem \eqref{eq:MixedProblemsigma} and $(\textbf{u}_{h,\Delta t},\pmb{\lambda}_{H,\Delta T}) \in (X^0_{h,\Gamma_\Sigma}\otimes V^0_{\Delta t})^d \times M^+_{H,\Delta T}$ a solution of the discretized mixed problem \eqref{eq:MixedProblemhsigma}. Assume that $\mathcal{S}$ is coercive. Then for a sufficiently small constant $C>0$ and $\frac{\max\{h, \Delta t\}}{\min\{H, \Delta T\}}<C$, the following a priori estimates hold: \\
\begin{align}
\| \pmb{\lambda} -\pmb{\lambda}_{H,\Delta T} \|_{0,-\frac{1}{2}, \sigma}  &\lesssim_\sigma \inf \limits_{\tilde{\pmb{\lambda}}_{H,\Delta T} \in M^+_{H,\Delta T}(\mathcal{F})} \| \pmb{\lambda} - \tilde{\pmb{\lambda}}_{H,\Delta T} \|_{0,-\frac{1}{2},\sigma} +(\Delta t)^{-\frac{1}{2}}  \|{\bf u}- {\bf u}_{h,\Delta t}\|_{-\frac{1}{2},\frac{1}{2},\sigma, \ast} \ ,\label{est_1}\\
\|{\bf u}-{\bf u}_{h,\Delta t}\|_{-\frac{1}{2},\frac{1}{2},\sigma, \ast} &\lesssim_\sigma \inf \limits_{{\bf v}_{h,\Delta t} \in (X^0_{h,\Gamma_\Sigma}\otimes V^0_{\Delta t})^d}
\|{\bf u}-{\bf v}_{h,\Delta t}\|_{\frac{1}{2},\frac{1}{2},\sigma, \ast}\nonumber \\& \qquad
+ \inf \limits_{\tilde{\pmb{\lambda}}_{H,\Delta T} \in M^+_{H,\Delta T}(\mathcal{F})}\left\{\|\tilde{\pmb{\lambda}}_{H,\Delta T} - \pmb{\lambda}\|_{\frac{1}{2},-\frac{1}{2},\sigma} +\|\tilde{\pmb{\lambda}}_{H,\Delta T} - \pmb{\lambda}_{H,\Delta T}\|_{\frac{1}{2},-\frac{1}{2},\sigma}\right\} \ \label{est_2}. 
\end{align}
\end{theorem}

\begin{proof}[Proof.]
The weak formulation \eqref{eq:MixedProblemsigma} and its discretization \eqref{eq:MixedProblemhsigma} imply that for arbitrary $\tilde{\pmb{\lambda}}_{H,\Delta T} \in M^+_{H,\Delta T}(\mathcal{F})$:
\begin{align} \label{lambdaEq}
\langle \pmb{\lambda}_{H,\Delta T}-\pmb{\tilde{\lambda}}_{H,\Delta T}, \mathbf{v}_{h,\Delta t} \rangle_{\sigma,\Gamma_C,\mathbb{R}^+} & = \langle \mathcal{S} \mathbf{u}_{h,\Delta t}, \mathbf{v}_{h,\Delta t} \rangle_{\sigma,\Gamma_\Sigma,\mathbb{R}^+} - \langle \mathbf{f}, \mathbf{v}_{h,\Delta t}\rangle_{\sigma,\Gamma_\Sigma,\mathbb{R}^+} - \langle \pmb{\tilde{\lambda}}_{H,\Delta T}, \mathbf{v}_{h,\Delta t} \rangle_{\sigma,\Gamma_C,\mathbb{R}^+} \nonumber \\  &= \langle \mathcal{S} \mathbf{u}_{h,\Delta t}, \mathbf{v}_{h,\Delta t} \rangle_{\sigma,\Gamma_\Sigma,\mathbb{R}^+} - \langle \mathcal{S} \mathbf{u}, \mathbf{v}_{h,\Delta t} \rangle_{\sigma,\Gamma_\Sigma,\mathbb{R}^+} +\langle  \pmb{\lambda} , \mathbf{v}_{h,\Delta t} \rangle_{\sigma,\Gamma_C,\mathbb{R}^+}\nonumber \\ &\qquad -\langle \pmb{\tilde{\lambda}}_{H,\Delta T}, \mathbf{v}_{h,\Delta t} \rangle_{\sigma,\Gamma_C,\mathbb{R}^+} \nonumber 
\\ &= \langle \mathcal{S} (\mathbf{u}_{h,\Delta t} -\mathbf{u}),\mathbf{v}_{h,\Delta t} \rangle_{\sigma,\Gamma_\Sigma,\mathbb{R}^+} + \langle \pmb{\lambda} - \pmb{\tilde{\lambda}}_{H,\Delta T}, \mathbf{v}_{h,\Delta t} \rangle_{\sigma,\Gamma_C,\mathbb{R}^+}.
\end{align} 
The inf-sup condition in Theorem \ref{discInfSup} and equation \eqref{lambdaEq} lead to:
\begin{align*}
\alpha\, \| \pmb{\lambda}_{H,\Delta T}-\pmb{\tilde{\lambda}}_{H,\Delta T}\|_{0,-\frac{1}{2},\sigma} &\leq \sup_{ \mathbf{v}_{h,\Delta t}\in (X^0_{h,\Gamma_\Sigma}\otimes V^0_{\Delta t})^d} \frac{ \langle  \pmb{\lambda}_{H,\Delta T}-\pmb{\tilde{\lambda}}_{H,\Delta T}, \mathbf{v}_{h,\Delta t} \rangle_{\sigma,\Gamma_C,\mathbb{R}^+}}{\|\mathbf{v}_{h,\Delta t}\|_{0,\frac{1}{2}, \sigma, \ast} }
\\ & =
\sup_{ \mathbf{v}_{h,\Delta t}\in (X^0_{h,\Gamma_\Sigma}\otimes V^0_{\Delta t})^d} \frac{\langle \mathcal{S} (\mathbf{u}_{h,\Delta t} -\mathbf{u}),\mathbf{v}_{h,\Delta t} \rangle_{\sigma,\Gamma_\Sigma,\mathbb{R}^+} + \langle \pmb{\lambda}-\pmb{\tilde{\lambda}}_{H,\Delta T}, \mathbf{v}_{h,\Delta t} \rangle_{\sigma,\Gamma_C,\mathbb{R}^+}}{\|\mathbf{v}_{h,\Delta t}\|_{0,\frac{1}{2}, \sigma, \ast}}\ .
\end{align*}
For the first term, the continuity of the  duality pairing and an inverse inequality in time \cite[p.~451]{setup} lead to\\
\begin{align*}
|\langle \mathcal{S} (\mathbf{u}_{h,\Delta t} -\mathbf{u}),\mathbf{v}_{h,\Delta t} \rangle_{\sigma,\Gamma_\Sigma,\mathbb{R}^+}| & \leq \| \mathcal{S}(\mathbf{u}_{h,\Delta t}-\mathbf{u})\|_{-\frac{1}{2},-\frac{1}{2},\sigma} \|{\mathbf{v}_{h,\Delta t}}\|_{\frac{1}{2},\frac{1}{2},\sigma, \ast} 
\\ &\lesssim\,
\| \mathbf{u}_{h,\Delta t} -\mathbf{u}\|_{-\frac{1}{2},\frac{1}{2},\sigma,\ast}(\Delta t)^{-\frac{1}{2}} \|\mathbf{v}_{h,\Delta t}\|_{0,\frac{1}{2},\sigma, \ast}\ .\end{align*}
For the second term, similarly we obtain:
\begin{align*}
|\langle \pmb{\lambda}-\pmb{\tilde{\lambda}}_{H,\Delta T}, \mathbf{v}_{h,\Delta t} \rangle_{\sigma,\Gamma_C,\mathbb{R}^+}|& \leq \| \pmb{\lambda}-\pmb{\tilde{\lambda}}_{H,\Delta T} \|_{0,-\frac{1}{2},\sigma} \|\mathbf{v}_{h,\Delta t}\|_{0,\frac{1}{2},\sigma, \ast}
\ .
\end{align*} 
It follows
\begin{align}\label{lambdaapriori}
 \|\pmb{\lambda} -\pmb{\lambda}_{H,\Delta T}\|_{0,-\frac{1}{2}, \sigma} &\leq \inf \limits_{\pmb{\tilde{\lambda}}_{H,\Delta T}}  \left( \|\pmb{\lambda} -\pmb{\tilde{\lambda}}_{H,\Delta T}\|_{0,-\frac{1}{2}, \sigma} + \|\pmb{\lambda}_{H,\Delta T} -\pmb{\tilde{\lambda}}_{H,\Delta T}\|_{0,-\frac{1}{2}, \sigma} \right) \nonumber \\  &\lesssim_\sigma  \inf \limits_{\pmb{\tilde{\lambda}}_{H,\Delta T}} \| \pmb{\lambda}- \pmb{\tilde{\lambda}}_{H,\Delta T} \|_{0,-\frac{1}{2},\sigma} +(\Delta t)^{-\frac{1}{2}}  \| \mathbf{u}_{h,\Delta t}-\mathbf{u}\|_{-\frac{1}{2},\frac{1}{2},\sigma, \ast} \ ,\nonumber
\end{align}
hence the a priori estimate \eqref{est_1} is proved.
To show the second estimate, \eqref{est_2}, first note that from  \eqref{lambdaEq}
\begin{equation*}
\langle\mathcal{S}(\mathbf{u}-\mathbf{u}_{h,\Delta t}),\mathbf{v}_{h,\Delta t}\rangle_{\sigma,\Gamma_\Sigma,\mathbb{R}^+} = \langle \pmb{\lambda} -\pmb{\lambda}_{H,\Delta T}, \mathbf{v}_{h,\Delta t}\rangle_{\sigma,\Gamma_C,\mathbb{R}^+}.
\end{equation*}
By assumption the Poincar\'{e}-Steklov operator is coercive, so that, using also the previous equality,
\begin{align*}
\|\mathbf{u}_{h,\Delta t} - \mathbf{v}_{h,\Delta t}\|_{-\frac{1}{2},\frac{1}{2},\sigma, \ast}^2 &\lesssim_\sigma\langle\mathcal{S}(\mathbf{u}_{h,\Delta t}-\mathbf{v}_{h,\Delta t}), \mathbf{u}_{h,\Delta t}- \mathbf{v}_{h,\Delta t}\rangle_{\sigma,\Gamma_\Sigma,\mathbb{R}^+} \\
& = \langle\mathcal{S}(\mathbf{u}-\mathbf{v}_{h,\Delta t}),\mathbf{u}_{h,\Delta t}- \mathbf{v}_{h,\Delta t}\rangle_{\sigma,\Gamma_\Sigma,\mathbb{R}^+} + \langle\mathcal{S}(\mathbf{u}_{h,\Delta t}-\mathbf{u}), \mathbf{u}_{h,\Delta t}- \mathbf{v}_{h,\Delta t}\rangle_{\sigma,\Gamma_\Sigma,\mathbb{R}^+}\\ 
& = \langle \mathcal{S}(\mathbf{u}-\mathbf{v}_{h,\Delta t}),\mathbf{u}_{h,\Delta t}- \mathbf{v}_{h,\Delta t}\rangle_{\sigma,\Gamma_\Sigma,\mathbb{R}^+}\\& \qquad+\langle \pmb{\tilde{\lambda}}_{H,\Delta T} - \pmb{\lambda}+\pmb{\lambda}_{H,\Delta T}-\pmb{\tilde{\lambda}}_{H,\Delta T}, \mathbf{u}_{h,\Delta t}-  \mathbf{v}_{h,\Delta t}\rangle_{\sigma,\Gamma_C,\mathbb{R}^+} 
\end{align*}
for all $\mathbf{v}_{h,\Delta t}$ and $\tilde{\pmb{\lambda}}_{H,\Delta T}$. Using the mapping properties of $\mathcal{S}$ and the continuity of the duality pairing, we obtain
\begin{align*}
\|\mathbf{u}_{h,\Delta t} - \mathbf{v}_{h,\Delta t}\|_{-\frac{1}{2},\frac{1}{2},\sigma, \ast}^2 &\lesssim \|\mathbf{u}-\mathbf{v}_{h,\Delta t}\|_{\frac{1}{2},\frac{1}{2},\sigma, \ast}\|\mathbf{u}_{h,\Delta t}- \mathbf{v}_{h,\Delta t}\|_{-\frac{1}{2},\frac{1}{2},\sigma, \ast}\\
& \quad +\|\pmb{\tilde{\lambda}}_{H,\Delta T} - \pmb{\lambda}\|_{\frac{1}{2},-\frac{1}{2},\sigma} \| \mathbf{u}_{h,\Delta t}-  \mathbf{v}_{h,\Delta t}\|_{-\frac{1}{2},\frac{1}{2},\sigma, \ast} \\ & \quad + \|\pmb{\tilde{\lambda}}_{H,\Delta T}-\pmb{\lambda}_{H,\Delta T}\|_{\frac{1}{2},-\frac{1}{2},\sigma} \| \mathbf{u}_{h,\Delta t}-  \mathbf{v}_{h,\Delta t}\|_{-\frac{1}{2},\frac{1}{2},\sigma, \ast}\ . 
\end{align*}
We conclude
\begin{align*}
\|\mathbf{u}_{h,\Delta t} - \mathbf{v}_{h,\Delta t}\|_{-\frac{1}{2},\frac{1}{2},\sigma, \ast} &\lesssim_\sigma \|\mathbf{u}-\mathbf{v}_{h,\Delta t}\|_{\frac{1}{2},\frac{1}{2},\sigma, \ast}
+\|\pmb{\tilde{\lambda}}_{H,\Delta T} - \pmb{\lambda}\|_{\frac{1}{2},-\frac{1}{2},\sigma}  + \|\pmb{\tilde{\lambda}}_{H,\Delta T} - \pmb{\lambda}_{H,\Delta T}\|_{\frac{1}{2},-\frac{1}{2},\sigma}\ . 
\end{align*}
Finally, we obtain with the triangle inequality 
\begin{align*}
\|\mathbf{u}-\mathbf{u}_{h,\Delta t}\|_{-\frac{1}{2},\frac{1}{2},\sigma, \ast} &\lesssim_\sigma \|\mathbf{u}-\mathbf{v}_{h,\Delta t}\|_{\frac{1}{2},\frac{1}{2},\sigma, \ast}
+\|\pmb{\tilde{\lambda}}_{H,\Delta T} - \pmb{\lambda}\|_{\frac{1}{2},-\frac{1}{2},\sigma} +\|\pmb{\tilde{\lambda}}_{H,\Delta T} - \pmb{\lambda}_{H,\Delta T}\|_{\frac{1}{2},-\frac{1}{2},\sigma} \,. 
\end{align*}
The a priori estimate \eqref{est_2} follows.
\end{proof}

\subsection{Coulomb friction} \label{sec:coulomb}

{For Coulomb friction, the friction coefficient $\mathcal{F}_c$ is a sufficiently regular function on $\Gamma_C$. In our formulation only the set of Lagrange multipliers needs to be adapted, as the friction threshold $\mathcal{F}_c\lambda_\perp$ now depends on the normal force $\lambda_\perp$: 
 \begin{align}
 M^+(\mathcal{F}_c\lambda_\perp)&:=\left\lbrace \pmb{\mu} \in H^{1/2}([0,T],\tilde{H}^{-1/2}(\Gamma_C))^d: \right.\nonumber \\& \left.\left\langle \pmb{\mu},\mathbf{v}\right\rangle_{0,\Gamma_C,(0,T]} \leq \left\langle\mathcal{F}_c \lambda_\perp, \|\textbf{v}_\parallel\|\right\rangle_{0,\Gamma_C,(0,T]},\, \forall \ \mathbf{v} \in H^{-1/2}([0,T],{H}^{1/2}(\Gamma_\Sigma))^d, v_\perp\leq 0 
\right\rbrace \label{eq:CoulombModLagrangeSet}.
 \end{align}
Its discretization is given by 
\begin{align}
M^+_{H,\Delta T}(\mathcal{F}_c\lambda_{\perp, H, \Delta T})&:=\Big\lbrace \pmb{\mu}_{H,\Delta T} \in (X^{-1}_{H,\Gamma_C}\otimes V^{-1}_{\Delta T})^d: \nonumber\\ &\qquad  \mu_{\perp,H,\Delta T} \geq 0 {\text{ and } \|\pmb{\mu}_{\parallel,H,\Delta T}\|\leq \mathcal{F}_c \lambda_{\perp, H, \Delta T}} \text{ on } (0,T] \times\Gamma_C \Big\rbrace.
\end{align}
The resulting continuous and discrete problems are, respectively, as in \eqref{eq:MixedProblem} with $M^+(\mathcal{F}_c\lambda_\perp)$ substituting $M^+(\mathcal{F})$ and as in \eqref{eq:MixedProblemh} with $M^+_{H,\Delta T}(\mathcal{F}_c\lambda_{\perp, H, \Delta T})$ substituting $M^+_{H,\Delta T}(\mathcal{F})$.\\

Note that generally $\mathcal{F}_c \lambda_\perp \not \in L^\infty(\Gamma_C)$ and that, unlike for Tresca friction, the discretization is no longer conforming when $M^+_{H,\Delta T}(\mathcal{F}_c\lambda_{\perp, H, \Delta T}) \not \subset M^+(\mathcal{F}_c\lambda_{\perp})$. Furthermore, because the sets $M^+(\mathcal{F}_c\lambda_{\perp})$ and $M^+_{H,\Delta T}(\mathcal{F}_c\lambda_{\perp, H, \Delta T})$ depend on the solution of the continuous problem, respectively its discretization, 
standard assumptions in the analysis of variational inequalities are no longer satisfied. Consequently, little is  rigorously known for the Coulomb friction law even for stationary contact problems.} 

\section{Two-body frictional contact}\label{Bilateral}

To study the bilateral frictional contact problem,
we need to introduce the Poincar\'{e}-Steklov operator 
\eqref{operator_S} for each of the two domains $\Omega_j,\,j=1,2$ involved. We will denote these operators by ${\cal S}_j,\,j=1,2$. 

Referring to the problem \eqref{strongformPDE2b} - \eqref{sfdn2},
the boundary $\Gamma$ is decomposed into parts where the displacement ($\Gamma_D$) or the traction ($\Gamma_N$) are prescribed, as well as the interface $\Gamma_\Sigma'=\Gamma_I\cup \Gamma_C$ of the two bodies. At the interface, transmission ($\Gamma_I$)  or frictional contact conditions ($\Gamma_C$) are imposed. For simplicity of the already involved notation, we fix $\mathbf{g} = \mathbf{0}$ on $(0,T]\times  \Gamma_\Sigma'$ and assume that $\Gamma_N = \emptyset$, so that $\Gamma=\Gamma_D \cup \Gamma_\Sigma$ and the boundary portion  $\Gamma_\Sigma=\Gamma_N\cup \Gamma_\Sigma'$ reduces to $\Gamma_\Sigma=\Gamma_\Sigma'$.\\
Now, we set $\mathbf{u} = \mathbf{u}_1|_{(0,T]\times \Gamma_\Sigma}$,
$\widetilde{\mathbf{u}} = (\mathbf{u}_1-\mathbf{u}_2)|_{\Gamma_\Sigma}$ and $\pmb{\lambda} = \mathbf{p}_{1}|_{ \Gamma_C}=(\mathbf{f}-\mathbf{p}_{2})|_{\Gamma_C}$, so that the contact conditions \eqref{fullcontactbc12b}, \eqref{fullcontactbc22b} on $(0,T]\times  \Gamma_C$ become
\begin{align}
 \widetilde{u}_\perp \geq 0
 \ ,\ \lambda_{\perp}\geq 0,\ \widetilde{u}_\perp \lambda_{\perp} =0,\label{fullcontactbc12blambda}\\
   \| \pmb{\lambda}_{\parallel} \|  \leq \mathcal{F},\ \pmb{\lambda}_{\parallel} \cdot\dot{\widetilde{\mathbf{u}}}+\mathcal{F}\| \dot{\widetilde{\mathbf{u}}}\| =0. \label{fullcontactbc22blambda}
\end{align}
The frictional contact problem \eqref{strongformPDE2b} - \eqref{sfdn2} can be formulated as a variational inequality for $(\mathbf{u},\widetilde{\mathbf{u}})$ in the half-space 
\begin{align*}\mathcal{C}&:=\left\lbrace (\textbf{v} : (0,T]\times  \Gamma \to \mathbb{R}^d,\widetilde{\mathbf{v}} : (0,T]\times  \Gamma_\Sigma \to \mathbb{R}^d):  \mathbf{v} = \mathbf{0} \; \text{~a.e.~on~} (0,T]\times  \Gamma_D,\  \right.\\
& \left. \qquad \qquad 
\widetilde{v}_{\perp} \geq 0\; \text{~a.e.~on~} (0,T]\times  \Gamma_C\ ,
\widetilde{\mathbf{v}} =  \mathbf{0} \; \text{~a.e.~on~} (0,T]\times  \Gamma_I \right\rbrace.\end{align*}
Using the friction functional given in \eqref{friction functional},
the variational inequality reads:\\

\noindent {\textit{find} $(\mathbf{u},\widetilde{\mathbf{u}}) \in \mathcal{C}$ \textit{such that for all}  $(\textbf{v},\widetilde{\mathbf{v}})  \in \mathcal{C}$
\begin{align} \label{eq:VarIneq2b}
&\langle {\mathcal{S}_1 {\bf u}}, \partial_{t,\parallel}({\bf v}-{\bf u})\rangle_{0,\Gamma_\Sigma,(0,T]} + \langle {\mathcal{S}_2 ({\bf u}-\widetilde{\bf u})}, \partial_{t,\parallel}({\bf v}-\widetilde{\bf v}-({\bf u}-\widetilde{\bf u}))\rangle_{0,\Gamma_\Sigma,(0,T]} + j(\widetilde{\bf v})-j(\widetilde{\bf u})\nonumber\\ &\qquad\geq \langle {\bf f}, \partial_{t,\parallel}({\bf v}-\widetilde{\bf v}-({\bf u}-\widetilde{\bf u}))\rangle_{0,\Gamma_\Sigma,(0,T]}.
\end{align}}

{The precise statement of the mixed formulation corresponding to the variational inequality \eqref{eq:VarIneq2b} reads as follows:} \\

\noindent {\textit{find} $(\mathbf{u},\widetilde{\mathbf{u}},{\pmb{\lambda}}) \in H^{1/2}([0,T],\tilde{H}^{1/2}(\Gamma_\Sigma))^d \times H^{1/2}([0,T],\tilde{H}^{1/2}(\Gamma_C))^d \times M^+(\mathcal{F})$ \textit{such that for all} $(\textbf{v},\widetilde{\mathbf{v}},{\pmb{\mu}}) \in H^{1/2}([0,T],\tilde{H}^{1/2}(\Gamma_\Sigma))^d \times H^{1/2}([0,T],\tilde{H}^{1/2}(\Gamma_C))^d \times M^+(\mathcal{F})$
\begin{subequations} \label{eq:MixedProblem2b}
 \begin{alignat}{2}
\left\langle \mathcal{S}_1\textbf{u},\textbf{v} \right\rangle_{0,\Gamma_\Sigma,(0,T]} + \left\langle \mathcal{S}_2(\textbf{u}-\widetilde{\mathbf{u}}),\textbf{v}-\widetilde{\mathbf{v}} \right\rangle_{0,\Gamma_\Sigma,(0,T]} - \left\langle \pmb{\lambda},\widetilde{\mathbf{v}} \right\rangle_{0,\Gamma_C,(0,T]}&= \left\langle \textbf{f},\textbf{v} -\widetilde{\mathbf{v}}\right\rangle_{0,\Gamma_\Sigma,(0,T]}  \label{eq:WeakMixedVarEq2b}\\
\left\langle \pmb{\mu} -\pmb{\lambda},\partial_{t,\parallel}\widetilde{\mathbf{u}} \right\rangle_{0,\Gamma_C,(0,T]} & \geq 0 .\label{eq:ContContactConstraints2b}
\end{alignat}
\end{subequations}}

\begin{theorem}\label{equivtheorem:2b}
{The variational inequality \eqref{eq:VarIneq2b}, the  mixed formulation \eqref{eq:MixedProblem2b} and the boundary value problem 
\eqref{strongformPDE2b} -  \eqref{sfdn2} are equivalent.}
\end{theorem}
\begin{proof}The displacement in $(0,T]\times  \Gamma_D$ is imposed in all the formulations, so that it suffices to consider the boundary conditions in $(0,T]\times  \Gamma_\Sigma$.

{We first show that the variational inequality \eqref{eq:VarIneq2b} implies the boundary value problem 
\eqref{strongformPDE2b} -  
\eqref{sfdn2}. As in Proposition \ref{prop:varineq}, the converse assertion, i.e. that the boundary value problem implies the variational inequality, is easier. \\
The solutions $\mathbf{u}_1, \mathbf{u}_2$ to the differential equations \eqref{strongformPDE2b} in $\Omega_1$ and $\Omega_2$ are recovered from $(\mathbf{u},\widetilde{\mathbf{u}})$ using the representation formula, as usual for boundary element methods. The transmission condition  \eqref{prob:strong_formulation2b_2} follows from $\widetilde{\mathbf{u}}=\mathbf{0}$ on $(0,T]\times  \Gamma_I$.}

{We now choose $(\textbf{v},\widetilde{\mathbf{v}})=(\mathbf{u}\pm\mathbf{w},\widetilde{\mathbf{u}})$ in \eqref{eq:VarIneq2b}. We obtain
\begin{align*}
\pm \langle {\mathcal{S}_1 {\bf u}}, \partial_{t,\parallel}{\bf w}\rangle_{0,\Gamma_\Sigma,(0,T]} \pm \langle {\mathcal{S}_2 ({\bf u}-\widetilde{\mathbf{u}})}, \partial_{t,\parallel}{\bf w}\rangle_{0,\Gamma_\Sigma,(0,T]} \geq \pm\langle {\bf f}, \partial_{t,\parallel}{\bf w}\rangle_{0,\Gamma_\Sigma,(0,T]}, 
\end{align*}
and hence
$$\langle {\mathcal{S}_1 {\bf u}}, \partial_{t,\parallel}{\bf w}\rangle_{0,\Gamma_\Sigma,(0,T]} + \langle {\mathcal{S}_2 ({\bf u}-\widetilde{\mathbf{u}})}, \partial_{t,\parallel}{\bf w}\rangle_{0,\Gamma_\Sigma,(0,T]} = \langle {\bf f}, \partial_{t,\parallel}{\bf w}\rangle_{0,\Gamma_\Sigma,(0,T]}.$$
Because $\bf w$ is arbitrary, we conclude the transmission conditions $\dot{\bf p}_{1,\parallel}+\dot{\bf p}_{2,\parallel}=\partial_t\mathcal{S}_1 {\bf u}_\parallel + \partial_t\mathcal{S}_2 ({\bf u}-\widetilde{\bf u})_\parallel=\partial_t{\bf f}_\parallel$ and ${p}_{1,\perp}+{p}_{2,\perp}=\mathcal{S}_1 {\bf u}_\perp + \mathcal{S}_2 ({\bf u}-\widetilde{\bf u})_\perp=f_\perp$. 
Using $\mathbf{p}_{1,\parallel}+\mathbf{p}_{2,\parallel} = \mathbf{0}$ at $t=0$, $\mathbf{p}_{1,\parallel}+\mathbf{p}_{2,\parallel} = \mathbf{f}$ for all $t>0$. In particular \eqref{prob:strong_formulation2b_1}, the last equation in \eqref{fullcontactbc12b} and the first equation in \eqref{fullcontactbc22b} follow.}

{We now show that $p_{1,\perp}\geq 0$ and $p_{2,\perp}\leq  f_\perp$. To do so, we choose $(\textbf{v},\widetilde{\mathbf{v}})=(\mathbf{u},\widetilde{\mathbf{u}}+\widetilde{\mathbf{w}})$ in \eqref{eq:VarIneq2b}, with $\widetilde{\bf w}_\parallel = \mathbf{0}$ and $\widetilde{w}_\perp \geq -\tilde{u}_\perp$.
Since $j(\widetilde{\mathbf{u}}+ \widetilde{\bf w})-j(\widetilde{\mathbf{u}}) \leq j(\widetilde{\bf w})=
0$, the variational inequality \eqref{eq:VarIneq2b} simplifies to\\
$$\langle {\mathcal{S}_2 ({\bf u}-\widetilde{\mathbf{u}})}_\perp, -\widetilde{w}_\perp\rangle_{0,\Gamma_\Sigma,(0,T]} \geq \langle f_\perp, -\widetilde{w}_\perp\rangle_{0,\Gamma_\Sigma,(0,T]}.$$
Because $\widetilde{w}_\perp \geq 0$ is arbitrary, we conclude $p_{2,\perp} = {\mathcal{S}_2 ({\bf u}-\widetilde{\mathbf{u}})}_\perp \leq f_\perp$ and $p_{1,\perp} = f_\perp -p_{2,\perp}\geq 0$, as required in \eqref{fullcontactbc12b}.}

{To show the remaining assertion in \eqref{fullcontactbc12b}, we choose $\mathbf{v} = \mathbf{u}$ in \eqref{eq:VarIneq2b}. This implies
$$ \langle {\mathcal{S}_2 ({\bf u}-\widetilde{\mathbf{u}})}, \partial_{t,\parallel}(-\widetilde{\mathbf{v}}+\widetilde{\mathbf{u}})\rangle_{0,\Gamma_\Sigma,(0,T]} +j(\widetilde{\mathbf{v}})-j(\widetilde{\mathbf{u}})\geq \langle {\bf f}, \partial_{t,\parallel}(-\widetilde{\mathbf{v}}+\widetilde{\mathbf{u}})\rangle_{0,\Gamma_\Sigma,(0,T]}.$$
If $\widetilde{\mathbf{u}}_\perp>0$ in a neighborhood of $(t,{\bf x}) \in (0,T]\times  \Gamma_C$, we let $\widetilde{\mathbf{v}} = \widetilde{\mathbf{u}}\pm \widetilde{\bf w}$ with $\widetilde{{v}}_\perp \geq 0$ and $\widetilde{\bf w}_\parallel = \mathbf{0}$. Using, as above, that $j(\widetilde{\mathbf{u}}\pm \widetilde{\bf w})-j(\widetilde{\mathbf{u}}) \leq 0$, we obtain
$$ \mp\langle {\mathcal{S}_2 ({\bf u}-\widetilde{\mathbf{u}})}_\perp, \widetilde{w}_\perp\rangle_{0,\Gamma_\Sigma,(0,T]} \geq \mp\langle f_\perp, \widetilde{w}_\perp\rangle_{0,\Gamma_\Sigma,(0,T]}$$
and therefore
$$\langle {\mathcal{S}_2 ({\bf u}-\widetilde{\mathbf{u}})}_\perp, \widetilde{w}_\perp\rangle_{0,\Gamma_\Sigma,(0,T]} =\langle f_\perp, \widetilde{w}_\perp\rangle_{0,\Gamma_\Sigma,(0,T]}.$$
We conclude that ${\mathcal{S}_2 ({\bf u}-\widetilde{\mathbf{u}})}({\bf x})_\perp = f({\bf x})_\perp$. Therefore, $p_{2,\perp} - f_\perp = 0$ and $p_{1,\perp} = f_\perp - p_{2,\perp} = 0$ whenever $u_{1,\perp}-u_{2,\perp}> g_{\perp}=0$, as required in \eqref{fullcontactbc12b}.}

{It remains to analyze the conditions satisfied by the tangential components to obtain the friction condition \eqref{fullcontactbc22b}. To show that $\| \mathbf{p}_{1,\parallel} \|  \leq \mathcal{F}$, we choose $(\textbf{v},\widetilde{\mathbf{v}})=(\mathbf{u},\widetilde{\mathbf{u}}\pm\widetilde{\mathbf{w}})$ in \eqref{eq:VarIneq2b}, with $\widetilde{w}_\perp = 0$. With this choice and using $j(\widetilde{\mathbf{u}}\pm \widetilde{\bf w})-j(\widetilde{\mathbf{u}}) \leq j(\widetilde{\bf w})
$, the variational inequality \eqref{eq:VarIneq2b} becomes 
\begin{equation}\label{auxeq}
\mp\langle {\mathcal{S}_2 ({\bf u}-\widetilde{\mathbf{u}})}_\parallel, \dot{\widetilde{\mathbf{w}}}_\parallel\rangle_{0,\Gamma_\Sigma,(0,T]}
+ \langle  \mathcal{F},\|\dot{\widetilde{\bf w}}_\parallel\| \rangle_{0,\Gamma_C,(0,T]}
\geq \mp\langle \mathbf{f}_\parallel, \dot{\widetilde{\mathbf{w}}}_\parallel\rangle_{0,\Gamma_\Sigma,(0,T]}.
\end{equation}
As $\widetilde{\bf w}_\parallel$ is arbitrary, we conclude $\|{\mathcal{S}_2 ({\bf u}-\widetilde{\mathbf{u}})}_\parallel-  \mathbf{f}_{\parallel}\|=\| \mathbf{p}_{2,\parallel} -  \mathbf{f}_{\parallel} \| \leq  \mathcal{F}$. Hence $\| \mathbf{p}_{1,\parallel} \| = \| \mathbf{p}_{2,\parallel} -  \mathbf{f}_{\parallel} \| \leq  \mathcal{F}$, as asserted in \eqref{fullcontactbc22b}.}

{We conclude by showing the final assertion in \eqref{fullcontactbc22b}, that $$\mathbf{p}_{1,\parallel} \cdot\widetilde{\mathbf{u}}_\parallel+\mathcal{F}\| \dot{\widetilde{\mathbf{u}}}\| =0.$$ This follows by choosing $\widetilde{\bf w}_\parallel = \widetilde{\mathbf{u}}_\parallel$ in \eqref{auxeq} and using ${\bf p}_{1,\parallel}={\bf f}_{{\parallel}}-{\bf p}_{2,\parallel}$.}\\

We now demonstrate the equivalence of the mixed formulation \eqref{eq:MixedProblem2b} and the boundary value problem 
\eqref{strongformPDE2b} -  \eqref{sfdn2}. As above, we focus on the derivation of the boundary value problem from the mixed formulation. The proof is similar to the proof of Theorem \ref{equivtheorem:unilateral} for the unilateral contact problem,  

First, equation \eqref{eq:WeakMixedVarEq2b} is readily seen to be equivalent to the definition $\pmb{\lambda} = \mathbf{p}_{1}|_{\Gamma_C}=(\mathbf{f}-\mathbf{p}_{2})|_{\Gamma_C}$ and the transmission condition $\mathbf{p}_{1}+\mathbf{p}_{2}=\mathbf{f}$. \\
Further, ${\pmb{\lambda}}\in M^+(\mathcal{F})$ assures the inequalities $\lambda_\perp \geq 0$ for the perpendicular component and $\|\pmb{\lambda}_\parallel\|\leq \mathcal{F}$ for the parallel component, a.e. on $(0,T]\times  \Gamma_C$. To see that $\widetilde{u}_\perp\geq 0$, consider
 $\pmb{\mu} =  \pmb{\lambda}+ \mathbf{w}\in M^+(\mathcal{F})$  with $w_\perp\geq 0$ in \eqref{eq:ContContactConstraints2b}:
$$\left\langle w_\perp,\widetilde{u}_\perp \right\rangle_{0,\Gamma_C,(0,T]}  \geq 0. $$
As $w_\perp\geq 0$ is arbitrary, $\widetilde{u}_\perp\geq 0$ follows almost everywhere on $(0,T]\times  \Gamma_C$.

Finally, choosing  
$\pmb{\mu} = \pmb{\lambda}_\parallel$, respectively $\pmb{\mu} =  \pmb{\lambda}_\parallel-2\lambda_\perp \mathbf{n}$ in \eqref{eq:ContContactConstraints2b}, we obtain
$$\pm \left\langle \lambda_\perp,\widetilde{u}_\perp \right\rangle_{0,\Gamma_C,(0,T]}  \geq 0, $$
and hence $\left\langle \lambda_\perp,\widetilde{u}_\perp \right\rangle_{0,\Gamma_C,(0,T]} =0.$ From the previous paragraph, $\lambda_\perp \geq 0,\widetilde{u}_\perp \geq 0$, so that $\lambda_\perp \widetilde{u}_\perp = 0$ holds a.e. on $(0,T]\times  \Gamma_C$. 

This completes the proof of the nonpenetration boundary conditions in the normal component, in the form of \eqref{fullcontactbc12blambda}, and it only remains to verify $\pmb{\lambda}_{\parallel} \cdot\dot{\widetilde{\mathbf{u}}}+\mathcal{F}\| \dot{\widetilde{\mathbf{u}}}\| =0$ in \eqref{fullcontactbc22blambda}.
For this we may assume $\dot{\widetilde{\mathbf{u}}}\neq \mathbf{0}$.  Choose $\pmb{\mu} \in M^+(\mathcal{F})$ with $\mu_\perp =\lambda_\perp$ and $\pmb{\mu}_\parallel = - \mathcal{F} \frac{\dot{\widetilde{\mathbf{u}}}}{\|\dot{\widetilde{\mathbf{u}}}\|}$, respectively $\pmb{\mu}_\parallel = 2 \pmb{\lambda}_\parallel+ \mathcal{F} \frac{\dot{\widetilde{\mathbf{u}}}}{\|\dot{\widetilde{\mathbf{u}}}\|}$, in \eqref{eq:ContContactConstraints2b}, to obtain $$0\leq \left\langle \pmb{\mu}_\parallel -\pmb{\lambda}_\parallel,\dot{\widetilde{\mathbf{u}}}_\parallel \right\rangle_{0,\Gamma_C,(0,T]} = \textstyle{\mp\left\langle \mathcal{F} \frac{\dot{\widetilde{\mathbf{u}}}_\parallel}{\|\dot{\widetilde{\mathbf{u}}}_\parallel\|} +\textbf{p}_{{1,}\parallel},\dot{\widetilde{\mathbf{u}}}_\parallel \right\rangle_{0,\Gamma_C,(0,T]}}.$$ This shows $0=\left\langle \mathcal{F} \frac{\dot{\widetilde{\mathbf{u}}}_\parallel}{\|\dot{\widetilde{\mathbf{u}}}_\parallel\|} +\textbf{p}_{{1,}\parallel},\dot{\widetilde{\mathbf{u}}}_\parallel \right\rangle_{0,\Gamma_C,(0,T]}$, and therefore $0=\mathcal{F}\|\dot{\widetilde{\mathbf{u}}}_\parallel\|+\mathbf{p}_{{1,}\parallel}\cdot \dot{\widetilde{\mathbf{u}}}_\parallel$ a.e. on $(0,T]\times  \Gamma_C$. 
\end{proof}

The mixed problem \eqref{eq:MixedProblem2b} is discretized as follows:\\

\noindent {\textit{find} $(\mathbf{u}_{h,\Delta t},\widetilde{\mathbf{u}}_{h,\Delta t},\pmb{\lambda}_{H,\Delta T}) \in (X^0_{h,\Gamma_\Sigma}\otimes V^0_{\Delta t})^d \times (X^0_{h,\Gamma_C}\otimes V^0_{\Delta t})^d \times M^+_{H,\Delta T}(\mathcal{F})$  \textit{such that for all}\\ $(\mathbf{v}_{h,\Delta t},\widetilde{\mathbf{v}}_{h,\Delta t},\pmb{\mu}_{H,\Delta T}) \in (X^0_{h,\Gamma_\Sigma}\otimes V^0_{\Delta t})^d \times (X^0_{h,\Gamma_C}\otimes V^0_{\Delta t})^d \times M^+_{H,\Delta T}(\mathcal{F})$
\begin{subequations} \label{eq:MixedProblem2bh}
 \begin{alignat}{2}
\left\langle \mathcal{S}_{1,h,\Delta t}\mathbf{u}_{h,\Delta t},\mathbf{v}_{h,\Delta t} \right\rangle_{0,\Gamma_\Sigma,(0,T]} + \left\langle \mathcal{S}_{2,h,\Delta t}(\mathbf{u}_{h,\Delta t}-\widetilde{\mathbf{u}}_{h,\Delta t}),\mathbf{v}_{h,\Delta t}-\widetilde{\mathbf{v}}_{h,\Delta t} \right\rangle_{0,\Gamma_\Sigma,(0,T]}\hspace*{-1.7cm}&\nonumber\\ - \left\langle \pmb{\lambda}_{H,\Delta T}, \widetilde{\mathbf{v}}_{h,\Delta t} \right\rangle_{0,\Gamma_C,(0,T]}&= \left\langle \textbf{f},\mathbf{v}_{h,\Delta t} -\widetilde{\mathbf{v}}_{h,\Delta t}\right\rangle_{0,\Gamma_\Sigma,(0,T]}  \label{eq:WeakMixedVarEq2bh}\\
\left\langle \pmb{\mu}_{H,\Delta T} -\pmb{\lambda}_{H,\Delta T},\partial_{t,\parallel}\widetilde{\mathbf{u}}_{h,\Delta t} \right\rangle_{0,\Gamma_C,(0,T]} & \geq 0 .\label{eq:ContContactConstraints2bh}
\end{alignat}
\end{subequations}}

\vspace{0.2in}
{\bf Remark.} The bilateral frictional contact problem has been introduced with respect to the case of  Tresca friction. In case of Coulomb friction, the above notation related to $\mathcal{F}$, $M^+(\mathcal{F})$, $M^+_{H,\Delta T}(\mathcal{F})$  has to be substituted by $\mathcal{F}_c \lambda_\perp$, 
 $M^+(\mathcal{F}_c\lambda_\perp)$, $M^+_{H,\Delta T}(\mathcal{F}_c\lambda_{\perp, H, \Delta T})$, respectively.

\section{Algorithmic details}\label{sec:psh}

The algebraic reformulation of the proposed approach is based on the discretization of the Poincar\'{e}-Steklov operator, which will be written in terms of space-time boundary integral operators as detailed below.

\subsection{Boundary integral representation of ${\cal S}$}
\label{birs}

The Poincar\'{e}-Steklov operator \eqref{operator_S} can be implemented by means of classical boundary integral operators, thereby allowing efficient computation. We consider the time-dependent single layer and double layer potential operators $V$ and $K$, that act onto the generic fields $\pmb{ \phi}$ and $\pmb{\psi}$, for $(t,{\bf x}) \in (0,T]\times  \Omega$, as follows:
$$[V \pmb{\phi}](t,\mathbf{x}) = \int_0^t\int_\Gamma G(t,\tau;\mathbf{x},\mathbf{y})\: {\pmb{\phi}}(\tau,\mathbf{y}) d\Gamma_{\mathbf{y}} d\tau,\:\:
[K \pmb{\psi}](t,\mathbf{ x}) =\int_0^t \int_\Gamma  \left[\sigma_{\mathbf{y}}\left(G\right)^\top (t,\tau;\mathbf{x},\mathbf{y})\mathbf{n}_{\mathbf{y}}\right] \pmb{ \psi}(\tau,\mathbf{y})  d\Gamma_{\mathbf{y}} d\tau,
$$
where $G$ is the fundamental solution to \eqref{navierlame}, whose explicit expression for $d=2,3$ can be found in \cite{ourpaper2}.
The subscript $\mathbf{y}$ applied to the stress tensor $\sigma$ denotes the variable for the application of the spatial derivative, while for the vector $\mathbf{n}$ it declares the point of $\Gamma$ where we are considering the normal direction.

The unknown displacement $\textbf{u}$ can be represented by the representation formula 
\begin{equation}\label{representation formula}
\textbf{u}=V\textbf{p}-K\textbf{u},\quad \textrm{in} \; (0,T]\times  \Omega.
\end{equation}
Letting $\textbf{x}\in\Omega\rightarrow \textbf{x}\in \Gamma$ in \eqref{representation formula}, we deduce the classical boundary integral equation 
\begin{equation}\label{first_BIE}
\frac{1}{2}\textbf{u}=\mathcal{V}\textbf{p}-\mathcal{K}\textbf{u},\quad \textrm{in}\; (0,T]\times  \Gamma,
\end{equation}
taking into account the free term $\frac{1}{2}\textbf{u}$ generated by the operator $K$ pursuant the limiting process.
Moreover, we need also to introduce the adjoint double layer operator $\mathcal{K}^\star$ and the hypersingular integral operator $\mathcal{W}$ for $(\textbf{x},t)\in (0,T\times\Gamma$:
\begin{align*}
[\mathcal{K}^\star \pmb{\phi}](t,\mathbf{x})& = \int_0^t\int_\Gamma \left[\sigma_{\mathbf{x}} \left(G\right)(t,\tau;\mathbf{x},\mathbf{y})\mathbf{n}_{\mathbf{x}}\right]  {\pmb{\phi}}(\tau,\mathbf{y}) d\Gamma_{\mathbf{y}}\ d\tau,\\
[\mathcal{W} \pmb{\psi}](t,\mathbf{ x}) &=\int_0^t \int_\Gamma  \left[\sigma_\textbf{x}\left(\sigma_{\mathbf{y}}\left(G\right)^\top (t,\tau;\mathbf{x},\mathbf{y})\mathbf{n}_{\mathbf{y}}\right)\textbf{n}_{\textbf{x}}\right] \pmb{ \psi}(\tau,\mathbf{y})  d\Gamma_{\mathbf{y}}\ d\tau\,,
\end{align*}
involved in the boundary integral equation
\begin{equation}\label{second_BIE}
\frac{1}{2}\textbf{p}=\mathcal{K}^*\textbf{p}-\mathcal{W}\textbf{u},\quad \textrm{in}\; (0,T]\times  \Gamma.
\end{equation}
The Poincar\'{e}-Steklov operator $\mathcal{S}$ 
defined in \eqref{operator_S}
can be expressed by two equivalent forms, based on the previously introduced integral operators, namely: the non-symmetric formulation
\begin{equation}\label{definition_S}
\mathcal{S}=\mathcal{V}^{-1}\left(\mathcal{K}+\frac{1}{2}\right)
\end{equation}
and the symmetric formulation
\begin{equation}\label{definition_S_symm}
\mathcal{S}=\left(\mathcal{K}^*+\frac{1}{2}\right)\mathcal{V}^{-1}\left(\mathcal{K}+\frac{1}{2}\right)-\mathcal{W} \,.
\end{equation}

{{\bf Remark.} 
While the symmetric and non-symmetric formulations are equivalent for the continuous problem, the resulting discrete formulations are no longer equivalent and have different advantages and disadvantages. The symmetric formulation relates to the elastic energy and is much better understood mathematically. Meanwhile, the non-symmetric formulation is easier to implement,  computationally cheaper and therefore preferred by engineers. Both formulations are studied in the numerical experiments in Section \ref{sec;numres}.}

\subsection{Implementation of ${\cal S}_{h,\Delta t}$ and related linear system}
\label{subPS}

From now on, to simplify the notation and because the numerical results reported in Section \ref{sec;numres}
are related to two-dimensional elastodynamics, we will fix $d = 2$.\\
At every iteration of the Uzawa algorithm, as described in Algorithm \ref{alg1}, we need to solve the weak equation
\eqref{eq:WeakMixedVarEqh}, involving $\mathcal{S}_{h,\Delta t}$. For the implementation of  the Poincar\'{e}-Steklov operator,
both the symmetric and the non-symmetric formulations  
will be used. While they are equivalent for the continuous problem, the resulting discretizations will lead to slightly different numerical results. 
Moreover, for the stability of the time-stepping scheme, we consider an energetic weak formulation which involves time derivatives of test functions.\\
In general, for a given right-hand side $\widetilde{\textbf{f}}\in H^{1/2}\left([0,T],H^{-1/2}(\Gamma_\Sigma)\right)^2$ we have to solve a weak problem of the form:\\
\textit{find} $\textbf{u}_{h,\Delta t} \in (X^0_{h,\Gamma_\Sigma}\otimes V^0_{\Delta t})^2$ \textit{such that} 
\begin{equation}\label{eq:Scontinuous}
\langle \mathcal{S} _{h,\Delta t}\textbf{u}_{h,\Delta t},\dot{\textbf{v}}_{h,\Delta t}\rangle_{0,\Gamma_\Sigma,(0,T]} = \langle \widetilde{\textbf{f}},\dot{\textbf{v}}_{h,\Delta t} \rangle_{0,\Gamma_\Sigma,(0,T]}, \quad \quad \forall\,\textbf{v}_{h,\Delta t} \in (X^0_{h,\Gamma_\Sigma}\otimes V^0_{\Delta t})^2.
\end{equation}

We first consider the definition of the Poincar\'{e}-Steklov operator in its symmetric form
: because of its dependence on the inverse operator ${\cal V}^{-1}$, to handle the equation at hand we need to define the dummy variable $\pmb{\psi}_{h,\Delta t}:=\mathcal{V}^{-1}\left(\mathcal{K}+\frac{1}{2}\right)\textbf{u}_{h,\Delta t}$ , belonging to the space $(X^{-1}_{h,\Gamma}\otimes V^{-1}_{\Delta t})^2$.

This translates in solving at every step of the Uzawa algorithm the following system of weak boundary integral equations: \\ 

\textit{find} $\textbf{u}_{h,\Delta t}\in (X^0_{h,\Gamma_\Sigma}\otimes V^0_{\Delta t})^2$ \textit{and} $\pmb{\psi}_{h,\Delta t}\in (X^{-1}_{h,\Gamma}\otimes V^{-1}_{\Delta t})^2$ \textit{such that}
\begin{equation}\label{energetic_formh}
\begin{array}{l}
-\left\langle \mathcal{V}\pmb{\psi}_{h,\Delta t},\dot{\pmb{\eta}}_{h,\Delta t}\right\rangle_{0,\Gamma,(0,T]} +\left\langle\left(\mathcal{K}+\frac{1}{2}\right)\textbf{u}_{h,\Delta t},\dot{\pmb{\eta}}_{h,\Delta t}\right\rangle_{0,\Gamma,(0,T]}=0, \quad\quad\quad\quad\quad\quad\quad\forall\, \pmb{\eta}_{h,\Delta t}\in(X^{-1}_{h,\Gamma}\otimes V^{-1}_{\Delta t})^2\\[4pt]
\left\langle\left(\mathcal{K}^*+\frac{1}{2}\right)\pmb{\psi}_{h,\Delta t}, \dot{\textbf{v}}_{h,\Delta t}\right\rangle_{0,\Gamma_\Sigma,(0,T]}-\left\langle \mathcal{W}\textbf{u}_{h,\Delta t}, \dot{\textbf{v}}_{h,\Delta t}\right\rangle_{0,\Gamma_\Sigma,(0,T]}=\left\langle\widetilde{\textbf{f}},\dot{\textbf{v}}_{h,\Delta t}\right\rangle_{0,\Gamma_\Sigma0,T]}, \forall\,\textbf{v}_{h,\Delta t}\in (X^0_{h,\Gamma_\Sigma}\otimes V^0_{\Delta t})^2.
\end{array}
\end{equation}

Let us now introduce the sets $\left\lbrace w^{(\pmb{\psi})}_m \right\rbrace_{m=1}^{M^{(\pmb{\psi})}_h}$ and $\left\lbrace w^{(\textbf{u})}_m \right\rbrace_{m=1}^{M^{(\textbf{u})}_h}$, which contain the piecewise linear basis functions of $X^{-1}_{h,\Gamma}$ {(here we could also choose piecewise constant basis functions, as they belong to $L^2(\Gamma)$, too)} and the piecewise linear {continuous} basis functions of $X^{0}_{h,\Gamma_\Sigma}$, respectively. The approximate components of the unknowns, in the spaces $X^{-1}_{h,\Gamma}\otimes V^{-1}_{\Delta t}$ and $X^0_{h,\Gamma_\Sigma}\otimes V^0_{\Delta t}$, will be of the following form:
\begin{equation}\label{approx_of_psi_and_u}
\psi_{i,h,\Delta t}(t,\textbf{x})=\sum_{\ell=0}^{N_{\Delta t}-1}\sum_{m=1}^{M^{(\pmb{\psi})}_h} \psi_{i,\ell,m}w^{(\pmb{\psi})}_m(\textbf{x})v_\ell(t),\:\:u_{i,h,\Delta t}(t,\textbf{x})=\sum_{\ell=0}^{N_{\Delta t}-1}\sum_{m=1}^{M^{(\textbf{u})}_h}u_{i,\ell,m}w^{(\textbf{u})}_m(\textbf{x})r_\ell(t),\:\: i=1,2,
\end{equation}
where the time basis $v_k$ and $r_k$ are defined as
$$v_\ell(t):=H[t-t_\ell]-H[t-t_{\ell+1}], \quad r_\ell(t):=H[t-t_\ell]\frac{t-t_\ell}{\Delta t}-H[t-t_{\ell+1}]\frac{t-t_{\ell+1}}{\Delta t}.$$

These choices for the approximation of $\textbf{u}_{h,\Delta t},\pmb{\psi}_{h,\Delta t}$ lead to the algebraic reformulation of \eqref{energetic_formh} as a linear system $\mathbb{S}\textbf{X}=\widetilde{\textbf{F}}$, having the form 
\begin{equation}\label{system}
\left(\begin{array}{cccc}
\mathbb{S}^{(0)} & \textbf{0} & \cdots & \textbf{0}\\
\mathbb{S}^{(1)} & \mathbb{S}^{(0)} & \cdots & \textbf{0}\\
\vdots & \vdots & \ddots & \vdots\\
\mathbb{S}^{(N_{\Delta t}-1)} & \mathbb{S}^{(N_{\Delta t}-2)} & \cdots & \mathbb{S}^{(0)}
\end{array}\right)
\left(
\begin{array}{c}
\textbf{X}_{(0)}\\
\textbf{X}_{(1)}\\
\vdots\\
\textbf{X}_{(N_{\Delta t}-1)}
\end{array}
\right)=
\left(
\begin{array}{c}
\widetilde{\textbf{F}}_{(0)}\\
\widetilde{\textbf{F}}_{(1)}\\
\vdots\\
\widetilde{\textbf{F}}_{(N_{\Delta t}-1)}
\end{array}
\right),
\end{equation}
where, for all time indices $\ell=0,...,N_{\Delta t}-1$, the unknown vectors are structured as
$$\textbf{X}_{(\ell)}=(\pmb{\Psi}_{(\ell)},\textbf{U}_{(\ell)})=\left(\psi_{1,\ell,1},\ldots,\psi_{1,\ell,M_h^{(\pmb{\psi})}},\psi_{2,\ell,1},\ldots,\psi_{2,\ell,M_h^{(\pmb{\psi})}},u_{1,\ell,1},\ldots,u_{1,\ell,M_h^{(\textbf{u})}},u_{2,\ell,1},\ldots,u_{2,\ell,M_h^{(\textbf{u})}}\right)^\top\,$$
and the right-hand side as
$$
\widetilde{\textbf{F}}_{(\ell)}=\left(0,\cdots,0,\,0,\cdots,0,\,\widetilde{F}_{1,\ell,1},\ldots,\widetilde{F}_{1,\ell,M_h^{(\textbf{u})}},\widetilde{F}_{2,\ell,1},\ldots,\widetilde{F}_{2,\ell,M_h^{(\textbf{u})}}\right)^\top 
\in \mathbb{R}^{2(M_h^{(\pmb{\psi})}+M_h^{(\textbf{u})})}\,,$$
with $\widetilde{F}_{i,\ell,m}:=\left\langle\widetilde{f}_i,w^{(\textbf{u})}_m\dot{r}_\ell\right\rangle_{0,\Gamma_\Sigma,(0,T]}$.\\
The matrix $\mathbb{S}$ shows a lower triangular Toeplitz structure, so that the system \eqref{system} can be solved by backsubstitution. This leads to a  \textit{marching-on-in-time} time stepping scheme, in which only the inversion of the block $\mathbb{S}^{(0)}$ is required at every time step.
The structure of blocks $\mathbb{S}^{(\ell)}$, of order $2(M_h^{(\pmb{\psi})}+M_h^{(\pmb{\textbf{u}})})\times 2(M_h^{(\pmb{\psi})}+M_h^{(\pmb{\textbf{u}})})$, is
$$\mathbb{S}^{(0)}=\left(\begin{array}{cc}
-\mathbb{V}^{(0)} & \mathbb{K}^{(0)}+\frac{1}{2}\:\mathbb{M}\\[4pt]
({\mathbb{K}}^{(0)}+\frac{1}{2}\:\mathbb{M})^\top & -\mathbb{W}^{(0)}
\end{array}\right),
\quad
\mathbb{S}^{(\ell)}=\left(\begin{array}{cc}
-\mathbb{V}^{(\ell)} & \mathbb{K}^{(\ell)}\\[4pt]
{\mathbb{K}^{(\ell)}}^\top & -\mathbb{W}^{(\ell)}
\end{array}\right),\:\ell=1,...,N_{\Delta t}-1.
$$
With obvious meaning of notation, the blocks $\mathbb{V}^{(\ell)}$, $\mathbb{K}^{(\ell)}$ and $\mathbb{W}^{(\ell)}$ are obtained from the discretization of the integral operators involved in \eqref{energetic_formh} and $\mathbb{M}$ is a mass matrix. For details about the numerical evaluation the entries of these blocks the reader is referred to \cite{thesis_Giulia, AimiJCAM}.\\
 With similar arguments, using the non-symmetric formulation of the Poincar\'{e}-Steklov operator, we obtain the following system of weak boundary integral equations:\\

\textit{find} $\textbf{u}_{h,\Delta t}\in (X^0_{h,\Gamma_\Sigma}\otimes V^0_{\Delta t})^2$ \textit{and} $\pmb{\psi}_{h,\Delta t}\in (X^{-1}_{h,\Gamma}\otimes V^{-1}_{\Delta t})^2$ \textit{such that}
\begin{equation}\label{energetic_form_non_symmh}
\begin{array}{l}
-\left\langle \mathcal{V}\pmb{\psi}_{h,\Delta t},\dot{\pmb{\eta}}_{h,\Delta t}\right\rangle_{0,\Gamma,(0,T]} +\left\langle\left(\mathcal{K}+\frac{1}{2}\right)\textbf{u}_{h,\Delta t},\dot{\pmb{\eta}}_{h,\Delta t}\right\rangle_{0,\Gamma,(0,T]}=0, \quad \forall\,\pmb{\eta}_{h,\Delta t}\in(X^{-1}_{h,\Gamma}\otimes V^{-1}_{\Delta t})^2,\\[4pt]
\left\langle\pmb{\psi}_{h,\Delta t}, \dot{\textbf{v}}_{h,\Delta t}\right\rangle_{0,\Gamma_\Sigma,(0,T]}=\left\langle\widetilde{\textbf{f}},\dot{\textbf{v}}_{h,\Delta t}\right\rangle_{0,\Gamma_\Sigma,(0,T]}, \quad \quad \quad \quad \quad \quad \quad \quad \forall\,\textbf{v}_{h,\Delta t}\in (X^0_{h,\Gamma_\Sigma}\otimes V^0_{\Delta t})^2.
\end{array}
\end{equation}

The discrete unknowns $\textbf{u}_{h,\Delta},\pmb{\psi}_{h,\Delta}$ are substituted again with \eqref{approx_of_psi_and_u}, leading to an algebraic form $\mathbb{S}\textbf{X}=\widetilde{\textbf{F}}$ of \eqref{energetic_form_non_symmh}, analogous to the system \eqref{system}, with the only difference that the blocks $\mathbb{S}^{(\ell)}$, again of order $2(M_h^{(\pmb{\psi})}+M_h^{(\pmb{\textbf{u}})})\times 2(M_h^{(\pmb{\psi})}+M_h^{(\pmb{\textbf{u}})}),$ are structured as
$$\mathbb{S}^{(0)}=\left(\begin{array}{cc}
-\mathbb{V}^{(0)} & \mathbb{K}^{(0)}+\frac{1}{2}\:\mathbb{M}\\[4pt]
\mathbb{M}^\top & \textbf{0}
\end{array}\right),
\quad
\mathbb{S}^{(\ell)}=\left(\begin{array}{cc}
-\mathbb{V}^{(\ell)} & \mathbb{K}^{(\ell)}\\[4pt]
\textbf{0} & \textbf{0}
\end{array}\right),\:\ell=1,...,N_{\Delta t}-1.
$$

\subsection{Algebraic formulation  of the Uzawa algorithm for unilateral frictional contact}
In the numerical solution of contact problems, the right-hand side $\widetilde{\textbf{f}}$ in \eqref{energetic_formh} and \eqref{energetic_form_non_symmh} specifies depending on the Neumann datum $\textbf{f}$ and the discretized Lagrange multiplier $\pmb{\lambda}_{H,\Delta T}(\textbf{x},t)$ as in \eqref{eq:WeakMixedVarEqh}.\\
Hence, let us consider the vector ${\bf F}=\left({\bf F}_{(\ell)}\right)_{\ell=0}^{N_{\Delta t}-1}$ with
$$
\textbf{F}_{(\ell)}=\left(0,\cdots,0,\,0,\cdots,0,\,F_{1,\ell,1},\ldots,F_{1,\ell,M_h^{(\textbf{u})}},F_{2,\ell,1},\ldots,F_{2,\ell,M_h^{(\textbf{u})}}\right)^\top
\in \mathbb{R}^{2(M_h^{(\pmb{\psi})}+M_h^{(\textbf{u})})}$$
and $F_{i,\ell,m}:=\left\langle f_i,w^{(\textbf{u})}_m\dot{r}_\ell\right\rangle_{0,\Gamma_\Sigma,(0,T]}$.\\ 
For the discretized Lagrange multiplier,
in practice we choose $H=h$ and $\Delta T = \Delta t$, therefore its Cartesian components can be expressed as
\begin{equation}\label{approx_of_lambda}
\lambda_{i,h,\Delta t}(t,\textbf{x})=\sum_{\ell=0}^{N_{\Delta t}-1}\sum_{m=1}^{M^{(\pmb{\lambda})}_h}\lambda_{i,\ell,m}w^{(\pmb{\lambda})}_m(\textbf{x})v_\ell(t),\; i=1,2,
\end{equation}
with $\left\lbrace w^{(\pmb{\lambda})}_m \right\rbrace_{m=1}^{M^{(\pmb{\lambda})}_h}$ the piecewise constant basis functions of $X^{-1}_{h,\Gamma_C}$. 
For each time step $\ell=0,...,N_{\Delta t}-1$, we introduce the vectors\\ 
$$\pmb{\Lambda}_{(\ell)}=\left(\lambda_{1,\ell,1},\ldots,\lambda_{1,\ell,M_h^{(\pmb{\lambda})}},\lambda_{2,\ell,1},\ldots,\lambda_{2,\ell,M_h^{(\pmb{\lambda})}}\right)^\top$$

which are collected in the vector $\pmb{\Lambda}=\left(\pmb{\Lambda}_{(\ell)}\right)_{\ell=0}^{N_{\Delta t}-1}\in \mathbb{R}^{N_{\Delta t}2M_h^{(\pmb{\lambda})}}$.
We denote by $\mathcal{J}_\perp$ the set of those indices $j$ of the vector $\pmb{\Lambda}$ which correspond to the components of $\pmb{\lambda}_{h,\Delta t}$ normal to $\Gamma$, and by $\mathcal{J}_\parallel$ the set of the remaining indices, 
corresponding to the components  of $\pmb{\lambda}_{h,\Delta t}$ tangential to $\Gamma$. Moreover, let $\mathcal{F}_j,\,j \in \mathcal{J}_\parallel$, denote the minimum of $\mathcal{F}$ on the element of the boundary mesh where the tangential component of $\pmb{\lambda}_{h,\Delta t}$ is being considered through the corresponding index $j \in \mathcal{J}_\parallel$ of $\pmb{\Lambda}$ vector.\\
The Uzawa algorithm for Tresca friction given in Algorithm \ref{alg1} then translates into the following algebraic procedure: 
\begin{algorithm}[H]
\caption{(Algebraic formulation of Uzawa algorithm for unilateral frictional contact)}
\label{alg2}
\begin{algorithmic}
\STATE Fix $\rho>0$ and $\epsilon>0$.
\STATE $k=0$,  $\pmb{\Lambda}^{(0)}= \textbf{0}$ and $\pmb{\Lambda}^{(-1)}=\textbf{1}$
\WHILE{$\Vert\pmb{\Lambda}^{(k)}-\pmb{\Lambda}^{(k-1)} \Vert_2/\Vert\pmb{\Lambda}^{(k)} \Vert_2> \epsilon$}
\STATE \textbf{solve} $\quad$ $\mathbb{S}\textbf{X}^{(k)}=\textbf{F}+\mathbb{M}^*\pmb{\Lambda}^{(k)}$
\STATE \textbf{extract}$\quad$
$\textbf{U}^{(k)}$ from $\textbf{X}^{(k)}$
\STATE \textbf{compute} $\quad$ $\pmb{\Lambda}^{(k+1)}= \text{pr}_C  (\pmb{\Lambda}^{(k)}  - \rho \widetilde{\mathbb{M}}(\textbf{U}^{(k)}-\textbf{G}))$ 
\STATE $k \leftarrow k+1$
\ENDWHILE
\end{algorithmic}
\end{algorithm}
where the stopping criterion is specified and $\text{pr}_C: \mathbb{R}^{N_{\Delta t}2M_h^{(\pmb{\lambda})}}\rightarrow \mathbb{R}^{N_{\Delta t}2M_h^{(\pmb{\lambda})}}$ is understood as the discretized version of the projector considered in Algorithm \ref{alg1}, acting on a 
vector ${\bf W}$ as 
\begin{equation}
(\text{pr}_C\ {\bf W})_j= \left\{
\begin{array}{lr}
\max \,\{W_j,0\}, & j \in \mathcal{J}_\perp\\
W_j,& |W_j|\leq\mathcal{F}_j \text{ and } j\in \mathcal{J}_\parallel\\
\mathcal{F}_j \frac{W_j}{|W_j|},& |W_j|>\mathcal{F}_j \text{ and } j\in \mathcal{J}_\parallel
\end{array}\,.
\right.
\end{equation}
Moreover, the matrix $\mathbb{M}^*\in \mathbb{R}^{N_{\Delta t}\,2(M_h^{(\pmb{\psi})}+M_h^{(\textbf{u})})\times N_{\Delta t}2M_h^{(\pmb{\lambda})}}$ represents the projection of discretized Lagrange multipliers on the discretized displacement space, trivially extended in such a way that the vector  $\mathbb{M}^*\pmb{\Lambda}^{(k)}$ matches the length of the vector ${\bf F}$, while
the vector $\textbf{G}$ contains the coefficients of the interpolant of $g$ in $X^0_{h,\Gamma_C}\otimes V^0_{\Delta t}$, suitably trivially extended to match the length of vector $\textbf{U}^{(k)}$. Finally, 
$\widetilde{\mathbb{M}} \in \mathbb{R}^ {N_{\Delta t}2M_h^{(\pmb{\lambda})}\times N_{\Delta t}2M_h^{(\textbf{u})}}$ is the mass matrix representing the interplay between the finite dimensional spaces of discretized Lagrange multipliers and discretized displacements, and taking also into account the time derivative with respect to the tangential components. 

For Coulomb friction, as discussed in Section \ref{sec:coulomb} the friction threshold $\mathcal{F}_j$ is replaced by $\mathcal{F}_{c,j} \Lambda_{j_\perp}$, where for $j \in \mathcal{J}_\parallel$ the index $j_\perp \in \mathcal{J}_\perp$ denotes the index of $\pmb{\Lambda}$ vector related to the normal component of $\pmb{\lambda}_{h, \Delta t}$  in the same element of the boundary mesh.

\subsection{Extension to bilateral contact}

In order to keep the notation as simple and clear as possible and in view of the numerical test presented in Section \ref{sec;numres}, we consider a simplified geometry for the two-body contact, where $\Omega_1$ is completely embedded in the unbounded complement $\Omega_2 = \mathbb{R}^2\setminus \overline{\Omega_1}$, so that the only existing boundary is the interface $\Gamma=\Gamma_\Sigma=\Gamma_\Sigma'$ between the two materials, subdivided into the two subsets $\Gamma_C$ and $\Gamma_I$.\\
To numerically solve the bilateral frictional contact problem,
we recall that the Poincar\'{e}-Steklov operators ${\cal S}_j,\,j=1,2$, defined as in 
\eqref{operator_S} for each of the two domains $\Omega_j,\,j=1,2$, can be represented in symmetric or non-symmetric form, as described in Section \ref{birs}, using the boundary integral operators ${\cal V}_j, {\cal K}_j, {\cal K}_j^*, {\cal W}_j,\, j=1,2$. The quartet of boundary integral operators depends on the material parameters in $\Omega_j$ through the corresponding fundamental solution.\\ With reference to discrete equation \eqref{eq:WeakMixedVarEq2bh}, we need to pose \eqref{energetic_formh} and \eqref{energetic_form_non_symmh} in both $\Omega_j,\,j=1,2$ and properly combine them.
For the two-body contact model problem taken into account, the unknowns on the boundary are the auxiliary variables $\pmb{\psi}_{j,h,\Delta t}:=\mathcal{V}_j^{-1}\left(\mathcal{K}_j+\frac{1}{2}\right)\textbf{u}_{j,h,\Delta t}$, $j=1,2$, both belonging to the discrete space $(X^{-1}_{h,\Gamma}\otimes V^{-1}_{\Delta t})^2$, the displacement field $\textbf{u}_{h,\Delta t}:=\textbf{u}_{1,h,\Delta t}$ belonging to the discrete space $(X^0_{h,\Gamma_\Sigma}\otimes V^0_{\Delta t})^2$ and finally the displacement gap $\widetilde{\textbf{u}}_{h,\Delta t}:=\textbf{u}_{2,h,\Delta t}-\textbf{u}_{1,h,\Delta t}$ belonging to the discrete space $(X^0_{h,\Gamma_C}\otimes V^0_{\Delta t})^2$.
The first two spaces were introduced in Section \ref{subPS}, while the last one consists of vector-valued functions whose spatial components are spanned by the set $\left\lbrace w^{(\widetilde{\textbf{u}})}_m \right\rbrace_{m=1}^{M^{(\widetilde{\textbf{u}})}_h}$, the restrictions  from $\Gamma_\Sigma$ to $\Gamma_C$ of piecewise linear, continuous basis functions .\\ The resulting linear system, to be solved in each step of the Uzawa iteration, involves a matrix with the same structure as in \eqref{system}.\\
Following the symmetric approach for the Poincar\'{e}-Steklov operator, the structure of the blocks $\mathbb{S}^{(\ell)}$, each of size $2(2M_h^{(\pmb{\psi})}+M_h^{(\textbf{u})}+M_h^{(\widetilde{\textbf{u}})})\times 2(2M_h^{(\pmb{\psi})}+M_h^{(\textbf{u})}+M_h^{(\widetilde{\textbf{u}})})$, is 
$$\mathbb{S}^{(0)}=\left(\begin{array}{cccc}
-\mathbb{V}_1^{(0)} & \mathbb{K}_1^{(0)}+\frac{1}{2}\:\mathbb{M} & {\bf 0} & {\bf 0}\\[4pt]
({\mathbb{K}}_1^{(0)}+\frac{1}{2}\:\mathbb{M})^\top & -\mathbb{W}_1^{(0)}-\mathbb{W}_2^{(0)} & (-{\mathbb{K}}_2^{(0)}+\frac{1}{2}\:\mathbb{M})^\top &  \mathbb{W}_{2,|}^{(0)}\\[4pt]
{\bf 0} & -{\mathbb{K}}_2^{(0)}+\frac{1}{2}\:\mathbb{M}  & -\mathbb{V}_2^{(0)} & {\mathbb{K}}_{2,|}^{(0)}- \frac{1}{2}\:\mathbb{M}_|\\[4pt]
{\bf 0} & {\mathbb{W}_{2,|}^{(0)}}^\top & {({\mathbb{K}}_{2,|}^{(0)}- \frac{1}{2}\:\mathbb{M}_|)}^\top & -\mathbb{W}_{2,||}^{(0)}
\end{array}\right),
$$
$$
\mathbb{S}^{(\ell)}=\left(\begin{array}{cccc}
-\mathbb{V}_1^{(\ell)} & \mathbb{K}_1^{(\ell)} & {\bf 0} & {\bf 0}\\[4pt]
{{\mathbb{K}}_1^{(\ell)}}^\top & -\mathbb{W}_1^{(\ell)}-\mathbb{W}_2^{(\ell)} & -{{\mathbb{K}}_2^{(\ell)}}^\top & \mathbb{W}_{2,|}^{(\ell)}\\[4pt]
{\bf 0} & -{\mathbb{K}}_2^{(\ell)} & -\mathbb{V}_2^{(\ell)} & {\mathbb{K}}_{2,|}^{(\ell)}\\[4pt]
{\bf 0} & {\mathbb{W}_{2,|}^{(\ell)}}^\top & {{\mathbb{K}}_{2,|}^{(\ell)}}^\top & -\mathbb{W}_{2,||}^{(\ell)}
\end{array}\right),\:\ell=1,...,N_{\Delta t}-1.
$$
With obvious meaning of notation, the blocks $\mathbb{V}_j^{(\ell)}$, $\mathbb{K}_j^{(\ell)}$ and $\mathbb{W}_j^{(\ell)}$, for $j=1,2$, are obtained from the discretization of the boundary integral operators related to each of the two domains. Moreover, the symbol $_|$ denotes the restriction by columns of the involved boundary integral operator block or of the mass matrix from $\Gamma_\Sigma$ to $\Gamma_C$, while the symbol $_{||}$ denotes the same restriction both by rows and columns.\\
For the alternative, non-symmetric formulation for the Poincar\'{e}-Steklov operator the blocks of the matrix $\mathbb{S}$ are still of size $2(2M_h^{(\pmb{\psi})}+M_h^{(\textbf{u})}+M_h^{(\widetilde{\textbf{u}})})\times 2(2M_h^{(\pmb{\psi})}+M_h^{(\textbf{u})}+M_h^{(\widetilde{\textbf{u}})})$ and assume the following simplified form: 
$$\mathbb{S}^{(0)}=\left(\begin{array}{cccc}
-\mathbb{V}_1^{(0)} & \mathbb{K}_1^{(0)}+\frac{1}{2}\:\mathbb{M} & {\bf 0} & {\bf 0}\\[4pt]
\mathbb{M}^\top & {\bf 0} & \mathbb{M}^\top & {\bf 0}\\[4pt]
{\bf 0} & -{\mathbb{K}}_2^{(0)}+\frac{1}{2}\:\mathbb{M}  & -\mathbb{V}_2^{(0)} & {\mathbb{K}}_{2,|}^{(0)} -\frac{1}{2}\:\mathbb{M}_|\\[4pt]
{\bf 0} & {\bf 0} &- \mathbb{M}_|^\top & {\bf 0}
\end{array}\right),
$$
$$
\mathbb{S}^{(\ell)}=\left(\begin{array}{cccc}
-\mathbb{V}_1^{(\ell)} & \mathbb{K}_1^{(\ell)} & {\bf 0} & {\bf 0}\\[4pt]
{\bf 0} & {\bf 0} &{\bf 0} &{\bf 0}\\[4pt]
{\bf 0} & -{\mathbb{K}}_2^{(\ell)} & -\mathbb{V}_2^{(\ell)} & {\mathbb{K}}_{2,|}^{(\ell)}\\[4pt]
{\bf 0} & {\bf 0} &{\bf 0} &{\bf 0}\\[4pt]
\end{array}\right),\:\ell=1,...,N_{\Delta t}-1.
$$
Finally, the extension of the Uzawa Algorithm 2 to the bilateral contact problem is given as follows, with the canonical meaning of the notation:
\begin{algorithm}[H]
\caption{(Algebraic formulation of Uzawa algorithm for bilateral frictional contact)}
\label{alg3}
\begin{algorithmic}
\STATE Fix $\rho>0$ and $\epsilon>0$.
\STATE $k=0$,  $\pmb{\Lambda}^{(0)}= \textbf{0}$ and $\pmb{\Lambda}^{(-1)}=\textbf{1}$
\WHILE{$\Vert\pmb{\Lambda}^{(k)}-\pmb{\Lambda}^{(k-1)} \Vert_2/\Vert\pmb{\Lambda}^{(k)} \Vert_2> \epsilon$}
\STATE \textbf{solve} $\quad$ $\mathbb{S}\textbf{X}^{(k)}=\textbf{F}+\mathbb{M}^*\pmb{\Lambda}^{(k)}$
\STATE \textbf{extract}$\quad$
$\widetilde{\textbf{U}}^{(k)}$ from $\textbf{X}^{(k)}$
\STATE \textbf{extend}$ \quad$
$\widetilde{\textbf{U}}^{(k)}$ trivially from 
$\mathbb{R}^{N_{\Delta t}2M_h^{(\widetilde{\textbf{u}})}}$ to $ \mathbb{R}^{N_{\Delta t}2M_h^{(\textbf{u})}}$
\STATE \textbf{compute} $\quad$ $\pmb{\Lambda}^{(k+1)}= \text{pr}_C  (\pmb{\Lambda}^{(k)}  - \rho \widetilde{\mathbb{M}}(\widetilde{\textbf{U}}^{(k)}-{\textbf{G}}))$ 
\STATE $k \leftarrow k+1$
\ENDWHILE
\end{algorithmic}
\end{algorithm}

\section{Numerical results}\label{sec;numres}

In this section we present numerical results involving polygonal and curved structures related to both unilateral and two-body frictional contact. Beyond test cases for generic parameters, experiments with material parameters corresponding to concrete and steel are considered.

\subsection{Example 1: unilateral frictional contact}

This example compares the evolution of the  displacement with no friction, Tresca friction and Coulomb friction in a model problem. We consider the contact problem with $c_S=0.5$ and $c_P=1$ in the square $\Omega=[-0.5,0.5]^2$ in the time interval $[0,T]=[0,2]$. The bottom, left, top and right  sides are denoted, respectively, by $\Gamma_b, \Gamma_r, \Gamma_t, \Gamma_l$. We set $\Gamma_N=\Gamma_t\cup \Gamma_r$, and define the contact region as $\Gamma_C=\Gamma_b\cup \Gamma_l$. In $\Gamma_C$ we consider a trivial force $\textbf{f}=\pmb{0}$, while in $\Gamma_N$ the force is given by $\textbf{f}=(0,f_2)$, with
$$
f_2(t,\textbf{x})=\left\{ \begin{array}{c l}
-0.1H[t], & \textbf{x}\in\Gamma_t\\
0, & \textbf{x}\in\Gamma_r
\end{array}\right. .
$$

Figure \ref{ex_1_contact_force} 
shows how the force acts on the square, pushing it down from the top. \\
We analyze three test problems, corresponding the following friction thresholds  on $\Gamma_C$: in Test 1 no friction is imposed, i.e.~the friction threshold is given by $\mathcal{F}=0$; in Test 2 we prescribe Tresca friction with $\mathcal{F}=0.05$; in Test 3 we prescribe Coulomb friction with $\mathcal{F}_c=0.5$.\\
The mixed formulation of each contact problem is solved using the Uzawa algorithm with a stopping criterion $\epsilon=10^{-4}$ and update parameter $\rho=10^2$.\\
\begin{figure}[h!]
\centering
\includegraphics[scale=0.6]{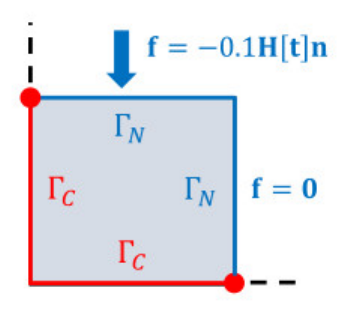}
\caption{Schematic representation of the Neumann datum.}
\label{ex_1_contact_force}
\end{figure}

Using the non-symmetric formulation \eqref{energetic_form_non_symmh} with $h=\Delta t=0.05$, we plot the numerical solution for the displacement in the midpoint of each side of the square in Figure \ref{Test_square}(a) for Test 1, in Figure \ref{Test_square}(b) for Test 2 and in Figure \ref{Test_square}(c) for Test 3.\\
\begin{figure}
\centering
\begin{tabular}{ccc}
\includegraphics[scale=0.33]{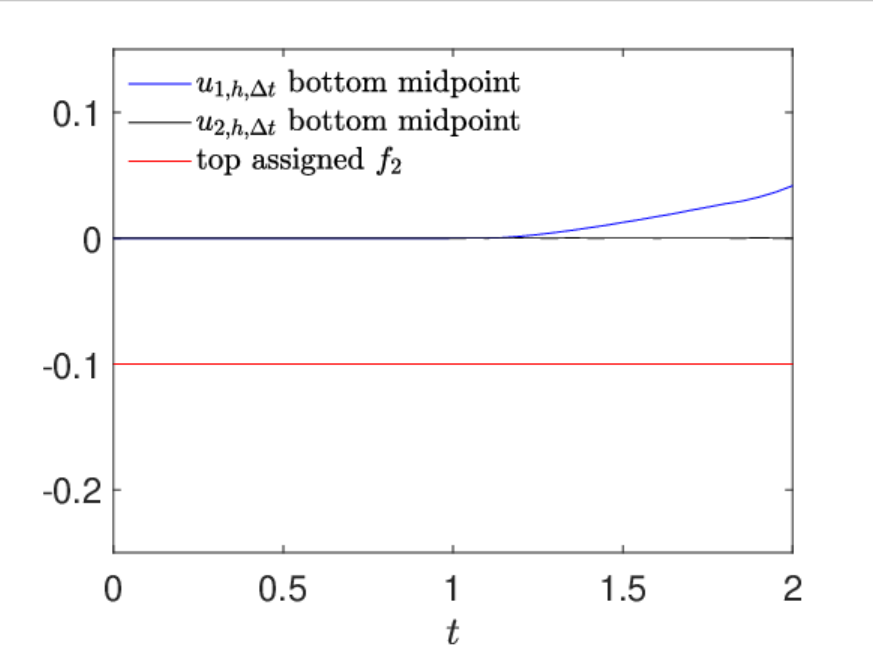} & \includegraphics[scale=0.33]{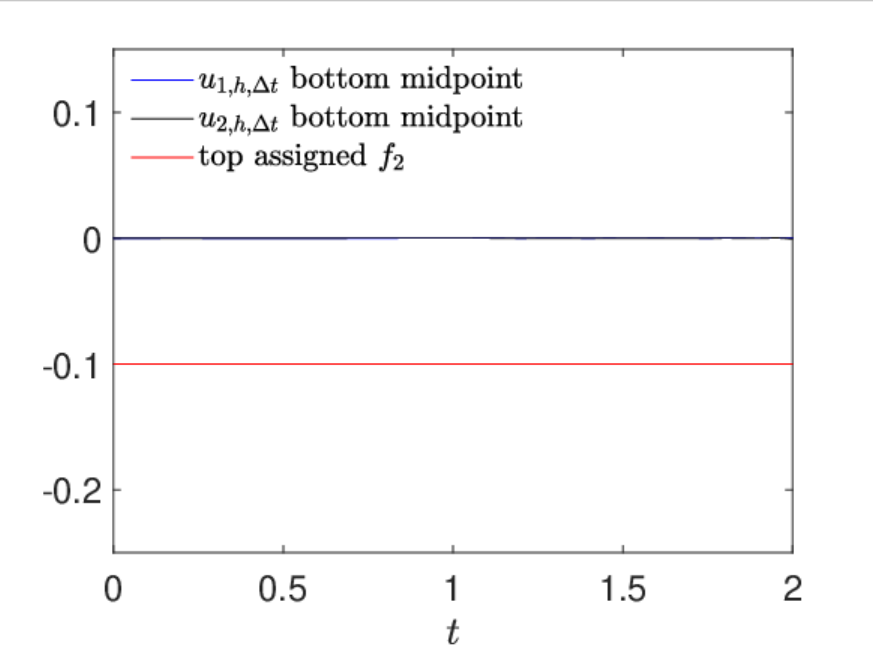} & \includegraphics[scale=0.33]{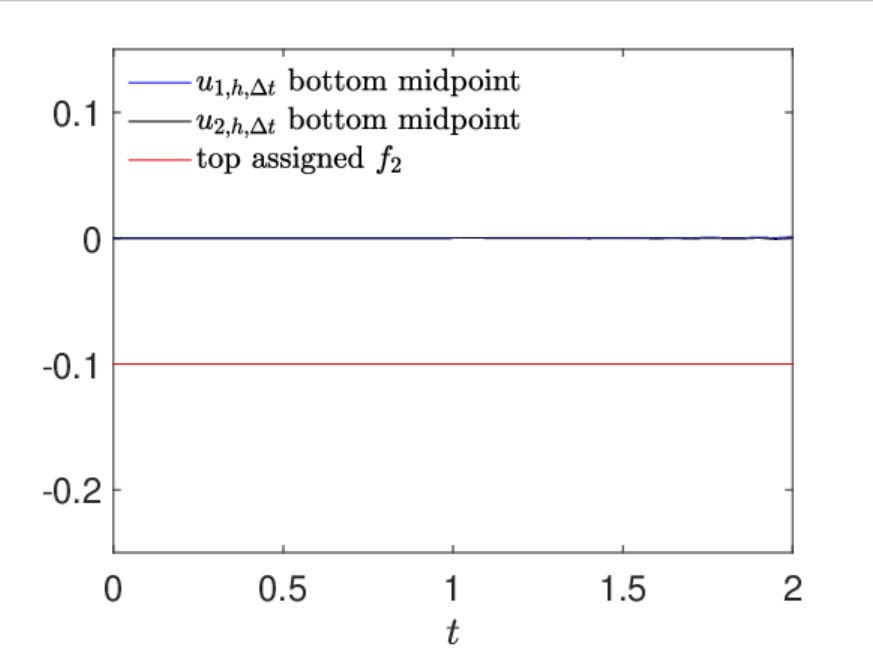} \\
\includegraphics[scale=0.33]{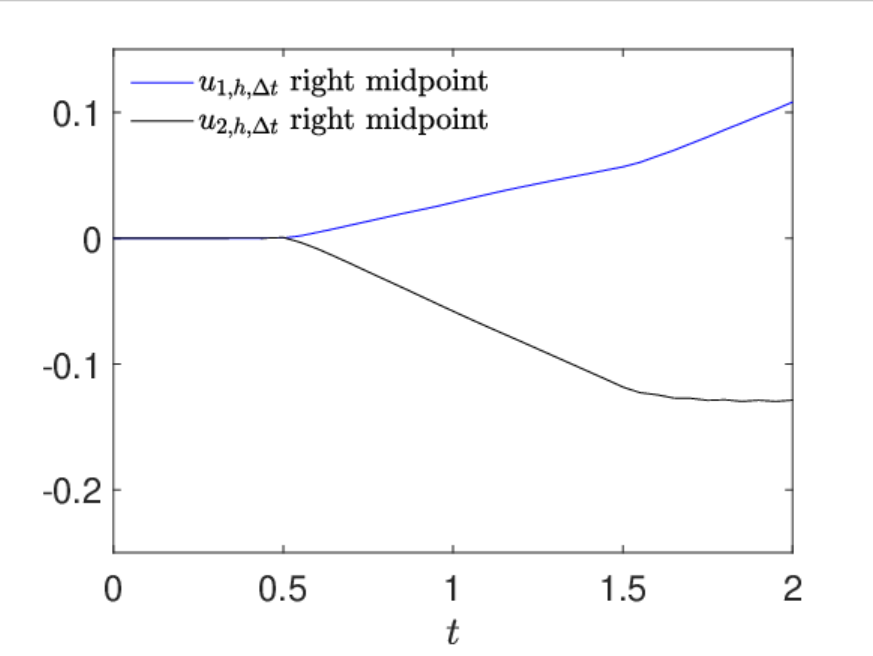} & \includegraphics[scale=0.33]{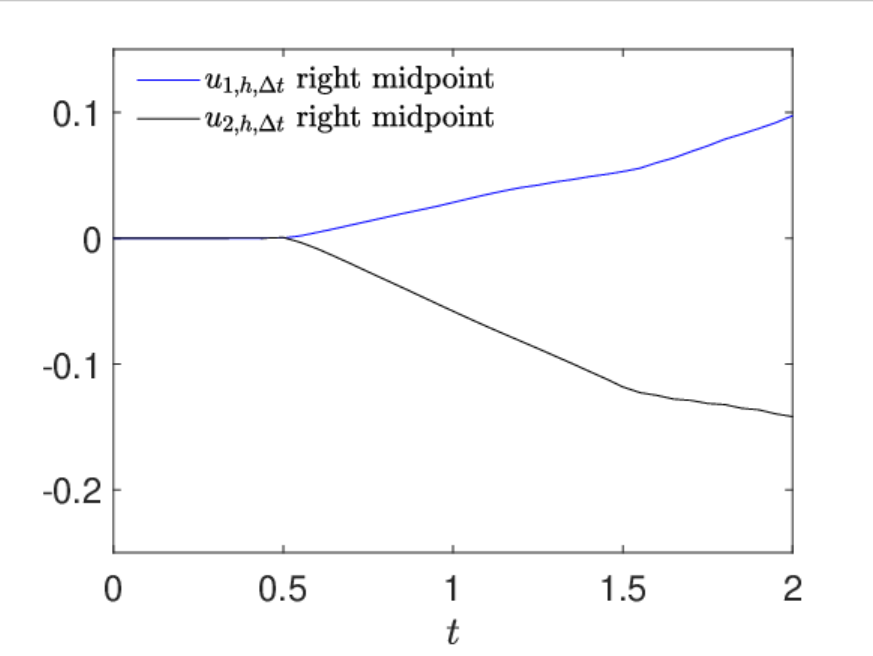} & \includegraphics[scale=0.33]{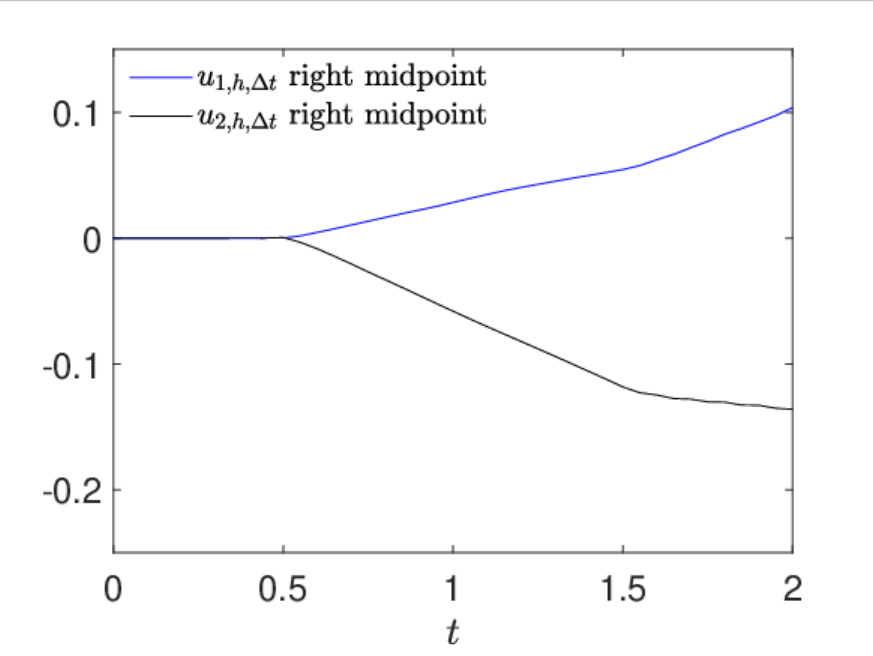} \\
\includegraphics[scale=0.33]{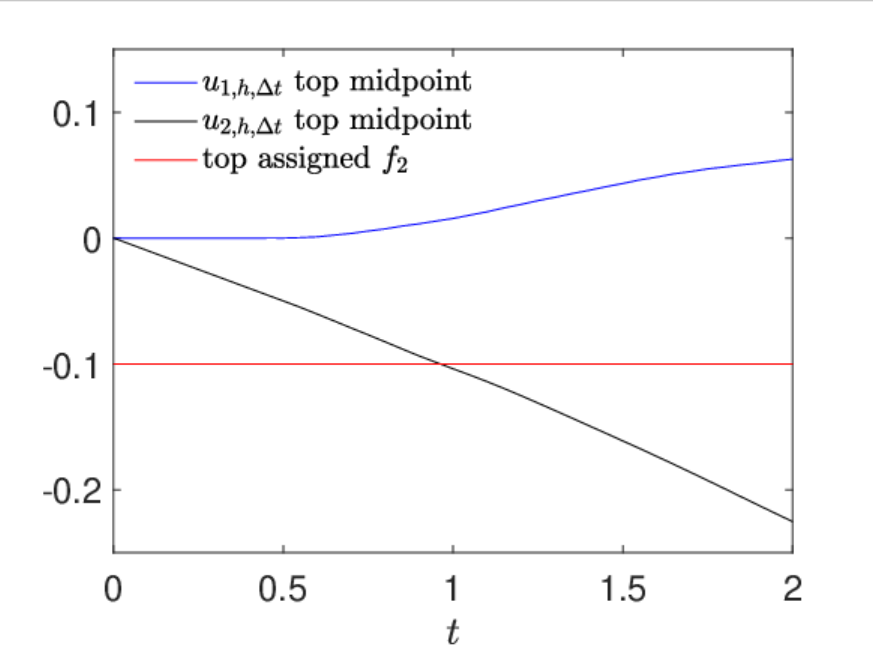} & \includegraphics[scale=0.33]{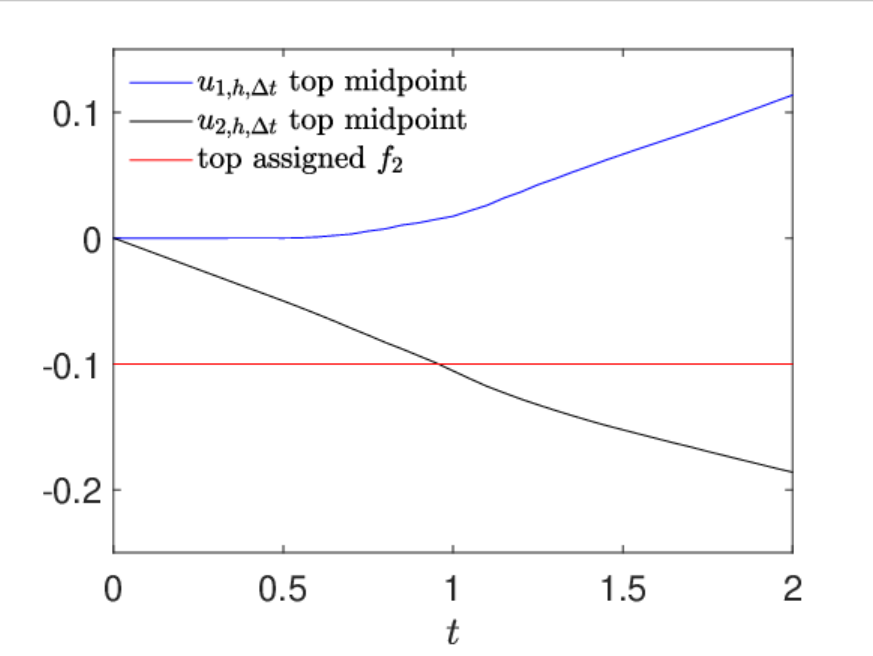} & \includegraphics[scale=0.33]{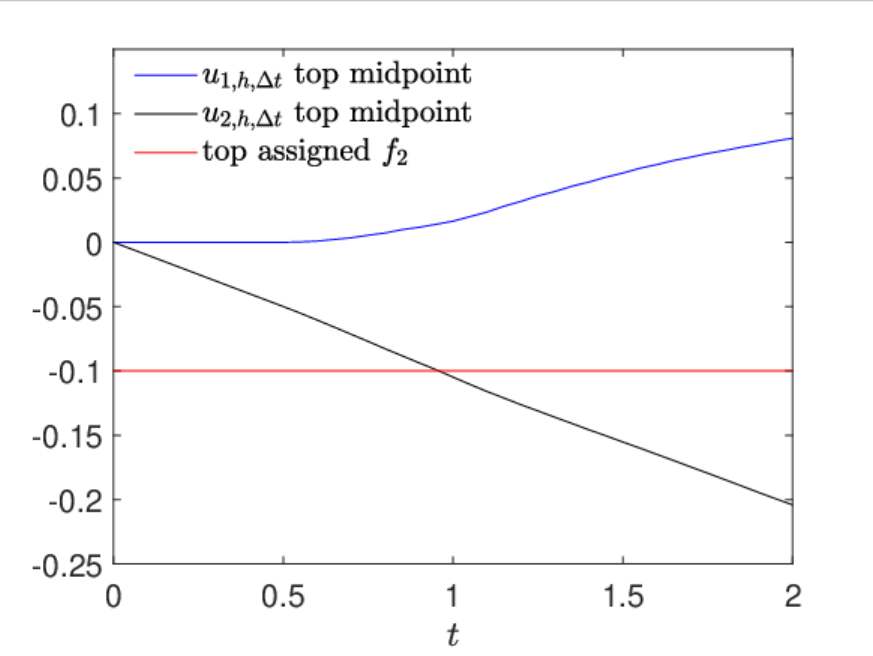} \\
\includegraphics[scale=0.33]{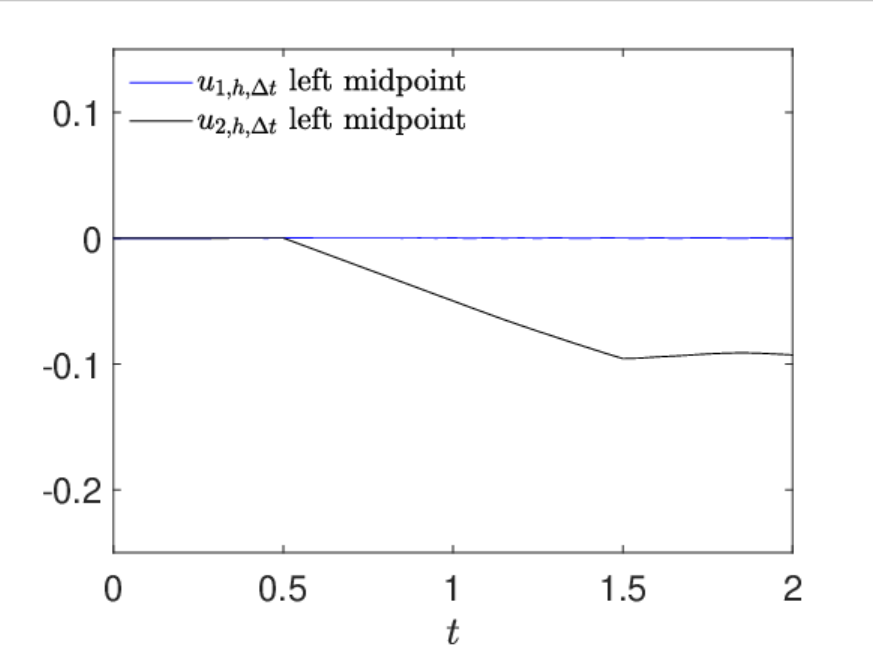} & \includegraphics[scale=0.33]{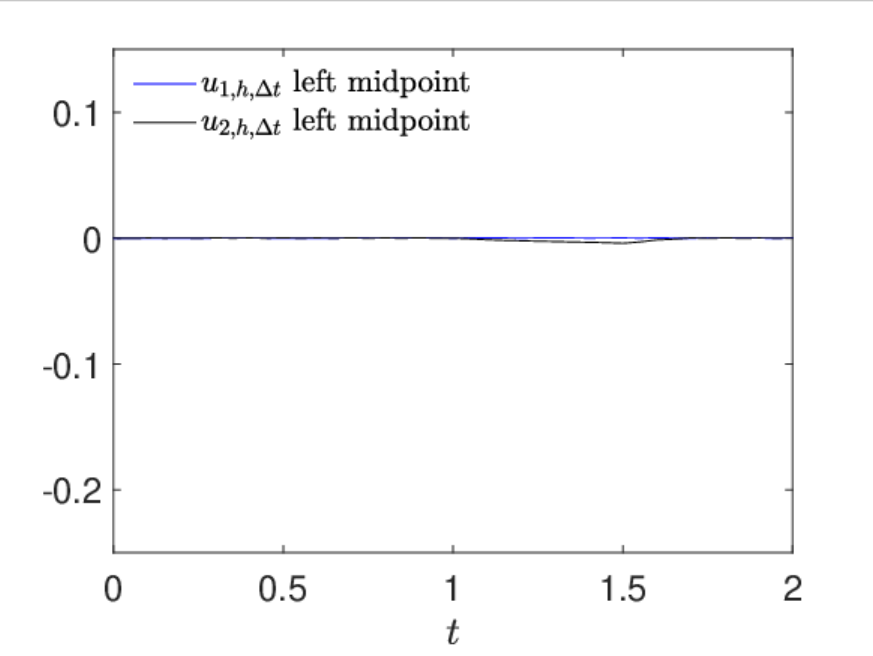} & \includegraphics[scale=0.33]{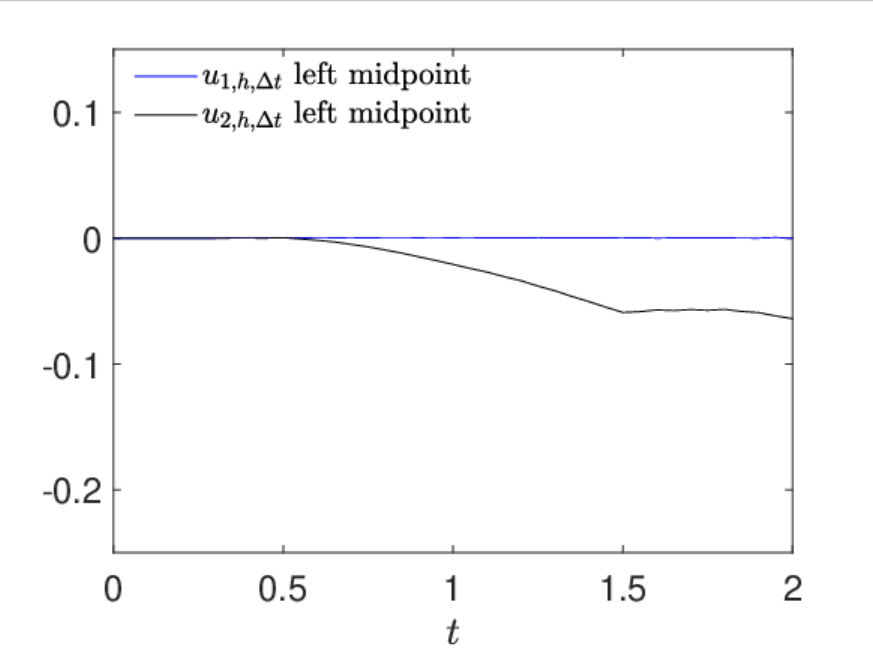} \\
(a)  & (b) & (c)
\end{tabular}
\caption{Time evolution of $\textbf{u}_{h,\Delta t}$ in the midpoint of each side of $\Gamma$ for Test 1 (a, no friction), Test 2 (b, Tresca friction) and  Test 3 (c, Coulomb friction), together with the corresponding vertical Neumann datum.}
\label{Test_square}
\end{figure}
Figures \ref{deformation_without_Coulomb_friction_square}(a), \ref{deformation_without_Coulomb_friction_square}(b) and \ref{deformation_without_Coulomb_friction_square}(c) show the global deformation of the square for Test 1, Test 2 and Test 3, respectively.
These figures show how the top side is pushed down, but tilts in the presence of friction, as the friction on the left wall hinders the downward movement on the left side.\\
For Test 3 with Coulomb friction, we show in Figure \ref{Example1_surface_lambda} the space-time surface over $\Gamma_C\times [0,T]$ of the vertical component of $\pmb{\lambda}$. It clearly shows that the vertical component of $\pmb{\lambda}$ is nonzero already for short times in the top-left corner, and the frictional contact propagates down the left side of the square. At $t=1$ the wave reaches the bottom side of the square and the jump of the vertical component of $\pmb{\lambda}$ corresponds to the nonpenetration condition.
\begin{figure}[h!]
\centering
\subfloat[]{\includegraphics[scale=0.55]{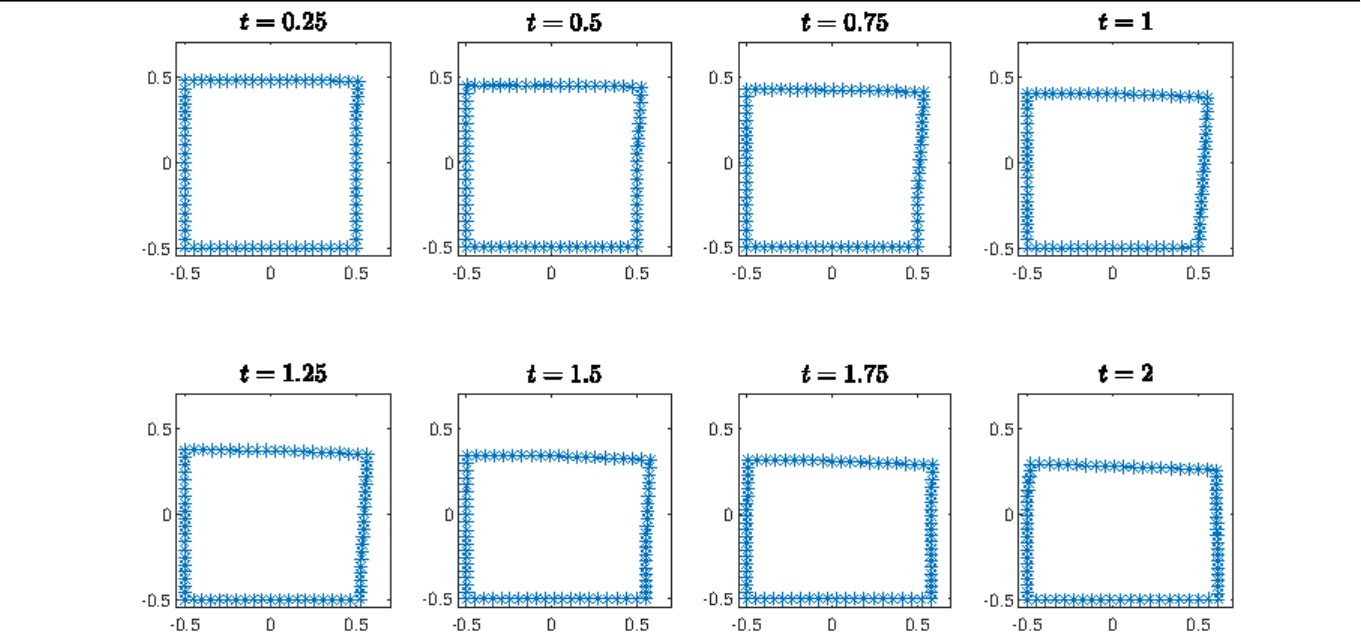}}
\vskip 0.1cm
\subfloat[]{\includegraphics[scale=0.55]{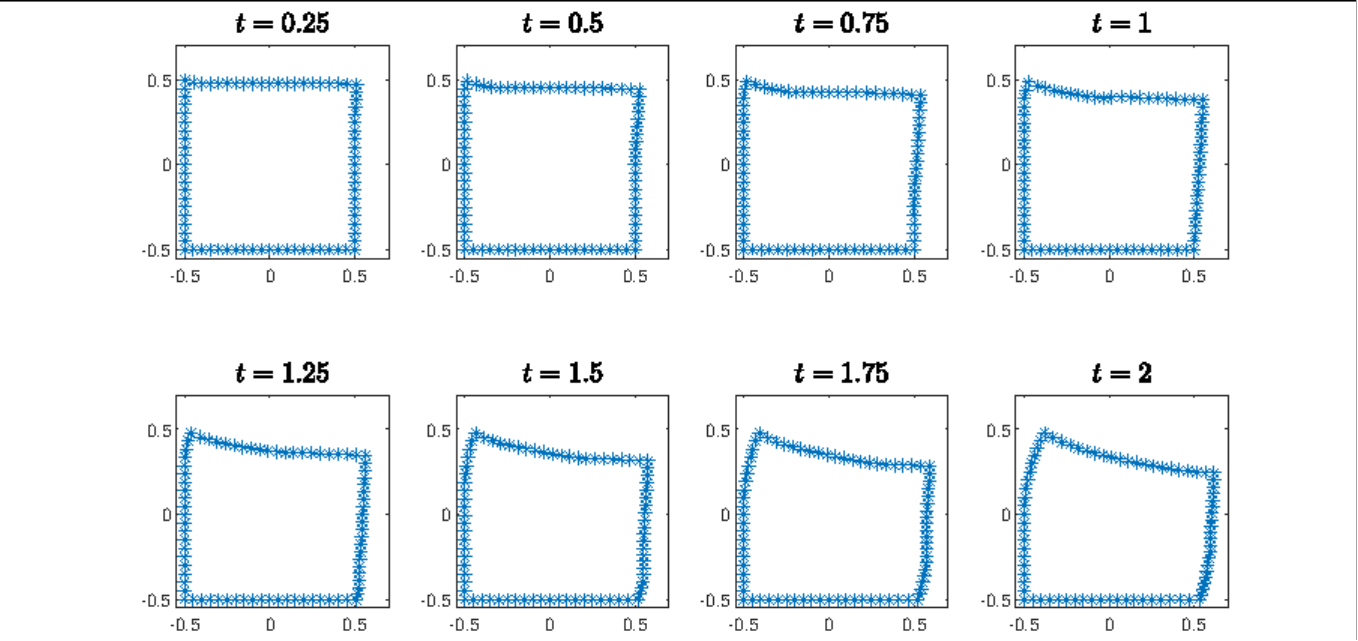}}
\vskip 0.1cm
\subfloat[]{\includegraphics[scale=0.55]{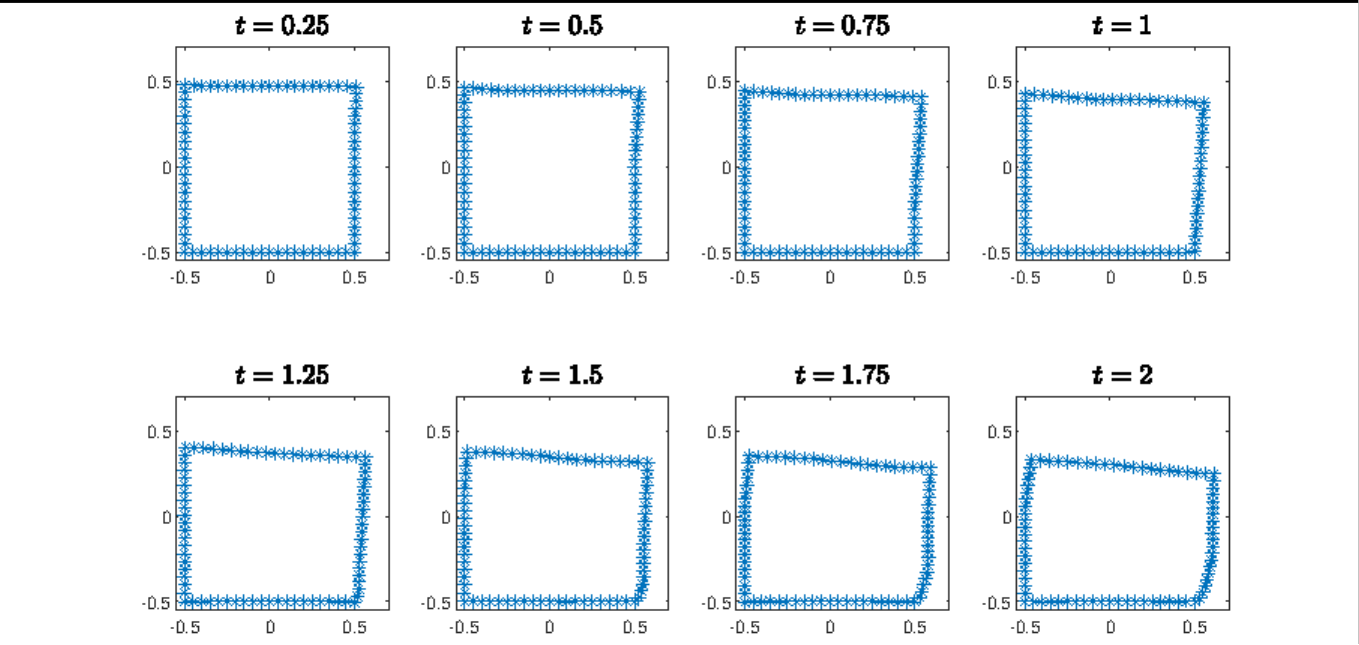}}
\caption{Global deformation of the square at several time instants for Test 1 (a, no friction), Test 2 (b, Tresca friction) and  Test 3 (c, Coulomb friction)}
\label{deformation_without_Coulomb_friction_square}
\end{figure}
\begin{figure}
\centering
\includegraphics[scale=0.58]{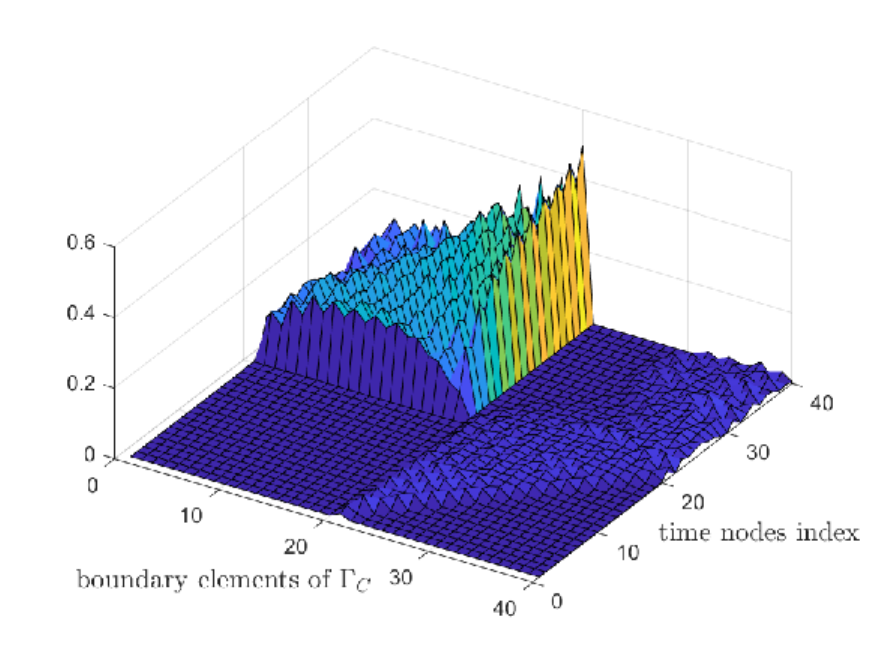}
\caption{Vertical component of $\pmb{\lambda}$ on $(0,T]\times \Gamma_C$,  for Test 3 with Coulomb friction (elements from 1 to 20 are in $\Gamma_b$, from 21 to 40 in $\Gamma_l$). }
\label{Example1_surface_lambda}
\end{figure}
{Finally, Figure \ref{example1_energy_in_time} shows the time evolution of the energy for the three types of contact. The reader can observe a linear increase for short times, before Tresca or Coulomb frictional contact dissipates some of the introduced energy, compared to the case without friction.} {Analogous results, not shown here for the sake of brevity, have been obtained using the symmetric formulation \eqref{energetic_formh}.}

\begin{figure}
\centering
{\includegraphics[scale=0.58]{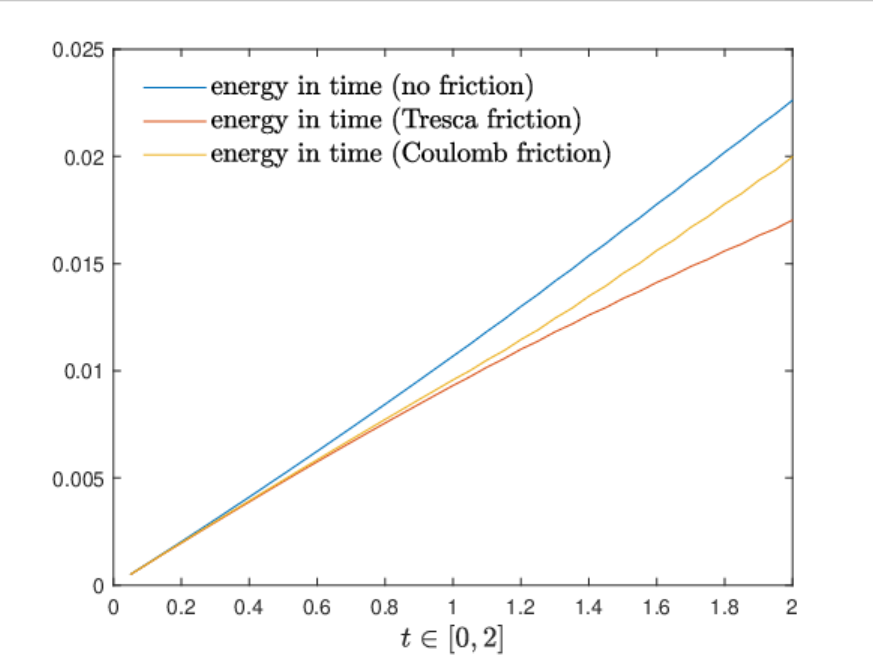}}
\caption{Energy as a function of time for the three types of contact (data obtained for $h=\Delta t=0.05$).}
\label{example1_energy_in_time}
\end{figure}

\subsection{Example 2: unilateral Coulomb frictional contact, {concrete-steel friction}}

The second example considers the same geometry and boundary conditions as in Example 1. The parameters are chosen to correspond to a square of concrete of side length $1\,m$
{in dynamic contact with a rigid wall. The frictional contact is described by the Coulomb friction law, with friction parameter  corresponding to the interaction between concrete and steel.}
For the small deformations and stresses considered here, concrete is well-described by a linearly elastic material with primary and secondary wave speeds given by $c_P=3.253\,\frac{m}{ms}$, $c_S=1.992\,\frac{m}{ms}$ (see \cite{concretenorm}). We impose a slowly increasing, vertical traction on the top side of the square, given by $f_2({\bf x},t)=-4\tanh((\frac{t}{15})^2)\, \frac{kg}{m\, (ms)^2}$, ${\bf x} \in \Gamma_t$. The final time is $T=0.6\,ms$, so that the primary wave can reach the bottom side and travel back to the top side of the square. Following \cite{wrig}, Table 4.1, the Coulomb friction coefficient is given by ${\cal F}_c=0.3$.\\
The mixed formulation of the contact problem is solved using the Uzawa algorithm with a stopping criterion $\epsilon=10^{-6}$ and update parameter $\rho=10^5$.\\
Figure \ref{concrete_steel} displays the square of the relative energy error in terms of the mesh size $h$, showing faster than linear convergence {for both symmetric and non-symmetric formulations.}\\
Figure \ref{Example2_surface_displecement} depicts the horizontal and vertical components of the approximate displacement field as a function of space and time, for elements on the boundary and for times $[0,0.6]$. Due to the physical properties of concrete and the mild traction applied on $\Gamma_t$, the obtained displacement is orders of magnitude smaller than the size of the concrete body. Focusing on the vertical component, we observe that the maximum deformation occurs, in time, at the top-right corner (21st boundary element), free from the effects of the friction which influences the bottom and left sides.\\ Figure \ref{Example2_energy_in_time} shows the evolution of the energy with time in a log-log plot. The linear increase in this plot hence corresponds to a polynomial increase of the energy as a function of time, mirroring the temporal behavior of the imposed Neumann datum.

\begin{figure}
\centering
\includegraphics[scale=0.55]{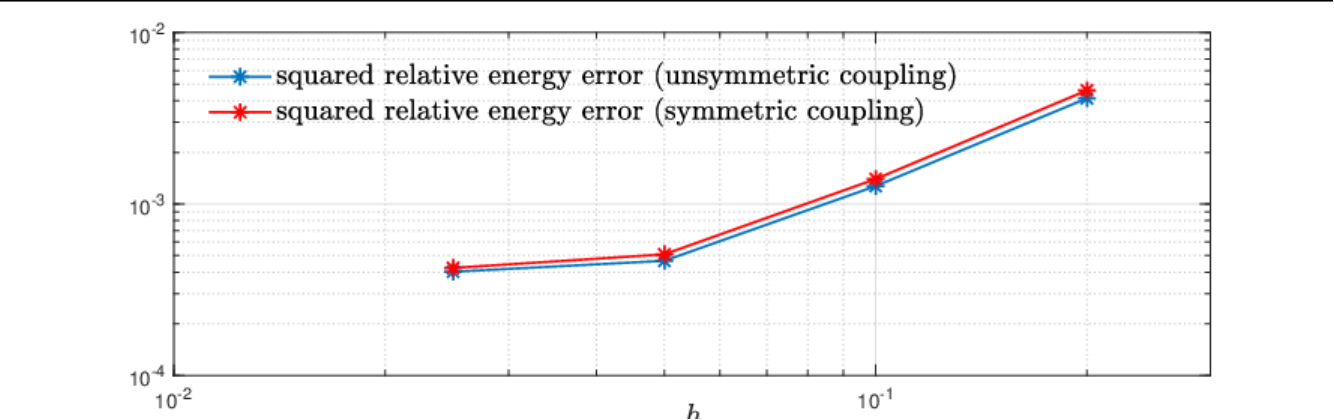}
\caption{Squared relative energy error for Example 2 (concrete-steel), depending on the space-time discretization parameters.}
\label{concrete_steel}
\end{figure}

\begin{figure}
\includegraphics[scale=0.53]{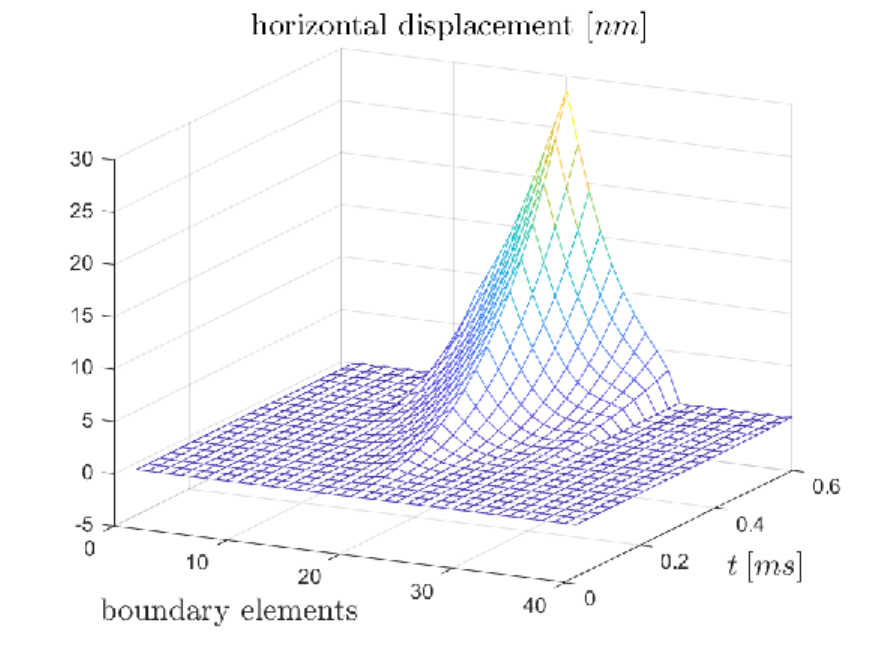}
\:
\includegraphics[scale=0.53]{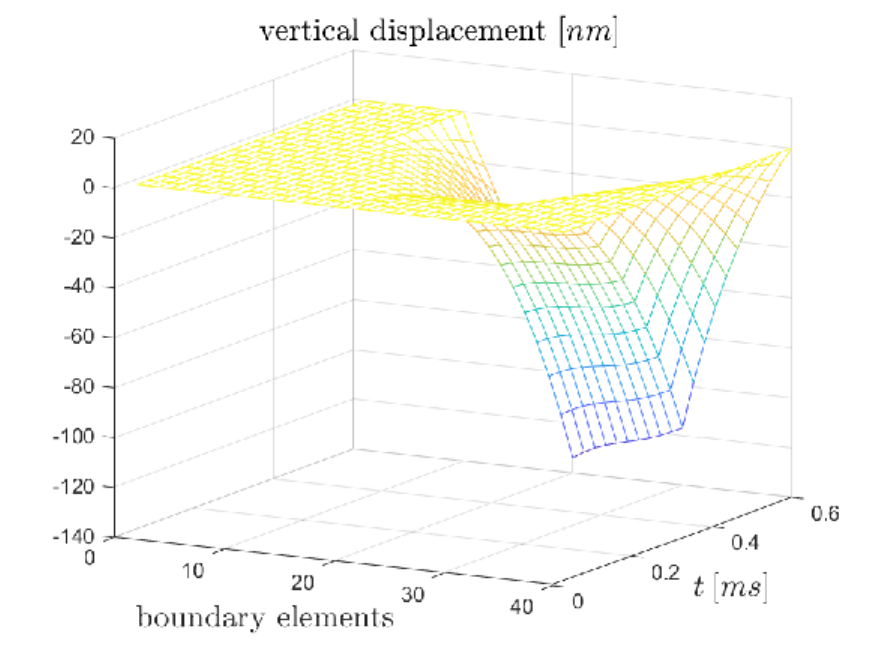}
\caption{Horizontal and vertical displacement expressed in nanometers on $\Gamma \times (0,T]$. Data obtained for $\Delta t=0.03$ $ms$ and $h=0.1$ $m$ (space elements from 1 to 10 are in $\Gamma_b$, from 11 to 20 in $\Gamma_r$, from 21 to 30 in $\Gamma_t$, from 31 to 40 in $\Gamma_l$).}
\label{Example2_surface_displecement}
\end{figure}

\begin{figure}
\centering
\includegraphics[scale=0.53]{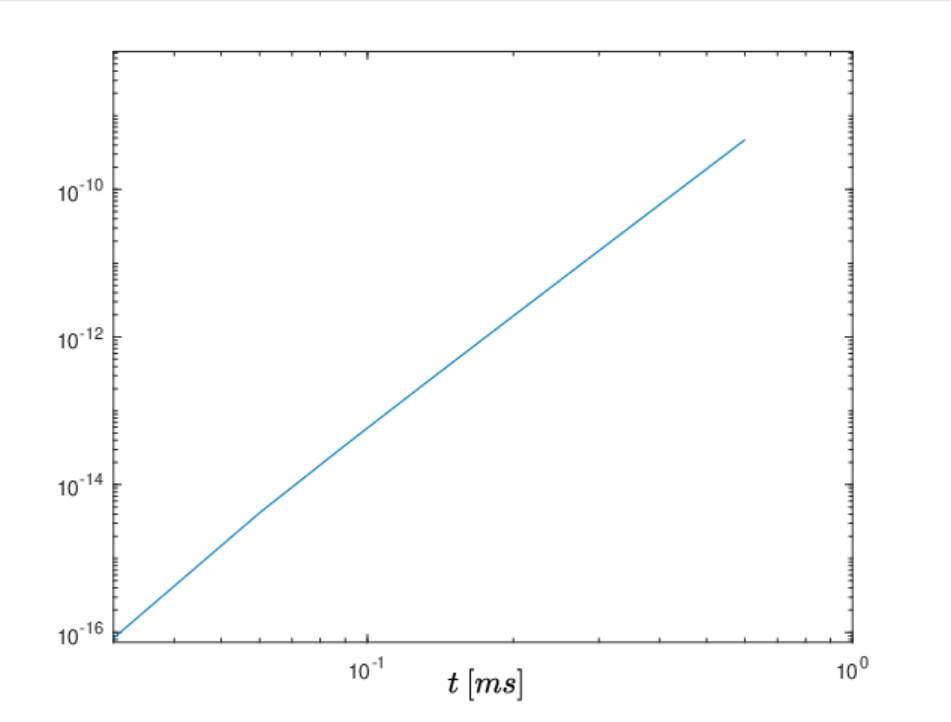}
\caption{Energy as a function of time (data obtained for $\Delta t=0.03$ $ms$ and $h=0.1$ $m$).}
\label{Example2_energy_in_time}
\end{figure}

\subsection{Example 3: concrete-concrete Coulomb frictional contact}
In this example we consider a concrete square of side length $1\,m$ embedded in surrounding concrete, under dynamic contact. As in Example 2, the material parameters of concrete correspond to a linearly elastic material with primary and secondary wave speeds $c_P=3.253\,\frac{m}{ms}$, $c_S=1.992\,\frac{m}{ms}$ (see \cite{concretenorm}). The interface $\Gamma_\Sigma'$ between the two domains is subdivided into $\Gamma_I=\Gamma_t \cup \Gamma_r$ and $\Gamma_C=\Gamma_l \cup \Gamma_b$, using the notation of the first example. A vertical traction is prescribed on the top side $\Gamma_t$, given by $f_2(t,{\bf x})=-4\tanh((\frac{t}{15})^2)\, \frac{kg}{m\, (ms)^2}$, ${\bf x} \in \Gamma_t$ as above. The final time is $T=0.6\,ms$ in such a way that the primary wave can reach the bottom side and travel back to the top side of the square. Following \cite{wrig}, Table 4.1, the Coulomb friction coefficient is given by ${\cal F}_c=0.75$, corresponding to the higher friction between two concrete bodies, compared to concrete and steel in Example 2.\\
The mixed formulation of the contact problem is solved using the Uzawa algorithm with a stopping criterion $\epsilon=10^{-6}$ and update parameter $\rho=10^5$.\\
Figure \ref{concrete_concrete} displays the square of the relative energy error in terms of the mesh size $h$, showing convergence of order approximately $\mathcal{O}(h^{\log(3)/\log(2)})\simeq\mathcal{O}(h^{1.58})$ {for both symmetric and non-symmetric formulations}.\\
Figure \ref{Example3_surface_gap_displecement} 
depicts the horizontal and vertical components of the approximate displacement gap $\widetilde{\mathbf{u}}$ as a function of space and time, for elements on the boundary and for times $[0,0.6]$. In this figure the gap has been trivially extended to $\Gamma_I$. The horizontal component is three orders of magnitude smaller than the vertical component, unlike in Example 2, and therefore the detachment of the top-right corner is now much smaller and almost negligible. This is due to the increased Coulomb friction coefficient, which is more than doubled for concrete-concrete contact, compared to the one for concrete-steel. In Figure \ref{Example3_deformation_zoom} we depict the approximate solution, magnified by a factor of $10^6$, showing on the left the deformation of the two-body concrete structure  and on the right a zoomed-in view of the top-left corner. There the detachment of the inner square from the surrounding body is observed. 

\begin{figure}
\centering
\includegraphics[scale=0.55]{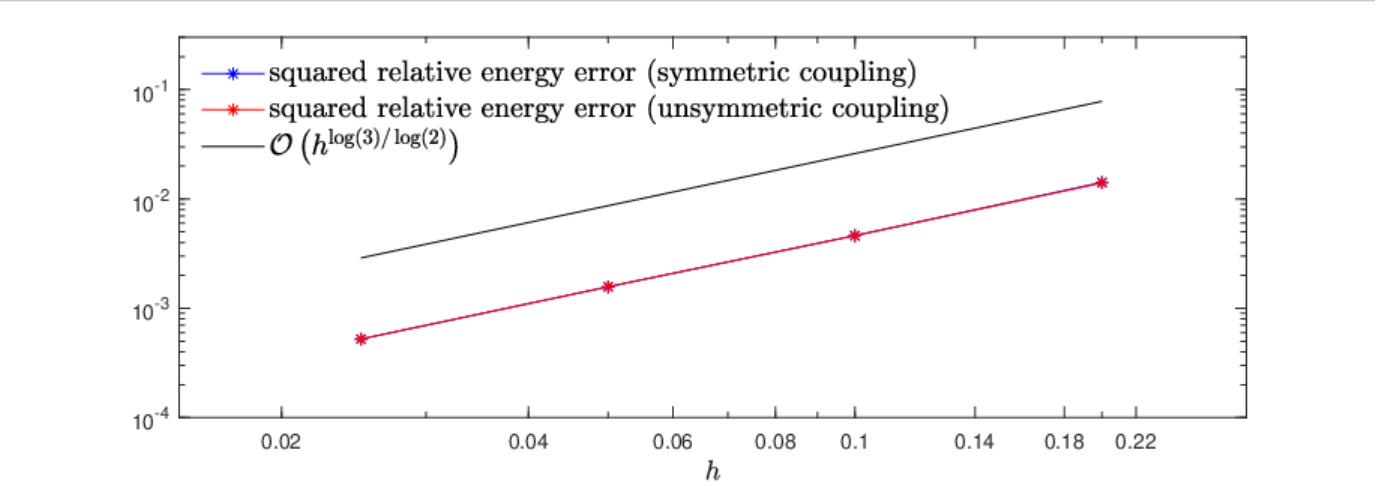}
\caption{Squared relative energy error for Example 3 (concrete-concrete), depending on the space-time discretization parameters.}
\label{concrete_concrete}
\end{figure}
\begin{figure}
\includegraphics[scale=0.53]{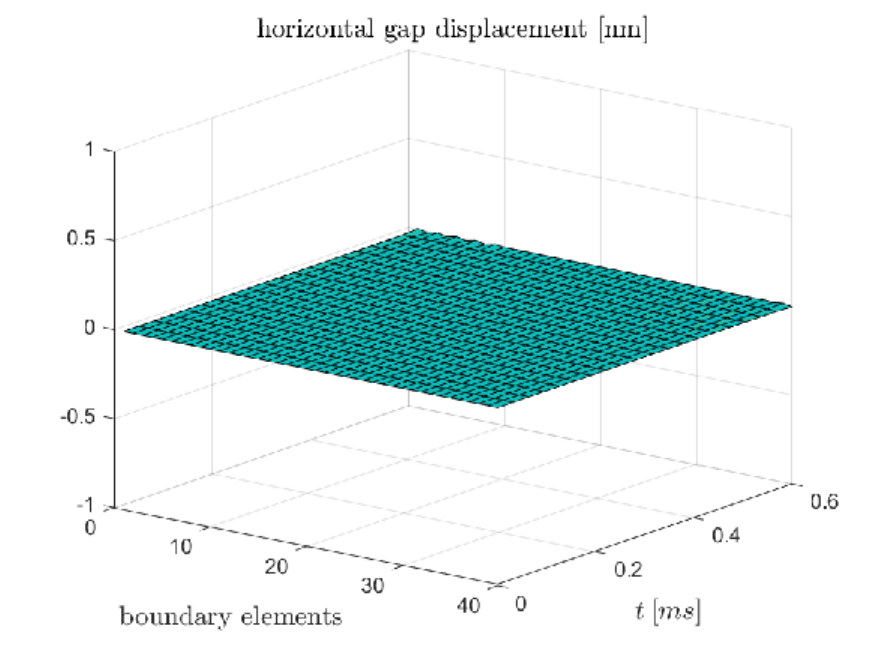}
\:
\includegraphics[scale=0.53]{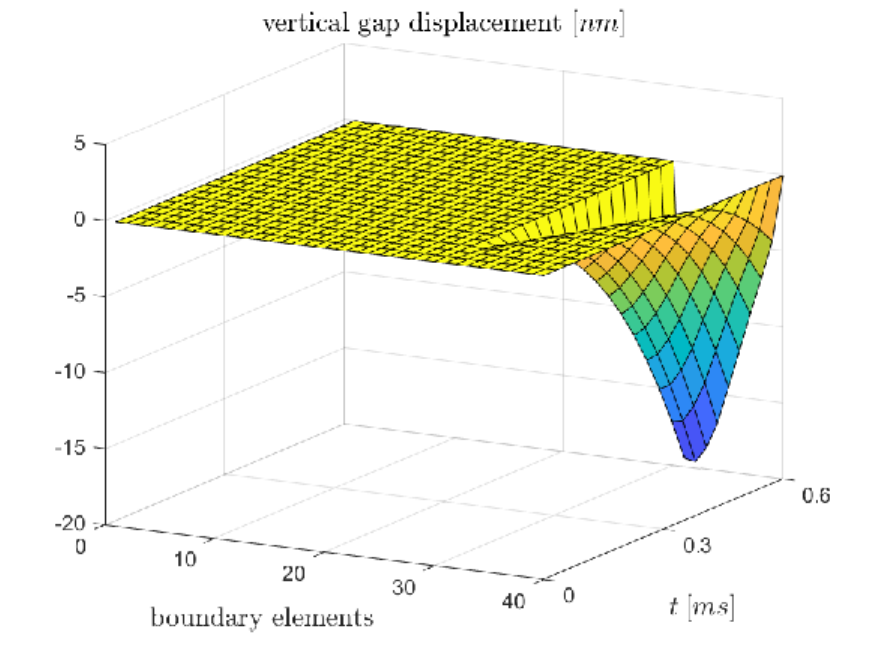}
\caption{Horizontal and vertical displacement gap expressed in nanometers on $\Gamma \times (0,T]$. Data obtained for $\Delta t=0.03$ $ms$ and $h=0.1$ $m$ (space elements from 1 to 10 are in $\Gamma_b$, from 11 to 20 in $\Gamma_r$, from 21 to 30 in $\Gamma_t$, from 31 to 40 in $\Gamma_l$).}
\label{Example3_surface_gap_displecement}
\end{figure}

\begin{figure}
\includegraphics[scale=0.53]{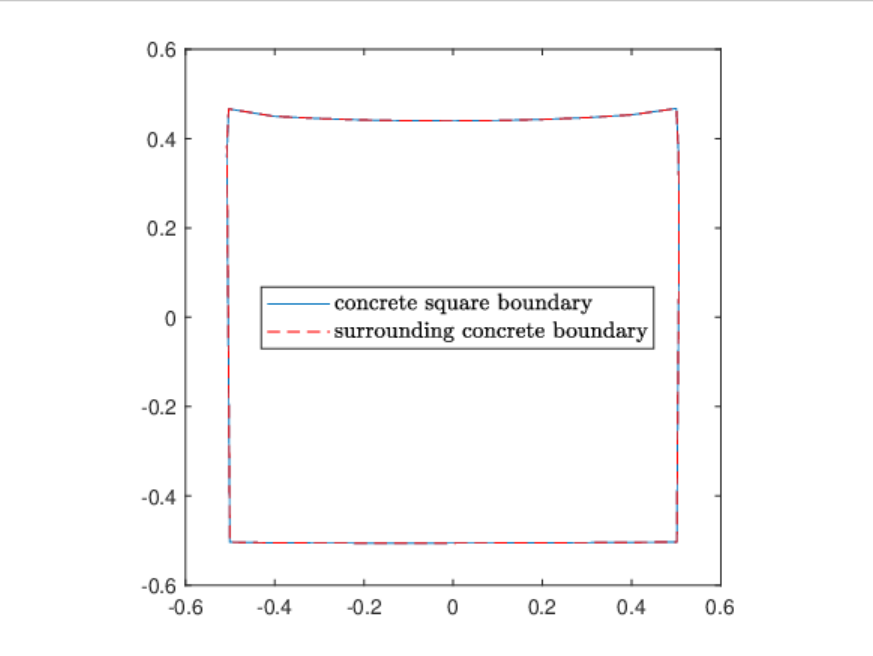}
\:
\includegraphics[scale=0.53]{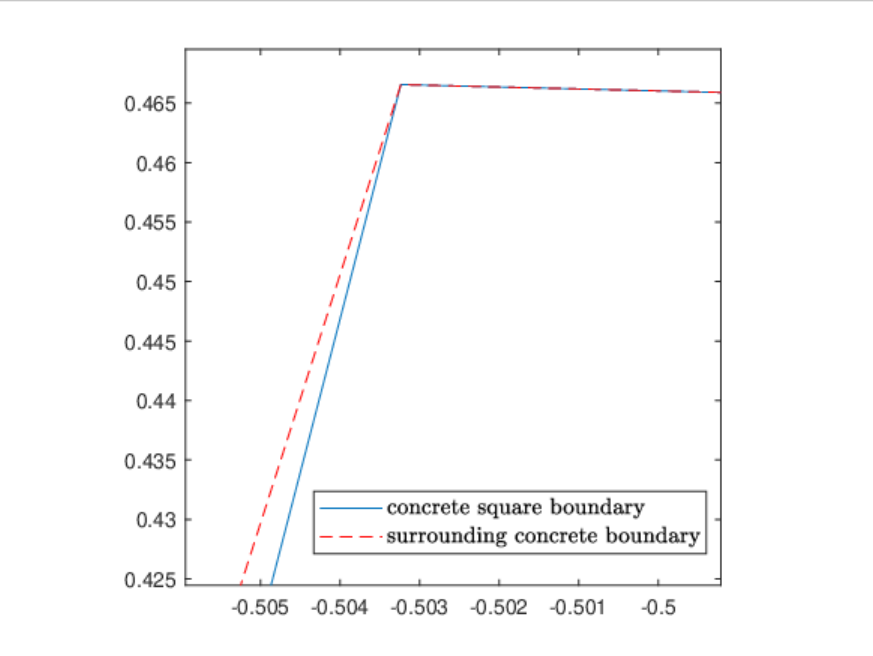}
\caption{Deformation of the two-body concrete structure on the left and zoomed-in view of the top-left corner on the right (approximate solution obtained for $\Delta t=0.03$ $ms$ and $h=0.1$ $m$, magnified by a factor of $10^6$).}
\label{Example3_deformation_zoom}
\end{figure}

\subsection{Example 4: dynamic contact problem in a circular geometry}

This final example considers a linearly elastic disk $\Omega=\left\lbrace (x,y)^\top \: :\: x^2+y^2\leq 0.2 \right\rbrace$ with $c_S=1$ and $c_P=2$, which is pushed upward by the impact of a moving surface. Contact can take place on all of the boundary of the disk, i.e.~$\Gamma=\Gamma_C$. No external forces are prescribed. The dynamics results from a time-dependent, flat obstacle at height
$$g(t)=4\left(\sqrt{1-1.5\:\left(\frac{t}{2}-0.2\right)^2}-1\right)\:H\left[1-\frac{t}{2}\right]-3.2\:H\left[\frac{t}{2}-1\right]-0.12.$$
We consider Coulomb friction contact between the disk and this obstacle, with $\mathcal{F}_c=2$. Figure \ref{deformation_circ_F_2} shows snapshots of the computed dynamics of the disk and the obstacle: for short times the elastic disk is squashed by the contact with the flat obstacle, with friction in the tangential direction; the elastic body then moves up and detaches. 

Figure \ref{en_0_and_2} compares the energy with ($\mathcal{F}_c=2$) and without ($\mathcal{F}_c=0$) Coulomb friction. After a rapid initial increase of the energy after the impact of the obstacle, the energy stabilizes and, after detachment, remains constant. The final energy is slightly larger with friction, than without, indicating a slightly higher transfer of energy to the disk from a frictional surface. This observation is in line with detailed studies of the influence of surface materials in the sports science literature \cite{pingpong}.

Figures \ref{deformation_circ_F_2} and \ref{en_0_and_2} were obtained 
using a discretization of $\Gamma = \partial \Omega$ by $80$ straight, uniform elements and $\Delta t\simeq 9.88\cdot 10^{-3}$. Using the symmetric formulation \eqref{energetic_formh} of the Poncar\'e-Steklov operator, the mixed formulation of the contact problem was solved using the Uzawa algorithm with a stopping criterion $\epsilon=10^{-6}$ and update parameter $\rho=10^5$. {Analogous numerical results, not shown here for the sake of brevity, were obtained using the non-symmetric formulation \eqref{energetic_form_non_symmh}.}

\begin{figure}
\centering
\includegraphics[scale=0.5]{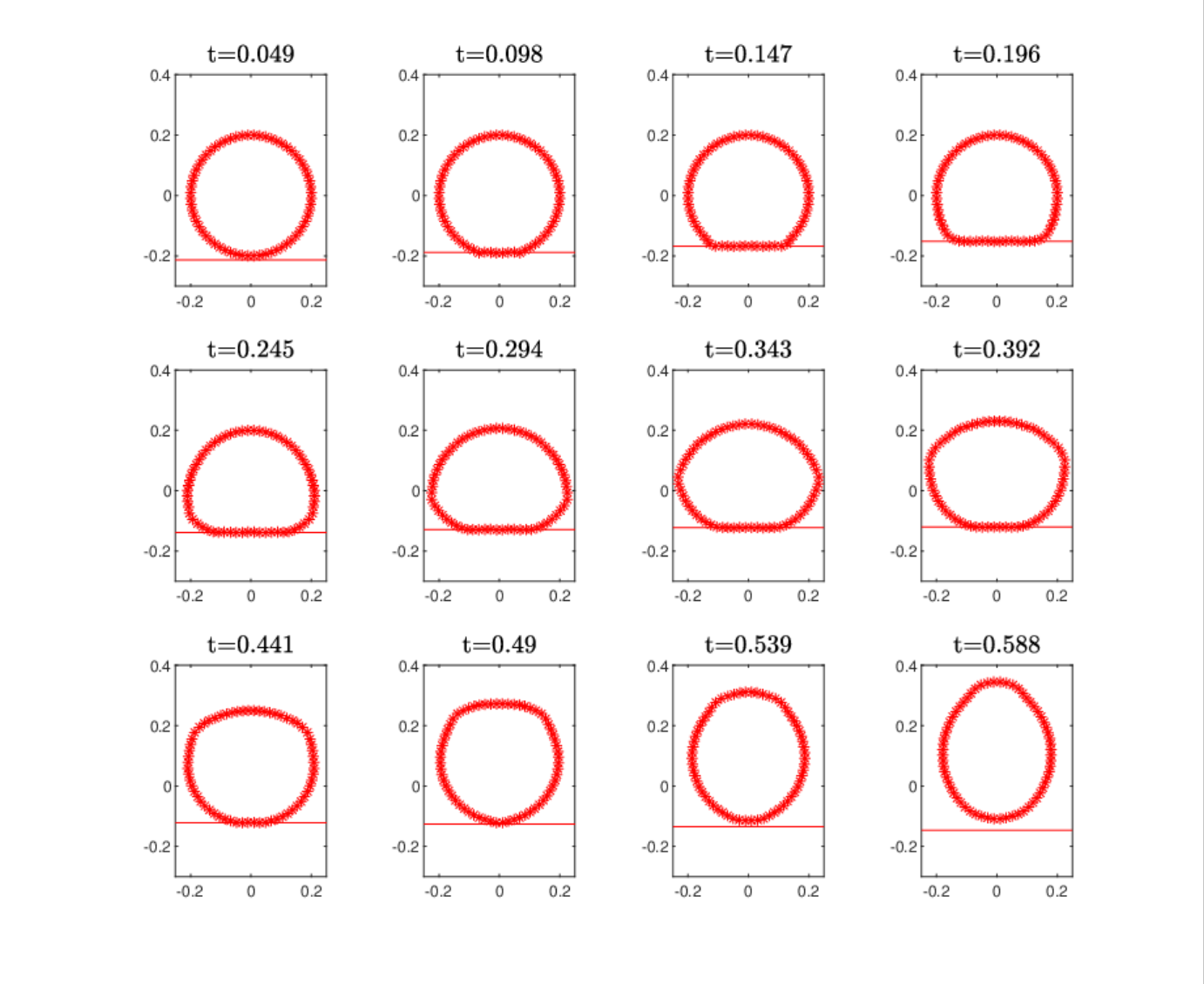}
\caption{Deformation of a disk under frictional contact with a moving obstacle.}
\label{deformation_circ_F_2}
\end{figure}

\begin{figure}
\centering
\includegraphics[scale=0.5]{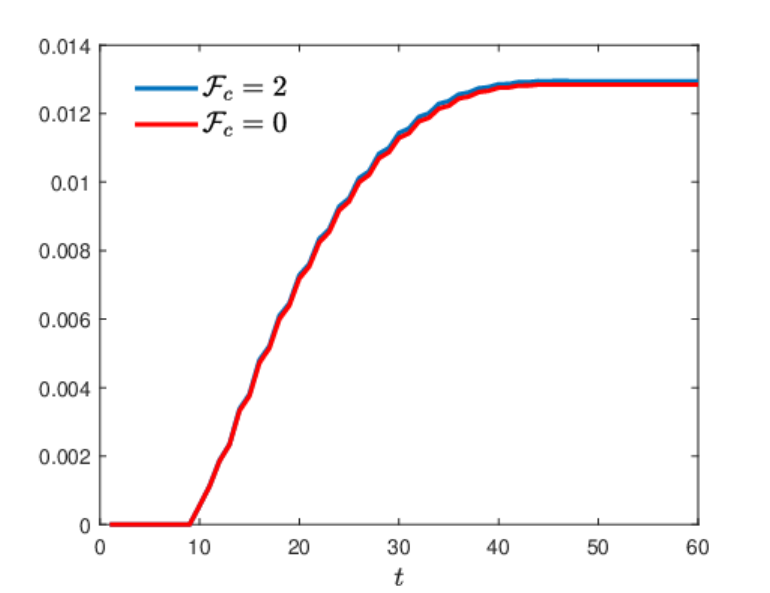}
\caption{Energy as a function of time, with and without Coulomb friction.}
\label{en_0_and_2}
\end{figure}

\section{Conclusions}\label{sec:conclusions}

This article investigates energetic Galerkin space-time boundary element methods to solve frictional contact problems in linear elastodynamics. Because the contact occurs at the
interface between two bodies, boundary elements provide a natural and efficient discretization approach. Here, the involved Poincar\'{e}-Steklov operator can be discretized in either symmetric or non-symmetric formulations.\\
We analyze the proposed method for unilateral contact involving Tresca friction, where we obtain an a priori  estimate for the error of the numerical solution in any space dimension. Numerical experiments in two space dimensions investigate the method beyond this idealized setting, for two-sided contact and realistic friction laws.\\
The implementation is discussed for such general problems, including the algebraic formulation of the boundary integral problem and its solution by a space-time Uzawa algorithm.\\
Detailed numerical results confirm the stability, energy conservation and convergence  of the  method. They study both the fundamental properties of the method for Tresca and Coulomb friction and apply it to problems for the dynamic unilateral and two-sided contact of concrete or steel in the linearly elastic regime. 
 
The practical implementation for three-dimensional problems and space-time adaptive mesh refinements \cite{banz,adaptive,stadaptive,cont1} remain the subject of future work.

\section*{Acknowledgements}

{This research was supported by the  Research in Residence program of the Centre International de Rencontres Math\'{e}matiques, Luminy, in 2023.}

\section*{Appendix}\label{appendix A}

In this appendix we review the definition and basic properties of space--time anisotropic Sobolev spaces, as relevant for our analysis in the main body of the article. We refer to  \cite{hd} for a detailed exposition for the scalar wave equation and to \cite{ourpaper, Becache1993, Becache1994} for elastodynamics. To include both open screens and closed boundaries for $\Gamma$, we denote by $\widetilde{\Gamma}$ be a closed, orientable Lipschitz manifold of dimension $d-1$ which contains $\Gamma$ as an open submanifold.    
We consider the standard Sobolev spaces of distributions with support in $\Gamma$:
$$\widetilde{H}^s(\Gamma) = \{u\in H^s(\widetilde{\Gamma}): \mathrm{supp}\ u \subset {\overline{\Gamma}}\}\ , \quad\ s \in \mathbb{R}\ .$$
The Sobolev space ${H}^s(\Gamma)$ of extensible distributions is then defined as the quotient space $ H^s(\widetilde{\Gamma}) / \widetilde{H}^s({\widetilde{\Gamma}\setminus\overline{\Gamma}})$.
To define a family of Sobolev norms in the frequency domain, we consider a partition of unity $\alpha_i$, $i=1, \dots,p$, subordinate to a covering of $\widetilde{\Gamma}$ by open sets $B_i \subset \mathbb{R}^{d-1}$, and for each $i$, a diffeomeorphism $\varphi_i$ from $B_i$ to the unit cube in $\mathbb{R}^{d-1}$. A family of Sobolev norms depending on the parameter $\omega \in \mathbb{C}\setminus \{0\}$ is defined as follows:
\begin{equation*}
 ||u||_{s,\omega,{\widetilde{\Gamma}}}=\left( \sum_{i=1}^p \int_{{\mathbb{R}^{d-1}}} (|\omega|^2+|\pmb{\xi}|^2)^s|\mathcal{F}\left\{(\alpha_i u)\circ \varphi_i^{-1}\right\}(\pmb{\xi})|^2 d\pmb{\xi} \right)^{\frac{1}{2}}\ ,
\end{equation*}
where $\mathcal{F}=\mathcal{F}_{\mathbf{x} \mapsto \pmb{\xi}}$ is the Fourier transform $\mathcal{F}\varphi(\pmb{\xi}) = \int e^{-i\mathbf{x}\cdot\pmb{\xi}} \varphi(\mathbf{x})\ d\mathbf{x}$. The norms on $H^s(\Gamma)$, $\|u\|_{s,\omega,\Gamma} = \inf_{v \in \widetilde{H}^s(\widetilde{\Gamma}\setminus\overline{\Gamma})} \ \|u+v\|_{s,\omega,\widetilde{\Gamma}}$ are equivalent for different $\omega$. We further define $\widetilde{H}^s(\Gamma)$, $\|u\|_{s,\omega,\Gamma, \ast } = \|e_+ u\|_{s,\omega,\widetilde{\Gamma}}$, using the extension $e_+$ of a distribution on $\Gamma$ by $0$ to a distribution on $\widetilde{\Gamma}$.  If a fixed value of $\omega$ is considered, we write $H^s_\omega(\Gamma)$ for $H^s(\Gamma)$ {endowed with the norm $\|\cdot\|_{s,\omega,\Gamma}$}, and $\widetilde{H}^s_\omega(\Gamma)$ for $\widetilde{H}^s(\Gamma)$ {endowed with the norm $\|\cdot\|_{s,\omega,\Gamma, \ast}$}. 
Note that $\|u\|_{s,\omega,\Gamma, \ast }\geq \|u\|_{s,\omega,\Gamma}$. 

We now define the space-time anisotropic Sobolev spaces which are used in this article:
\begin{definition}
For {$\sigma>0$ and} $r,s \in\mathbb{R}$ we set
\begin{align}
 H^r_\sigma(\mathbb{R}^+,{H}^s(\Gamma))&=\{ u \in \mathcal{D}^{'}_{+}(H^s(\Gamma)): e^{-\sigma t} u \in \mathcal{S}^{'}_{+}(H^s(\Gamma))  \textrm{ and }   ||u||_{r,s,\Gamma} < \infty \}\ , \nonumber \\
 H^r_\sigma(\mathbb{R}^+,\widetilde{H}^s({\Gamma}))&=\{ u \in \mathcal{D}^{'}_{+}(\widetilde{H}^s({\Gamma})): e^{-\sigma t} u \in \mathcal{S}^{'}_{+}(\widetilde{H}^s({\Gamma}))  \textrm{ and }   ||u||_{r,s,\Gamma, \ast} < \infty \}\ .\label{sobdef}
\end{align}
Here, $\mathcal{D}^{'}_{+}({H}^s({\Gamma}))$ consists of all distributions on $\mathbb{R}$ with support in $[0,\infty)$, which take values in the real-valued subspace of ${H}^s({\Gamma})$. We similarly define $\mathcal{D}^{'}_{+}(\widetilde{H}^s({\Gamma}))$, and denote the spaces of tempered distributions by $\mathcal{S}^{'}_{+}({H}^s({\Gamma}))$ and $\mathcal{S}^{'}_{+}(\widetilde{H}^s({\Gamma}))$.  The Sobolev spaces \eqref{sobdef} are Hilbert spaces with the norms
\begin{align}
\|u\|_{r,s,\sigma}:=\|u\|_{r,s,\Gamma,\sigma}&=\left(\int_{-\infty+i\sigma}^{+\infty+i\sigma}|\omega|^{2r}\ \|\hat{u}(\omega)\|^2_{s,\omega,\Gamma}\ d\omega \right)^{\frac{1}{2}}\ ,\nonumber \\
\|u\|_{r,s,\sigma,\ast} := \|u\|_{r,s,\Gamma,\sigma,\ast}&=\left(\int_{-\infty+i\sigma}^{+\infty+i\sigma}|\omega|^{2r}\ \|\hat{u}(\omega)\|^2_{s,\omega,\Gamma,\ast}\ d\omega \right)^{\frac{1}{2}}\,. \label{sobnormdef}  
\end{align}
\end{definition}
When $r=s=0$ we obtain the weighted $L^2$-space with scalar product $\langle u,v \rangle_{{\sigma,\Gamma,\mathbb{R}^+}}=\int_0^\infty e^{-2\sigma t} \int_\Gamma u \,v \, d\Gamma_{\mathbf{x}}\ dt$, which was introduced in \eqref{sigma_product}). 

We finally recall the relevant mapping properties of the boundary integral operators in these Sobolev spaces:
\begin{theorem}\label{mappingproperties}
Let 
$\sigma>0$. Then the following operators are continuous for $r\in \R$:
\begin{align*}
& \mathcal{V}:  {H}^{r+1}_\sigma(\R^+, \tilde{H}^{-\frac{1}{2}}(\Gamma))^d\to  {H}^{r}_\sigma(\R^+, {H}^{\frac{1}{2}}(\Gamma))^d \ ,
\\ & \mathcal{K}^\ast:  {H}^{r+1}_\sigma(\R^+, \tilde{H}^{-\frac{1}{2}}(\Gamma))^d\to {H}^{r}_\sigma(\R^+, {H}^{-\frac{1}{2}}(\Gamma))^d \ ,
\\ & \mathcal{K}:  {H}^{r+1}_\sigma(\R^+, \tilde{H}^{\frac{1}{2}}(\Gamma))^d\to {H}^{r}_\sigma(\R^+, {H}^{\frac{1}{2}}(\Gamma))^d \ ,
\\ & \mathcal{W}:  {H}^{r+1}_\sigma(\R^+, \tilde{H}^{\frac{1}{2}}(\Gamma))^d\to {H}^{r}_\sigma(\R^+, {H}^{-\frac{1}{2}}(\Gamma))^d \ .
\end{align*}\\
\end{theorem}

\end{document}